\documentclass[11pt, reqno]{amsart}
\usepackage{fullpage}

\newif\ifHideFoot
\HideFoottrue  
\HideFootfalse

\newif\ifHideApp
\HideApptrue  
\HideAppfalse  

\ifHideFoot \HideApptrue \fi

\usepackage{amssymb,amsmath,amsthm,amscd,mathrsfs,graphicx, color}
\usepackage[cmtip,all,matrix,arrow,tips,curve]{xy}

\usepackage{hyperref}

\usepackage[usenames,dvipsnames]{xcolor}
\usepackage{xypic, xy}
\hypersetup{colorlinks=true,citecolor=ForestGreen,linkcolor=Maroon,urlcolor=NavyBlue}
\usepackage[normalem]{ulem}
\usepackage{mathtools}
\usepackage{mathpazo}

\numberwithin{equation}{section}
\newtheorem{teo}{Theorem}[section]
\newtheorem{pro}[teo]{Proposition}
\newtheorem{lem}[teo]{Lemma}
\newtheorem{cor}[teo]{Corollary}

\newtheorem{teoalpha}{Theorem}
\newtheorem{proalpha}[teoalpha]{Proposition}
\newtheorem{coralpha}[teoalpha]{Corollary}
\newtheorem{quealpha}[teoalpha]{Question}
\newtheorem*{teoalpha*}{Theorem}

\newtheorem*{projannsenbis}{Proposition~\ref{P:Jannsen}(bis)}

\theoremstyle{definition}
\newtheorem{dfn}[teo]{Definition}
\newtheorem{dfnalpha}[teoalpha]{Definition}
\newtheorem{exa}[teo]{Example}

\theoremstyle{remark}
\newtheorem{rem}[teo]{Remark}

\newtheorem{nota}[teo]{Notation}

\ifHideFoot

\newcommand{\Yano}[1]{}
\newcommand{\Jeff}[1]{}
\newcommand{\Charles}[1]{}

\else

\newcommand{\marg}[1]{\normalsize{{
   \color{red}\footnote{{\color{blue}#1}}}{\marginpar[\vskip
   -.25cm{\color{red}\hfill$\implies$\tiny\thefootnote}]{\vskip
    -.2cm{\color{red}$\impliedby$\tiny\thefootnote}}}}}
\newcommand{\Yano}[1]{\marg{(Yano) #1}}
\newcommand{\Jeff}[1]{\marg{(Jeff) #1}}
\newcommand{\Charles}[1]{\marg{(Charles) #1}}

\fi

\DeclareMathOperator{\coniveau}{N}

\def\mmu{{\pmb\mu}}

\def\cx{{\mathbb C}}
\def\ff{{\mathbb F}}

\def\rat{{\mathbb Q}}
\def\integ{{\mathbb Z}}
\def\kk{{\mathbb K}}
\def\ww{{\mathbb W}}
\def\iso{\simeq}

\renewcommand{\bar}[1]{{\overline{#1}}}

\DeclareMathOperator{\Hom}{Hom}

\DeclareMathOperator{\gal}{Gal}

\DeclareMathOperator{\spec}{Spec}
\DeclareMathOperator{\pic}{Pic}

\DeclareMathOperator{\Ab}{Ab}
\DeclareMathOperator{\A}{A}

\DeclareMathOperator{\chow}{CH}

\DeclareMathOperator{\characteristic}{char}
\def\ra{\rightarrow}
\def\tensor{\otimes}

\title{On the image of the second $l$-adic Bloch map}

\author{Jeffrey D. Achter}
\address{Colorado State University, Department of Mathematics,
 Fort Collins, CO 80523,
 USA}
\email{j.achter@colostate.edu}

\author{Sebastian Casalaina-Martin }
\address{University of Colorado, Department of Mathematics,
 Boulder, CO 80309, USA }
\email{casa@math.colorado.edu}

\author{Charles Vial}
\address{Universit\"at Bielefeld, Germany}
\email{vial@math.uni-bielefeld.de}

\thanks{The first- and second-named authors were partially supported
  by grants 637075 and 581058, respectively, from the Simons
  Foundation.}

\AtBeginDocument{
	\def\MR#1{}
}

\begin{document}

 \begin{abstract} 
      For a smooth projective geometrically uniruled threefold defined over a perfect field we show that there exists a canonical abelian variety over the field, namely the second algebraic representative, whose rational Tate modules model canonically the third $l$-adic cohomology groups of the variety for all primes $l$. In addition, there exists a rational correspondence inducing these identifications. In the case of a geometrically  rationally chain connected variety, one obtains canonical identifications between the \emph{integral} Tate modules of the second algebraic representative and the third $l$-adic cohomology groups of the variety, and if the variety is a geometrically stably rational threefold, these identifications are  induced by an integral correspondence. 
 Our overall strategy consists in studying -- for arbitrary smooth projective varieties -- the image of the
  second $l$-adic Bloch map restricted to the Tate module of algebraically trivial
  cycle classes in terms of the ``correspondence (co)niveau
  filtration''.  This complements results with rational coefficients
  due to Suwa.
  In the appendix, we review the construction of the Bloch map and its basic properties.
    \end{abstract}

  \maketitle

 \section*{Introduction}

\subsection{Mazur's question with $\mathbb Q_\ell$-coefficients}\label{SS:MazurQ}

In the context of the generalized Hodge conjecture, it is natural to ask the following question \cite[Que.~2.43]{voisinDiag}\,: \emph{Let $X$ be a smooth 
complex projective manifold, and let $\nu,i$ be natural numbers.  Given a 
weight-$i$ Hodge structure $L\subseteq H^i(X,\mathbb Q)$ such that the Tate twist $L(\nu)$ is effective (\emph{i.e.}, $L^{p,q}=0$ for $p,q>\nu$), does there exist a complex projective manifold $Y$ and an inclusion of Hodge structures $L(\nu)\hookrightarrow H^{i-2\nu}(Y,\mathbb Q)$\,?}  
Essentially by definition, one can rephrase this question, via the Hodge coniveau filtration $\coniveau^\bullet_H H^i(X,\mathbb Q)$,  as asking whether for a given $\nu$, there exists a smooth projective manifold $Y$ such that $\coniveau^\nu_H H^i(X,\mathbb Q(\nu))\subseteq H^{i-2\nu}(Y,\mathbb Q)$.  The generalized Hodge conjecture predicts that the Hodge coniveau filtration coincides with the so-called geometric coniveau filtration, $\coniveau^\bullet H^i(X,\mathbb Q)$\,; in this sense one can rephrase the question in terms of the geometric coniveau filtration, and it is this version we will focus on in this paper.

 As a motivating example for this work, consider the case where $i=2n-1$ is odd and where $\nu=n$. 
Setting  $J^{2n-1}_a(X)$ to be the algebraic intermediate Jacobian, \emph{i.e.}, the image of the Abel--Jacobi map $AJ: \operatorname{A}^n(X)\to J^{2n-1}(X)$ restricted to algebraically trivial cycles, it is well-known that $\coniveau^{n-1}H^{2n-1}(X,\mathbb Q)= H^1(\widehat J^{2n-1}_a(X),\mathbb Q)\simeq H^1( J^{2n-1}_a(X),\mathbb Q)$, answering the question in the case of the geometric coniveau filtration.  \medskip

For smooth projective varieties over arbitrary fields, one can rephrase the Hodge theoretic question above by replacing Betti cohomology with $\ell$-adic cohomology.  
 Mazur \cite{mazurprobICCM,mazurprob} asked the following\,:

\begin{quealpha}[Mazur's question with $\mathbb Q_\ell$-coefficients] \label{Q:MazurQ}
	Let $X$ be a smooth projective variety over a
	field~$K$ with separable closure $\bar{K}$. Given a natural number $n$,  
	does there exist an abelian variety $A/K$ such that for all primes $\ell\neq \characteristic(K)$ there is an isomorphism of Galois modules
	\begin{equation}\label{E:model-QQ2}
	\begin{CD}
	V_\ell A @>\simeq>>   \coniveau^{n-1}H^{2n-1}(X_{\bar K},\mathbb Q_\ell(n)) \ \text{?}
	\end{CD}
	\end{equation}
\end{quealpha}
\noindent In fact, one can pose an analogue of Mazur's question for any Weil
cohomology $\mathcal H(\cdot)$.  In positive characteristic $p$, much of our
work also extends to the case of cohomology $H(-,\rat_p)$, which can
be recovered as the $F$-invariants of the crystalline cohomology.  In
this introduction, where possible,  we phrase statements uniformly in a
prime $l$, although we remind the reader that, for example, $\dim
H^i(X,\rat_p)$ is typically smaller than $\dim H^i(X,\rat_\ell)$\,; we
reserve $\ell$ for primes distinct from the characteristic of the base
field. \medskip

From the motivic perspective, it is natural to ask that the isomorphism~\eqref{E:model-QQ2} be induced by a correspondence.  
Note that given the isomorphism~\eqref{E:model-QQ2}, the Tate conjecture
provides for each $\ell$ a correspondence $\Gamma_\ell \in \operatorname{CH}^n(A \times_{ K} X)\otimes \rat_\ell$ inducing the isomorphism for that $\ell$.  One might expect to find a correspondence $\Gamma$ with integral coefficients, and that it be  independent of $l$\,:

\begin{quealpha}[Mazur's motivic question with $\mathbb Q_l$-coefficients] \label{Q:MazurQMot}
Does there exist a correspondence $\Gamma \in \operatorname{CH}^n(A \times_{ K} X)$ inducing for all primes $l$ the above isomorphisms~\eqref{E:model-QQ2}\,?
\end{quealpha}

Our first observation is that our results in \cite{ACMVdmij} provide an affirmative answer to Questions \ref{Q:MazurQ} and  \ref{Q:MazurQMot} for any field $K$ of characteristic zero, and for any $n$. Recall that in \cite{ACMVdmij}, it is shown that the algebraic intermediate Jacobian $J_a^{2n-1}(X_{\mathbb C})$ attached to a smooth projective variety $X$ defined over $K\subseteq \cx$ admits a \emph{distinguished model} $J^{2n-1}_{a,X/K}$ over $K$ in the sense that this model makes the Abel--Jacobi map $AJ : \operatorname{A}^n(X_\cx) \to J_a^{2n-1}(X_{\mathbb C})$ restricted to algebraically trivial cycles an $\operatorname{Aut}(\cx/K)$-equivariant map.

\begin{teoalpha}[{\cite[Thm.~2.1]{ACMVdmij}}]\label{T:MazurQ}
		Let $X$ be a smooth projective variety over a
	field~$K\subseteq \cx$. Given a natural number $n$,  let  $J^{2n-1}_{a,X/K}$ denote the distinguished model of the intermediate Jacobian $J^{2n-1}(X_{\mathbb C})$.
Then there exists a correspondence $\Gamma \in \operatorname{CH}^n(J^{2n-1}_{a,X/K} \times_{ K} X)$ inducing for all primes $\ell$ an inclusion  of Galois modules
	\begin{equation*}
	\xymatrix{\Gamma_* : \ V_\ell J^{2n-1}_{a,X/K} \ \ar@{^(->}[r] & H^{2n-1}(X_{\bar K},\mathbb Q_\ell(n))}
	\end{equation*}
with image $\coniveau^{n-1}H^{2n-1}(X_{\bar K},\mathbb Q_\ell(n))$.
\end{teoalpha}

Thus we turn next to Questions~\ref{Q:MazurQ} and \ref{Q:MazurQMot} in the case where $K$ has positive characteristic. 
As a first partial result, we establish in Proposition~\ref{P:easy} a positive answer to Question~\ref{Q:MazurQ} and~\ref{Q:MazurQMot} under the further assumptions that $K$ is perfect, $2n-1 \leq d_X:= \dim X$, and $H^{2n-1}(X_{\bar K}, \mathbb Q_l(n))$ has geometric coniveau~$n-1$ (this was established in \cite[Thm.~2.1(d)]{ACMVdcg} for $K\subseteq \cx$).
While the condition that $H^{2n-1}(X_{\bar K}, \mathbb Q_l(n))$ have
geometric coniveau~$n-1$ is in general quite restrictive, it does not
impose a condition for $n=1$ or $n=d_X$, and therefore
Proposition~\ref{P:easy} establishes an affirmative answer to  Mazur's
Questions \ref{Q:MazurQ} and \ref{Q:MazurQMot} for $n=1,d_X$ and $K$
perfect (see Remark~\ref{R:easy} for the case $n=d_X$).   Thus, moving
forward, we will essentially be focusing on the case $n=2$ in positive
characteristic.

\subsection{Mazur's question with $\mathbb Q_\ell$-coefficients in positive characteristic}

We focus now on Question~\ref{Q:MazurQ} in positive characteristic, and set aside the issue of the correspondence in Question~\ref{Q:MazurQMot}.   
We  attempt to use algebraic representatives and the Bloch map as replacements for intermediate Jacobians and Abel--Jacobi maps.

More precisely,  recall that over an arbitrary field a replacement for $J^{2n-1}_{a,X/K}$  is the \emph{algebraic representative} $\operatorname{Ab}^n_{X/K}$.  If it exists (as in the case $n=1,2,d_X$), it comes with a $\operatorname{Gal}(K)$-equivariant morphism
$$\phi^n_{X_{\bar K}/\bar K} : \operatorname{A}^n(X_{\bar K}) \longrightarrow \operatorname{Ab}^n_{X/K}(\bar K),$$
where as in the case $\bar K = \cx$ we set
$$\operatorname{A}^n(X_{\bar K}) := \{\alpha \in \operatorname{CH}^n(X_{\bar K}) \ \big\vert \ \alpha \mbox{ is algebraically trivial} \}.$$
 For $n=1,d_X$, the algebraic representative is the reduced Picard variety with the Abel--Jacobi map,  and the Albanese variety with the Albanese map,  respectively.  
 For the case $n=2$, the existence
was proved for smooth projective varieties over an algebraically closed field in \cite{murre83}.  This was extended to smooth projective varieties defined over a perfect field (\emph{e.g.}, a finite field) in \cite{ACMVdcg}, and to  smooth projective varieties defined over any field in \cite{ACMVfunctor}. 
In characteristic~$0$, this agrees with the distinguished model of $J^3_a(X_{\mathbb C})$ of \cite{ACMVdmij}.  We refer to \S \ref{SS:AlgRep} for more details.

At the same time, recall  that for a smooth projective variety $X$ over a field~$K$   
with separable closure $\bar{K}$, and a prime $l$,
	Bloch \cite{bloch79} in the case $l\neq \characteristic(K)$ and later Gros--Suwa~\cite{grossuwaAJ} in the case $l=\characteristic(K)$ defined a map
   \begin{equation}
  \label{E:introbloch}
\lambda^n :\ \operatorname{CH}^n(X_{\bar K})[l^\infty] \longrightarrow
H^{2n-1}(X_{\bar K},\rat_l/\mathbb Z_l(n))
\end{equation}
on $l$-primary torsion, extending the Abel--Jacobi map on homologically trivial $l$-torsion cycle classes (see~\eqref{E:an-Bloch}) in the case $K=\mathbb C$.  
Suwa \cite{Suwa}  in the case $l\neq \characteristic(K)$ and Gros--Suwa~\cite{grossuwaAJ} in the case $l=\characteristic(K)$ then defined an  \emph{$l$-adic Bloch map}
\begin{equation}
  \label{E:introladicbloch}
T_l \lambda^n :\ T_l \operatorname{CH}^n(X_{\bar K}) \longrightarrow
H^{2n-1}(X_{\bar K},\mathbb Z_l(n))_\tau
\end{equation}
by taking the Tate module of the Bloch map $\lambda^n$.   Here the  subscript $\tau$ indicates the
quotient by the torsion subgroup\,; recall  that a result of Gabber states that for a given $X$, the
cohomology groups are torsion-free for all but finitely many $l$.
Tensoring by $-\otimes_{\mathbb Z_l} \mathbb Q_l$ defines a map 
$$V_l \lambda^n : V_l \operatorname{CH}^n(X_{\bar K}) \longrightarrow
H^{2n-1}(X_{\bar K},\mathbb Q_l(n)).$$  In the appendix, we give a more direct construction of the $\ell$-adic Bloch map for primes $\ell \neq \characteristic(K)$, following Bloch's original construction, but taking an inverse limit rather than a direct limit.  We show these two constructions agree, and that the $\ell$-adic Bloch map agrees with the $\ell$-adic Abel--Jacobi map in the case $K=\mathbb C$, when one restricts to homologically trivial cycle classes.  We then review in~\S \ref{S:properties} a few properties of the $l$-adic Bloch map that we use in the body of the paper.

It is well-known that one can use these maps to model $H^{2n-1}(X_{\bar K},\mathbb Q_l(n))$ for $n=1,d_X$ via the Picard and Albanese\,: in those cases, we have isomorphisms (see Propositions~\ref{P:Kummer} and~\ref{P:Rojtman})
	\begin{equation*}
	\xymatrix@C=2em {
		V_l ({\operatorname{Pic}^0_{X/K}}_{\operatorname{red}}) \ar[rr]^<>(0.5){(V_l\phi^1_{X_{\bar K}/\bar K})^{-1}}_\simeq &&V_l \operatorname{A}^1(X_{\bar K})\ar@{=}[r]&V_l
		\operatorname{CH}^1(X_{\bar K}) \ar@{->}[rr]^<>(0.5){V_l\lambda^1}_{\simeq\quad}&&H^{1}(X_{\bar
			K},\mathbb Q_l(1))\\
		V_l \operatorname{Alb}_{X/K} \ar[rr]^<>(0.5){(V_l\phi^{d_X}_{X_{\bar K}/\bar K})^{-1}}_\simeq &&V_l \operatorname{A}^{d_X}(X_{\bar K})\ar@{=}[r]&V_l
		\operatorname{CH}^{d_X}(X_{\bar K}) \ar@{->}[rr]^<>(0.5){V_l\lambda^{d_X}}_{\simeq\qquad}&&H^{2d_X-1}(X_{\bar
			K},\mathbb Q_l(d_X)).
	}
	\end{equation*}
Therefore, we focus here on the cases $n\ne 1,d_X$, and in particular, the case $n=2$.    
While in positive characteristic,  the relationship between 
the Tate module of the algebraic representative and the coniveau filtration is not known in general, 
a result of Suwa relates the coniveau filtration to the image of the $l$-adic Bloch map (restricted to algebraically trivial cycles).

\begin{proalpha}[{Suwa \cite[Prop.~5.2]{Suwa}}]\label{P:Suwa}
	Let $X$ be a smooth projective variety over a perfect field $K$ and let $n$ be a natural number.    
	For all prime numbers  $l$,
	the image of the composition 
	\begin{equation}
	\xymatrix{
		V_l \operatorname{A}^n(X_{\bar K})\ar@{^(->}[r]&V_l
		\operatorname{CH}^n(X_{\bar K}) \ar@{->}[r]^<>(0.5){V_l\lambda^n}&H^{2n-1}(X_{\bar
			K},\mathbb Q_l(n))\\
	}
	\end{equation}
	is equal to  ${\coniveau}^{n-1} H^{2n-1}(X_{\bar K},\mathbb Q_l(n))$.  
\end{proalpha}

We refer to Proposition~\ref{P:ImBloch-Pre} for a more precise statement.
As a minor technical point, we mention that Suwa proves this result for a slightly different coniveau filtration, which we call the correspondence coniveau filtration\,; he shows this filtration agrees with the usual coniveau filtration in characteristic $0$, and we extend this result to perfect fields in Proposition~\ref{P:Jannsen}.

Since $T_l \lambda^2$ is an inclusion (see Proposition~\ref{P:M-9.2}), 
as an immediate corollary of our results above, and Suwa's proposition, one obtains\,:

\begin{coralpha}[Modeling $\mathbb Q_l$-cohomology] \label{C:mainQ} 
 	Let $X$ be a smooth projective variety over a perfect field $K$ and let $l$ be a prime number. 
	 If  $V_l \phi^2_{X_{\bar K}/\bar K}: V_l\operatorname{A}^2(X_{\bar K})\to V_l \operatorname{Ab}^2_{X/K}$ is an isomorphism, \emph{e.g.}, if $X$ is a geometrically uniruled threefold (Proposition~\ref{P:phi}(3)), 
	then the composition 
	\begin{equation} \label{E:canQ}
	\xymatrix{
		V_l \operatorname{Ab}^2_{X/K} \ar[rr]^<>(0.5){(V_l\phi^2_{X_{\bar K}/\bar K})^{-1}}_\simeq &&V_l \operatorname{A}^2(X_{\bar K})\ar@{^(->}[r]&V_l
		\operatorname{CH}^2(X_{\bar K}) \ar@{^(->}[rr]^<>(0.5){V_l\lambda^2}&&H^{3}(X_{\bar
			K},\mathbb Q_l(2))
	}
	\end{equation}
	 is an inclusion  of $\operatorname{Gal}(K)$-modules with image ${\coniveau}^{1} H^{3}(X_{\bar K},\mathbb Q_l(2))$. 
   \end{coralpha}

The assumption that $ V_l\phi^2_{X_{\bar K}/\bar K}$ be an isomorphism is implied  (see Lemma~\ref{L:SRHom-facts}) by the, possibly vacuous,  
 assumption that
$ \phi^2_{X_{\bar K}/\bar K}: \operatorname{A}^2(X_{\bar K})\to\operatorname{Ab}^2_{X/K}$ be an isomorphism on $l$-primary torsion.  
It turns out that this assumption also plays a crucial role in our 
 work~\cite{ACMVdiag}, so we single it out\,:

\begin{dfnalpha}[Standard assumption at $l$] \label{D:StAs}
 	 Let $X$ be a smooth projective variety over a field $K$ and let $l$ be a prime number. We say that $\phi^2_{X_{\bar K}/\bar K}$ (or by abuse, $X$) satisfies the \emph{standard assumption at the prime $l$} if  
	\begin{equation}
\begin{CD}
\phi^2_{X_{\bar K}/\bar K}[l^\infty]:\ \operatorname{A}^2(X_{\bar K})[l^\infty]  @>\simeq>>   \operatorname{Ab}^2_{X/K}[l^\infty]
\end{CD}
	\end{equation}
is an isomorphism. We say that $\phi^2_{X_{\bar K}/\bar K}$ (or by abuse $X$) satisfies the standard assumption if it satisfies the standard assumption at $l$ for all primes $l$.
\end{dfnalpha}

In Proposition~\ref{P:phi}, we give sufficient conditions for the standard assumption to be satisfied\,; in particular, the standard assumption holds if $\operatorname{char}(K)=0$ \cite[Thm.~10.3]{murre83} or if $X$ is geometrically rationally connected with $K$ perfect. We are unaware of an example of a smooth projective variety for which the  standard assumption at a prime $l$ fails.

We  mention  here that both Proposition~\ref{P:Suwa} and Corollary~\ref{C:mainQ} hold with $\mathbb Q_l/\mathbb Z_l$-coefficients so long as one replaces the coniveau filtration~\eqref{eq:N} with the correspondence niveau filtration~\eqref{eq:Ntilde},
 thus providing an answer to Mazur's Question~\ref{Q:MazurQ} with $\mathbb Q_l/\mathbb Z_l$-coefficients and with the geometric coniveau filtration $\coniveau^\bullet$ replaced with the correspondence niveau filtration $\widetilde{\coniveau}^\bullet$.

\medskip 
As a final note, we mention that in principle, the technique used to prove 
Corollary~\ref{C:mainQ} would work for  any $n$, assuming that there exists an algebraic representative in codimension-$n$,  that $V_l \phi^n_{X_{\bar K}/\bar K}$ is an isomorphism, and that $V_l \lambda^n$ is an inclusion\,; however, unlike the case $n=1,2,d_X$,  
in general, for $n\ne 1,2,d_X$,  one does not expect these conditions to hold.  Nevertheless, in the body of the paper, we explain the general case, and indicate where special assumptions are needed.

\subsection{Mazur's question with $\mathbb Z_\ell$-coefficients}

Next we consider Mazur's question in the case of  $\mathbb
Z_l$-coefficients (while acknowledging that, in positive
characteristic $p$ and with $l=p$, it might be more natural to seek an isomorphism of
$F$-crystals than an isomorphism of cohomology groups
$H^\bullet(-,\integ_p)$). To start with, as in the case of $\mathbb Q_l$-coefficients, it is well-known that one can model  $H^{2n-1}(X_{\bar K},\mathbb Z_l(n))_\tau$ for $n=1,d_X$ via the Picard and Albanese\,: in those cases, we have isomorphisms (see Propositions~\ref{P:Kummer} and~\ref{P:Rojtman})
	\begin{equation*}
	\xymatrix@C=2em {
		T_l (\operatorname{Pic}^0_{X/K})_{\operatorname{red}} \ar[rr]^<>(0.5){(T_l\phi^1_{X_{\bar K}/\bar K})^{-1}}_\simeq &&T_l \operatorname{A}^1(X_{\bar K})\ar@{=}[r]&T_l
		\operatorname{CH}^1(X_{\bar K}) \ar@{->}[rr]^<>(0.5){T_l\lambda^1}_{\simeq\quad}&&H^{1}(X_{\bar
			K},\mathbb Z_l(1))\\
		T_l \operatorname{Alb}_{X/K} \ar[rr]^<>(0.5){(T_l\phi^{d_X}_{X_{\bar K}/\bar K})^{-1}}_\simeq &&T_l \operatorname{A}^{d_X}(X_{\bar K})\ar@{=}[r]&T_l
		\operatorname{CH}^{d_X}(X_{\bar K}) \ar@{->}[rr]^<>(0.5){T_l\lambda^{d_X}}_{\simeq\qquad}&&H^{2d_X-1}(X_{\bar
			K},\mathbb Z_l(d_X))_\tau
	}
	\end{equation*}

Motivated by the discussion in \S \ref{SS:MazurQ} leading to Corollary~\ref{C:mainQ}, we 
proceed in a similar way, focusing on the case $n=2$. 
The starting point is again to 
assume
that  $ \phi^2_{X_{\bar K}/\bar K}[l^\infty]: \operatorname{A}^2(X_{\bar K})[l^\infty]\to\operatorname{Ab}^2_{X/K}[l^\infty]$ is an isomorphism\,; \emph{i.e.}, we assume the so-called standard assumption for $n=2$ (Definition~\ref{D:StAs})\,; \emph{e.g.}, we assume $\characteristic(K)=0$ or $X$ is geometrically rationally chain connected.

By taking Tate modules one has (Lemma~\ref{L:SRHom-facts}) that 
\begin{equation}\label{E:StAs}
T_l \phi^2_{X_{\bar K}/\bar K}:T_l \operatorname{A}^2(X_{\bar K})\longrightarrow T_l \operatorname{Ab}^2_{X/K}
\end{equation}
is an isomorphism as well.  Consequently, one can consider 
the composition  
 \begin{equation}\label{E:AbInc}
\xymatrix{
	T_l \operatorname{Ab}^2_{X/K} \ar[rr]^<>(0.5){(T_l\phi^2_{X_{\bar K}/\bar K})^{-1}}_\simeq &&T_l \operatorname{A}^2(X_{\bar K})\ar@{^(->}[r]&T_l
	\operatorname{CH}^2(X_{\bar K}) \ar@{^(->}[rr]^<>(0.5){T_l\lambda^2}&&H^{3}(X_{\bar
		K},\mathbb Z_l(2))_\tau.
}
\end{equation}
That $T_l \lambda^2$ is an inclusion is reviewed in Proposition~\ref{P:M-9.2}.  
In Proposition~\ref{P:ImBloch-Pre}, we show that the image of the map  $T_l \lambda^2 :   T_l \operatorname{A}^2(X_{\bar K}) \to H^{3}(X_{\bar
	K},\mathbb Z_l(2))_\tau$ contains ${\widetilde \coniveau}^{n-1}H^{2n-1}(X_{\bar K},\mathbb Z_l(n))_\tau$\,; here ${\widetilde \coniveau}^\bullet$ is the \emph{correspondence niveau filtration} defined in~\eqref{eq:Ntilde}. 
Combined with Suwa's Proposition~\ref{P:Suwa} (together with Proposition~\ref{P:Jannsen} comparing ${\widetilde \coniveau}^\bullet$ with $\coniveau^\bullet$), we find that the image contains  ${\widetilde \coniveau}^{n-1}H^{2n-1}(X_{\bar K},\mathbb Z_l(n))_\tau$ as a finite index subgroup. With the expectation that the standard assumption should be true in general, we are thus led to ask\,:

 \begin{quealpha}[Mazur's question with $\mathbb Z_l$-coefficients] \label{quealpha}
 	Let $X$ be a smooth projective variety over a
 	field~$K$. 
 	Does there exist an abelian variety $A/K$ such that for almost all primes $l$ there is an isomorphism of Galois modules
 	\begin{equation}\label{E:model-ZZ}
 	\begin{CD}
 	T_l A @>\simeq>>  {\widetilde \coniveau}^{n-1}H^{2n-1}(X_{\bar K},\mathbb Z_l(n))_\tau \ \text{?}
 	\end{CD}
 	\end{equation}
 	
 \end{quealpha}

   The main technical result of this paper is Theorem~\ref{T:mainBody}, a particular instance of which takes the form below.

 \begin{teoalpha}[Image of the $l$-adic Bloch map]\label{T:main}
  Let $X$ be a smooth projective   variety over a perfect field~$K$.  
   Assume 
that $ \phi^2_{X_{\bar K}/\bar K}[l^\infty]: \operatorname{A}^2(X_{\bar K})[l^\infty]\to\operatorname{Ab}^2_{X/K}[l^\infty]$ is an isomorphism for all but finitely many primes $l$\,; \emph{i.e.},   $X$ satisfies the standard assumption  at all but finitely many primes $l$ (Definition~\ref{D:StAs}).  
  Then, for all but finitely many prime numbers $l$,
   the image of the composition
   \begin{equation*}
   \xymatrix{
    T_l \operatorname{A}^2(X_{\bar K})\ar@{^(->}[r]&T_l
    \operatorname{CH}^2(X_{\bar K}) \ar@{^(->}[r]^<>(0.5){T_l\lambda^2}&H^{3}(X_{\bar
     K},\mathbb Z_l(2))_\tau\\
   }
   \end{equation*}
   is equal to  ${{\widetilde \coniveau}}^{1} H^{3}(X_{\bar K},\mathbb Z_l(2))_\tau$. 
 \end{teoalpha}

We obtain the following immediate corollary providing a partial answer to Question~\ref{quealpha}\,:

\begin{coralpha}[Modeling $\mathbb Z_l$-cohomology] \label{C:main}
Under the hypotheses of Theorem~\ref{T:main}, 
 the inclusion~\eqref{E:AbInc} induces an isomorphism of Galois modules
$$
\begin{CD}
T_l \operatorname{Ab}^2_{X/K} @>\simeq>> {\widetilde \coniveau}^1H^{3}(X_{\bar K},\mathbb Z_l(2))_\tau
\end{CD}
$$
for all but finitely many  prime numbers $l$.
\end{coralpha}

 Theorem~\ref{T:main} is proved in \S \ref{SS:proofThm1}. In fact, we can control the primes for which Theorem~\ref{T:main} and Corollary~\ref{C:main} might fail\,; this is related to miniversal cycle classes, as well as decomposition of the diagonal (see Theorem~\ref{T:mainBody}).
 Moreover, in a way we make precise in Lemma~\ref{L:StAs-spec} and Proposition~\ref{P:finitefield}, to prove  Theorem~\ref{T:main} and Corollary~\ref{C:main}, if one knows the standard assumption holds for all varieties over finite fields, then one does not need to assume the standard assumption for $X$.

We note again that in principle, the technique used to prove 
the Corollary~\ref{C:main} would work for  any $n$, assuming that there exists an algebraic representative in codimension~$n$,  that $ \phi^n_{X_{\bar K}/\bar K}[l^\infty]$ is an isomorphism, and that $T_l \lambda^n$ is an inclusion\,; but   
in general, for $n\ne 1,2,d_X$,  one does not expect these conditions to hold.  Nevertheless, in the body of the paper, we explain the general case, and indicate where special assumptions are needed.

\subsection{Universal cycles and the image of the second $l$-adic Bloch map}
\label{SS:uni-intro}

Still under the standard assumption that 
$$
\begin{CD}
\phi^2_{X_{\bar K}/\bar K}[l^\infty]: \operatorname{A}^2(X_{\bar K})[l^\infty] @>\simeq>> \operatorname{Ab}^2_{X/K}[l^\infty]
\end{CD}
$$ 
is an isomorphism for all primes $l$, we show that a sufficient condition for the composition~\eqref{E:AbInc} to have image equal to $ {{\widetilde \coniveau}}^{1} H^{3}(X_{\bar K},\mathbb Z_l(2))_\tau$
 for all $l$ is provided by the existence of a so-called \emph{universal cycle} for~$\phi^2_{X_{\bar K}/\bar K}$\,; see \S \ref{SS:mini} for a definition.
 As before, we start by determining the image of the second $l$-adic Bloch map under such conditions\,:

\begin{teoalpha}[Universal cycles and the image of the second $l$-adic Bloch map]\label{T:main-uni}
	Let $X$ be a smooth projective variety over a perfect field~$K$.  
Assume that $X$ satisfies the standard assumption
	for all primes $l$.
	If $\phi^2_{X_{\bar K}/\bar K}: \operatorname{A}^2(X_{\bar K})\to\operatorname{Ab}^2_{X/K}(\bar K)$ admits a universal cycle, then  for all prime numbers 
	$l$
	the image of the composition
	\begin{equation*}
		\xymatrix{
			T_l \operatorname{A}^2(X_{\bar K})\ar@{^(->}[r]&T_l
			\operatorname{CH}^2(X_{\bar K}) \ar@{^(->}[r]^<>(0.5){T_l\lambda^2}&H^{3}(X_{\bar
				K},\mathbb Z_l(2))_\tau\\
		}
	\end{equation*}
	is equal to  ${{\widetilde \coniveau}}^{1} H^{3}(X_{\bar K},\mathbb Z_l(2))_\tau$.
\end{teoalpha}

Theorem~\ref{T:main-uni} is a particular instance of our main Theorem~\ref{T:mainBody}.
As an immediate consequence, we obtain\,:

\begin{coralpha}[Universal cycles and modeling $\mathbb Z_l$-cohomology] \label{C:main-uni}
	Under the hypotheses of Theorem~\ref{T:main-uni}, 
	the inclusion~\eqref{E:AbInc} induces for \emph{all} prime numbers $l$  an isomorphism of Galois modules
	$$\begin{CD}
	T_l \operatorname{Ab}^2_{X/K}  @>\simeq>>  {\widetilde \coniveau}^1H^3(X_{\bar K},\mathbb Z_l(2))_\tau.
	\end{CD}
	$$
	
\end{coralpha}

Due to the connection with universal cycles and decomposition of the diagonal (see Proposition~\ref{P:decmini}), this is connected with the notion of rationality, as we discuss in the next section.

\subsection{Decomposition of the diagonal and the image of the second $l$-adic Bloch map}
\label{SS:stbrat-intro}
We next turn our focus to the case of smooth projective varieties $X$ over a perfect field $K$ with $\operatorname{CH}_0(X_{\bar K})\otimes \rat$  universally supported in dimension~2 (see Definition~\ref{D:unisupport}).
 It is well-known, via a decomposition of the diagonal argument \cite{BS}, that in this case we have ${\coniveau}^{1} H^{3}(X_{\bar K},\mathbb Q_l(2)) =  H^{3}(X_{\bar K},\mathbb Q_l(2))$, but also that $V_l \phi^2_{X_{\bar K}/\bar K}$ is 
 an isomorphism (Proposition~\ref{P:phi}(3)), for all primes $l$. 
 As a consequence, we see that in this case \eqref{E:canQ} induces an  isomorphism $V_l \operatorname{Ab}^2_{X/K}\simeq H^3(X_{\bar K},\mathbb Q_l(2))$.  
We show the following result for cohomology with $\mathbb Z_l$-coefficients\,:

\begin{teoalpha}
	[Decomposition of the diagonal and the image of the second $l$-adic Bloch map]
	\label{T:co=H} 
    Let $X$ be a smooth projective variety over a perfect field $K$ of characteristic exponent $p$. 
  \begin{enumerate}
\item Assume $\operatorname{CH}_0(X_{\bar K})\otimes \integ\big[{\frac{1}{N}}\big]$ is universally supported in dimension 2 for some $N>0$, \emph{e.g.}\ $X$ is a geometrically uniruled threefold. 
 Then,
for all primes $\ell\nmid Np$, the inclusion 
${\widetilde \coniveau}^1	H^3(X_{\bar K},\integ_\ell(2)) \subseteq 	H^3(X_{\bar K},\integ_\ell(2))$ is an equality and the second $\ell$-adic Bloch map restricted to algebraically trivial cycles
\begin{equation*}
\xymatrix{
	T_\ell \operatorname{A}^2(X_{\bar K})\ar@{^(->}[r]&T_\ell
	\operatorname{CH}^2(X_{\bar K}) \ar@{^(->}[r]^<>(0.5){T_\ell\lambda^2}&H^{3}(X_{\bar K},\mathbb Z_\ell(2))_\tau\\
}
\end{equation*} is an isomorphism of
$\operatorname{Gal}(K)$-modules. 

\item Assume $\operatorname{CH}_0(X_{\bar K})\otimes \integ\big[{\frac{1}{N}}\big]$ is universally supported in dimension 1 for some $N>0$, \emph{e.g.}\  $X$ is a geometrically rationally chain connected. 
Then, for all primes $l$,  the second $l$-adic Bloch map restricted to algebraically trivial cycles
 \begin{equation*}
\xymatrix{
	T_l \operatorname{A}^2(X_{\bar K})\ar@{^(->}[r]&T_l
	\operatorname{CH}^2(X_{\bar K}) \ar@{^(->}[r]^<>(0.5){T_l\lambda^2}&H^{3}(X_{\bar K},\mathbb Z_l(2))_\tau\\
}
\end{equation*} is an isomorphism of
$\operatorname{Gal}(K)$-modules. Moreover, for all primes $\ell \nmid Np$, $H^3(X_{\bar K},\integ_\ell(2))$ is torsion-free.

   \end{enumerate}
\end{teoalpha}

 A slight generalization of Theorem~\ref{T:co=H} that deals with the prime $p$ in case resolution of singularities holds over $K$ in dimensions $< \dim X$ is proved in Proposition~\ref{P:lambda}.
    See \S \ref{SS:dec}, and in particular Remark~\ref{R:stablyrat}, for the classical link between stable rationality, rational connectedness, and decomposition of the diagonal.  
\medskip

 Again, from Theorem~\ref{T:co=H} (combined with Proposition~\ref{P:phi}), we have the following corollary.
 
\begin{coralpha}[Decomposition of the diagonal and modeling $\mathbb Z_\ell$-cohomology] \label{C:main2}
  Let $X$ be a smooth projective variety over a perfect field~$K$ of characteristic exponent~$p$. 
  \begin{enumerate}
  	\item Assume $\operatorname{CH}_0(X_{\bar K})\otimes \integ\big[{\frac{1}{N}}\big]$ is universally supported in dimension 2 for some $N>0$, \emph{e.g.}\ $X$ is a geometrically uniruled threefold. 
  	   	Then,
  	for all primes $\ell\nmid Np$, $T_\ell\phi^2_{X_{\bar K}/\bar K} : T_\ell \operatorname{A}^2(X_{\bar K})\longrightarrow T_\ell \operatorname{Ab}^2_{X/K}$ is an isomorphism and the canonical inclusion~\eqref{E:AbInc} induces an isomorphism of Galois modules
  	$$
  	\begin{CD}
  	T_\ell \operatorname{Ab}^2_{X/K} @>\simeq>>  H^3(X_{\bar K},\mathbb Z_\ell(2))_\tau.
  	\end{CD}
  	$$
  	
  	\item Assume $\operatorname{CH}_0(X_{\bar K})\otimes \rat$ is universally supported in dimension 1, \emph{e.g.}\  $X$ is a geometrically rationally chain connected. Then, for all primes $l$, $T_l\phi^2_{X_{\bar K}/\bar K} : T_l \operatorname{A}^2(X_{\bar K})\longrightarrow T_l \operatorname{Ab}^2_{X/K}$ is an isomorphism and the canonical inclusion~\eqref{E:AbInc} induces an isomorphism of Galois modules
  	$$
  	\begin{CD}
  	T_l \operatorname{Ab}^2_{X/K}@>\simeq>> H^3(X_{\bar K},\mathbb Z_l(2))_\tau.
  	\end{CD}
  	$$
  	   \end{enumerate}
\end{coralpha}

\subsection{Stably rational \emph{vs.}\ geometrically stably rational varieties over finite fields}
We now turn to the motivic question\,:

\begin{quealpha}[Mazur's motivic question with $\mathbb Z_\ell$-coefficients] \label{Q:MazurZ}
	For which smooth projective varieties $X$ over a field $K$ do there exist an abelian variety $A/K$ and  a correspondence $\Gamma \in \operatorname{CH}^2(A \times_{ K} X)$ such that \emph{for all} primes $\ell\neq \characteristic(K)$ 
	\begin{equation*}
	\begin{CD}
	\Gamma_*: T_\ell A @>\simeq>>  {\widetilde \coniveau}^1H^3(X_{\bar K},\mathbb Z_\ell(2))_\tau
	\end{CD}
	\end{equation*}
	is an isomorphism of $\operatorname{Gal}(K)$-modules\,?
\end{quealpha}

In other words, we turn now to the issue of addressing the existence of a correspondence $\Gamma\in \operatorname{CH}^{2}( \operatorname{Ab}^2_{X/K}\times _K X)$ 
inducing the isomorphisms~\eqref{E:model-ZZ}.
 As already mentioned, it is easy to establish  a positive answer under the further assumption that $2n-1 \leq \dim X$ and $H^{2n-1}(X_{\bar K}, \mathbb Q_\ell(n))$ has geometric coniveau~$n-1$\,; see Proposition~\ref{P:easy}.
In case $X$ is a smooth projective geometrically uniruled threefold, then Proposition~\ref{P:mazuruniruled} establishes more precisely
the existence of   a  correspondence
$\Gamma\in \operatorname{CH}^{2}( \operatorname{Ab}^2_{X/K}\times _K X)\otimes \rat$ such that the induced morphism of
Galois modules
$
\Gamma_*:\ V_l \operatorname{Ab}^2_{X/K} \stackrel{\simeq}{\longrightarrow} H^{3}(X_{\bar K}, \mathbb Q_l(2))
$ coincides with the canonical map~\eqref{E:canQ} and
 is an isomorphism for all primes $l$.\medskip

 On the other hand, due to the failure of the integral Tate conjecture over finite fields \cite{Antieau, Kameko, PY},
 an  isomorphism as in~\eqref{E:model-ZZ} might not be induced by some correspondence $\Gamma\in \operatorname{CH}^{2}( \operatorname{Ab}^2_{X/K}\times _K X)$. However, using the $\ell$-adic Bloch map, we provide a positive answer for the third $\ell$-adic cohomology groups of smooth projective  stably rational varieties over finite or algebraically closed fields, thereby addressing Question~\ref{Q:MazurZ}\,:

 \begin{teoalpha}[Modeling $\mathbb Z_\ell$-cohomology via correspondences]\label{T:model}
  Let $X$ be a smooth projective stably rational variety
  over a field $K$ that is either finite or algebraically closed. Then there  exists a correspondence $\Gamma\in \operatorname{CH}^{2}(\operatorname{Ab}^2_{X/K}\times _K
  X)$
  inducing for all primes $\ell \neq \mathrm{char}\, K$  the isomorphisms~\eqref{E:AbInc} 
   \begin{equation}\label{E:model1}
  \begin{CD}
  \Gamma_*:\ T_\ell \operatorname{Ab}^2_{X/K} @>\simeq>> H^{3}(X_{\bar K}, \mathbb Z_\ell(2))
  \end{CD}
  \end{equation}
  of $\operatorname{Gal}(K)$-modules.
  Moreover, if $ \operatorname{char}(K)=0$, the correspondence $\Gamma$ induces an
  isomorphism
  \begin{equation}\label{E:model2}
  \begin{CD}
  \Gamma_*:\ H_1(J^3(X_{\mathbb C}),\integ) @>\simeq>> H^{3}(X_{\cx}, \mathbb Z(2)).
  \end{CD}
  \end{equation}
 \end{teoalpha}

The proof of Theorem~\ref{T:model} is given in Theorem~\ref{T:coniFil},  \emph{via} a decomposition of the diagonal argument.
There we also explain how the conclusion of Theorem~\ref{T:model} holds at $l=p$ in case $\dim X \leq 4$, due to the existence resolution of singularities in dimensions $\leq 3$.
There are two reasons for restricting to algebraically closed fields or finite fields in Theorem~\ref{T:model}.  First, in order to use alterations, we restrict to the case of perfect fields.  Second,
 in order to obtain the existence of the universal line-bundle, we use that $K$ is finite or separably closed (see \cite[\S\ref{SS:UnivCodim1}]{ACMVfunctor}).

 \subsection{Notation and conventions}\label{S:NotCon}
 For a field $K$, we will denote by $\bar K$ a separable closure, and by $\bar K^a$ an algebraic closure. 
 If $X$ is a scheme of finite type over a field $K$, we denote by $\operatorname{CH}^i(X)$ its Chow group of codimension-$i$ cycle classes, and by $\operatorname{A}^i(X) \subset \operatorname{CH}^i(X)$ the subgroup consisting of algebraically trivial cycle classes (see \cite[\S 10.3]{fulton}). If $X$ is pure-dimensional, we denote $d_X$ its dimension. In case $X$ is smooth over $K$, we still denote $d_X$ its dimension, which should be thought of as a locally constant function on $X$.
 A \emph{variety} over $K$ is a separated geometrically reduced scheme of finite type over~$K$.   The symbol $l$ is allowed to denote an arbitrary prime, whereas $\ell$
is always assumed invertible in the base field $K$. The phrase ``for almost all'' means ``for all but finitely many''.

Let  $M$ be an abelian group,  let $l$ be a prime, and let $\nu$ be an integer.  We denote\,:
 $$
 \begin{array}{ll}
M_\mathrm{tors}&:=\operatorname{Tor}_1^{\mathbb Z}(\mathbb Q/\mathbb
  Z,M) =\text{ the torsion subgroup of $M$}\,;\\
M_{\operatorname{cotors}} &:= \text{ the quotient of $M$ by its largest
  divisible subgroup\,;} \\
M_\tau &:= M/M_{\operatorname{tors}} = \text{ the quotient of $M$ by its
  torsion subgroup.}\\
M[l^\nu]&:=\operatorname{Tor}_1^{\mathbb Z}(\mathbb Z/l^\nu\mathbb
  Z,M) = \text{the $l^\nu$-torsion subgroup of $M$\,;}\\
M[l^\infty] &:= \varinjlim_{\nu} M[l^\nu] =  \text{the $l$-primary
  torsion subgroup of $M$\,;}\\
T_l M &:= \varprojlim_{\nu} M[l^\nu] =
  \operatorname{Hom}_{\mathbb Z}(\mathbb Q_l/\mathbb Z_l,M) = \text{ the Tate
  module of $M$\,;}\\
V_l M &:=T_l M  \otimes_{\mathbb Z_l}\mathbb Q_{l}.
 \end{array}
$$

 In the definition of $T_l M$, the transition maps are given by the
 multiplication by $l$ morphisms $M[l^{\nu+1}] \stackrel{\cdot l}{\to} M[l^\nu]$ and the
 equality $\varprojlim_{\nu} M[l^\nu] = \operatorname{Hom}_{\mathbb Z}(\mathbb
 Q_l/\mathbb Z_l,M) $ can be found, \emph{e.g.}, in \cite[Prop.~0.19]{milneADT}.
 Note that $T_l M = T_l (M[l^\infty])$.
 Note also that for a $\mathbb Z_l$-module $M$, we have
 $M_{\mathrm{tors}}=\operatorname{Tor}_1^{\mathbb Z_l}(\mathbb
 Q_l/\mathbb Z_l,M)$.

 We denote the $l$-adic valuation by $v_l$, so that for a
 natural number $r$, we have $r = \prod_l l^{v_l(r)}$.

For a smooth projective variety $X$ over an algebraically closed field $K=\bar K^a=\bar K$, we denote by $H_i(X_{\bar K},\mathbb Z_\ell)$ the $\ell$-adic \emph{homology}.   This will  be primarily an indexing convention that is useful from the motivic perspective, since in the case where $X$ is smooth and projective of pure dimension~$d_X$, the cap product with the fundamental class of $X$ induces for all $i$ (\emph{e.g.}, \cite[p.173]{laumon76}) an isomorphism $-\cap[X]: H^{2d_X-i}(X_\bar K,\mathbb Z_\ell(d_X))\stackrel{\simeq}{\to}H_{i}(X_{\bar K},\mathbb Z_\ell)$.

 \section{On various notions of coniveau filtrations}\label{S:CoNi}

 Given a smooth projective variety $X$ over a field $K$, one obtains coniveau
 filtrations on cohomology with various coefficients. In~\eqref{eq:N},
 \eqref{eq:N'} and~\eqref{eq:Ntilde} below we recall the definitions of the (classical) geometric
 coniveau filtration $\coniveau^\bullet$, of the (less classical)
correspondence coniveau filtration $\coniveau'^\bullet$ and of the (still less classical) correspondence niveau filtration $\widetilde{\coniveau}^\bullet$.

 Although the filtrations might not agree in general,
 in Proposition~\ref{P:Jannsen} below we recall that they are related by 
 $$\widetilde \coniveau^\bullet \subseteq \coniveau'^\bullet \subseteq \coniveau^\bullet,$$
 where, over a
 perfect field $K$ and with $\rat_\ell$-coefficients, the second inclusion is an equality while the first is conjecturally an equality.
 
In this section, the ring of coefficients $\Lambda$ denotes either $\mathbb Z$, $\mathbb Q$, $\integ/\ell^r\integ$, $\integ_\ell$, $\rat_\ell$, or $\mathbb Q_\ell/\mathbb Z_\ell$.     Cohomology groups $H^\bullet(-,\Lambda)$ are computed in the corresponding topology\,: \emph{e.g.}, $H^\bullet(-,\mathbb Z)$ is computed in the analytic topology, while $H^\bullet(-,\mathbb Z/\ell^r\mathbb Z)$ is computed in the \'etale topology.

 \subsection{Recalling the geometric coniveau filtrations}\label{S:CoNiDef}
 Let $X$ be a smooth projective variety over a field $K$ with separable closure $\bar K$.
 We recall  the \emph{$\nu$-th piece of the geometric coniveau filtration}
 \begin{align}\label{eq:N}
 \coniveau^\nu H^{i}(X_{\bar K},\Lambda) &:= \sum_{Z\subset X_{\bar
 		K}} \operatorname{im} \Big(
 H^{i}_Z(X_{\bar K},\Lambda) \to H^{i}(X_{\bar K},\Lambda) \Big)\\
 &=\sum_{Z\subset X_{\bar
 		K}} \ker \Big(
 H^{i}(X_{\bar K},\Lambda) \to H^{i}(X_{\bar
 	K}\smallsetminus Z,\Lambda) \Big) \label{eq:Nker}
\end{align}
\noindent where the sum runs through all Zariski closed subsets $Z$ of $X_{\bar K}$ of
codimension $\geq \nu$.  (The equivalence of
\eqref{eq:N} and~\eqref{eq:Nker} comes from the long exact sequence of
a pair.)

 For our purpose, we will have to work with the following variant of the coniveau
 filtration.
 We recall the \emph{$\nu$-th piece of the correspondence coniveau
  filtration}
 \begin{equation}\label{eq:N'}
 {\coniveau'}^\nu H^{i}(X_{\bar K},\Lambda) :=  \sum_{\Gamma :Z\vdash  X_{\bar
 		K}}
 \operatorname{im} \Big( \Gamma_*: H^{i-2\nu}( Z,\Lambda(-\nu)) \to H^{i}(X_{\bar
  K},\Lambda)\Big),
 \end{equation}
 where the sum is over all smooth projective varieties $Z$ over $\bar K$ 
 and all correspondences $\Gamma \in \operatorname{CH}_{d_X-\nu}(Z\times_{\bar K}X_{\bar K})$.

 Let us also introduce a related filtration, which was considered in \cite{FM-filtration} and in \cite{Vial2}.
 The \emph{$\nu$-th piece of the correspondence niveau
 	filtration} is defined as
 \begin{equation}\label{eq:Ntilde}
 {\widetilde{\coniveau}}^\nu H^{i}(X_{\bar K},\Lambda) :=  \sum_{\Gamma :Z\vdash  X_{\bar
 		K}}
 \operatorname{im} \Big( \Gamma_*: H_{i-2\nu}( Z,\Lambda(\nu-i)) \to H^{i}(X_{\bar
 	K},\Lambda)\Big),
 \end{equation}
 where the sum is over all smooth
 projective varieties $Z$ over $\bar K$  and all correspondences $\Gamma \in \operatorname{CH}^{d_X-\nu}(Z\times_{\bar
 	K}X_{\bar
 	K})$. As we will see in the proof of Proposition~\ref{P:Jannsen} below, one may restrict the sum in~\eqref{eq:Ntilde} to smooth projective varieties $Z$ over $\bar K$ of pure dimension $i-2\nu$.
  Here, as outlined in the Notation and Conventions \S \ref{S:NotCon}, we use homology as a convenience\,; for $Z$ of pure dimension $d_Z$, we set $H_{i-2\nu}( Z,\Lambda(\nu-i)) := H^{2d_Z -(i-2\nu)}( Z,\Lambda(d_Z-\nu-i))$.

By considering fields of definitions of $Z$ (and $\Gamma$) that are finite Galois over $K$ and by considering Galois orbits,
we note that in the definitions of all three filtrations above, we could have restricted the sums to those $Z$ (and $\Gamma$) defined over $K$.
 \medskip
 
The above three filtrations admit the following already-known containment relation\,:

 \begin{pro}
  	\label{P:Jannsen}
  Let $X$ be a smooth projective variety over a field  $K$. 
  Suppose that $\Lambda$ is one of $\integ/ \ell^r\integ$, $\integ_\ell$, 
  $\rat_\ell$, or $\mathbb Q_\ell/\mathbb Z_\ell$
  or that $K = \cx$ and $\Lambda$ is one of $\mathbb Z$, $\mathbb Q$ or $\mathbb Q/\mathbb Z$.
  Then there are natural inclusions
  $$
  {\widetilde \coniveau}^\nu H^i(X_{\bar K},\Lambda)\subseteq   {\coniveau'}^\nu H^i(X_{\bar K},\Lambda)\subseteq \coniveau^\nu H^i(X_{\bar
   K},\Lambda).
  $$
  In case $(\nu,i) = (n,2n)$ or $(n-1,2n-1)$,    the first inclusion is an equality.
  Moreover, assuming $\Lambda$ is either $\rat_\ell$
or $\rat$, if $K$ is perfect 
 then the second inclusion is an equality, and if Grothendieck's Lefschetz standard conjecture holds then the first inclusion is an equality.
 \end{pro}

 \begin{proof} The containment ${\widetilde \coniveau}^\nu H^i(X_{\bar K},\Lambda)\subseteq \coniveau'^\nu H^i(X_{\bar
 		K},\Lambda)$ is explained in \cite[\S 1.1]{Vial2} in the case $K\subseteq \cx$ and $\Lambda = \rat$. The same argument applies here and we spell it out for the sake of completeness.
 	 First, note that up to replacing $Z$ with $Z\times_{\bar K} \mathbb P_{\bar K}^n$, 
 	 we can assume $\dim Z \geq i-2\nu$. 
 	 Second, if $\iota:Y\hookrightarrow Z$ is a smooth linear intersection of $Z$ of dimension $i-2\nu$, then the push-forward $\iota_*:H_{i-2\nu}( Y,\Lambda) \to H_{i-2\nu}( Z,\Lambda)$ is surjective by the Lefschetz hyperplane theorem (\emph{e.g.}, \cite[Exp.~XIV, Cor.~3.3]{SGA4-3}),
 	  and the image of $\Gamma_*:H_{i-2\nu}(Z,\Lambda(i-\nu)) \to H^{i}(X_{\bar K},\Lambda)$ coincides thus with the image of $\Gamma_*\circ \iota_*$. Therefore, 
 	one may restrict the sum in~\eqref{eq:Ntilde} to those smooth projective varieties $Z$ over $\bar K$ of pure dimension~$i-2\nu$.
 	In particular, since $H_{i-2\nu}( Z,\Lambda) = H^{i-2\nu}( Z,\Lambda(i-2\nu))$ when $\dim Z = i-2\nu$, we get the asserted containment.

As outlined in the argument in \cite[Prop.~1.1]{Vial2},
a sufficient condition for the first inclusion to be an equality is the following\,: if for all smooth projective varieties $Z$ over $\bar K$  there exists a smooth projective variety $Z'$ over $\bar K$ and a correspondence $L \in \operatorname{CH}^{i-2\nu}(Z'\times_{\bar K} Z)$
  inducing an isomorphism $L_* : H_{i-2\nu}(Z',\Lambda(i-\nu)) \stackrel{\simeq}{\longrightarrow} H^{i-2\nu}(Z,\Lambda(\nu))$, then the image of $\Gamma_* : H^{i-2\nu}( Z,\Lambda(-\nu)) \to H^{i}(X_{\bar
	K},\Lambda)$ coincides with the image of $\Gamma_*\circ L_*$ and it follows that the containment 
${\widetilde \coniveau}^\nu H^i(X_{\bar K},\Lambda)\subseteq \coniveau'^\nu H^i(X_{\bar
	K},\Lambda)$ is an equality.
Now, in case $(\nu,i) = (n,2n)$, 
  the class of $Z\times_{\bar{K}} Z$ induces an isomorphism
 $H_0(Z,\Lambda) \stackrel{\simeq}{\to} H^0(Z,\Lambda)$, while in case
 $(\nu,i) = (n-1,2n-1)$,
using the identification of
torsion line bundles and \'etale covers, the universal line bundle
$\mathcal{L}$ on $\operatorname{Pic}^0_{Z/\bar K} \times_{\bar K} Z$ induces
natural identifications
$H_1(\operatorname{Pic}^0_{Z/\bar K},\mathbb Z_\ell) =
H^{2g-1}(\operatorname{Pic}^0_{Z/\bar K},\mathbb Z_\ell(g))=
T_\ell \operatorname{Pic}^0_{Z/\bar K} = H^1(Z,\mathbb
Z_\ell(1))$.
                	In case $\Lambda = \rat$ or $\rat_{\ell}$, a correspondence $L$ as above exists with $Z'=Z$ for all $Z$ and all $(\nu,i)$  provided Grothendieck's standard conjecture holds. 
 	\medskip
 	
 	We now turn to the containment ${\coniveau'}^\nu H^i(X_{\bar K},\Lambda)\subseteq \coniveau^\nu H^i(X_{\bar K},\Lambda)$. Let $\Gamma \in \operatorname{CH}_{d_X-\nu}(Z\times_{\bar K}X_{\bar K})$ be a correspondence with  $Z$  a smooth projective variety over $\bar K$. By refined intersection, the image of $\Gamma_*$ is supported on the closed subscheme $\bar Z := p_{X_{\bar K}}(\Gamma)$ of dimension $\leq d_X-\nu$, where $p_{X_{\bar K}} : \Gamma \to X_{\bar K}$ is the natural projection. In other words, the composition 
        \begin{equation}
          \label{eq:N'N}
 \xymatrix{H^{i-2\nu}( Z,\Lambda(-\nu)) \ar[r]^{\qquad \Gamma_*} & H^{i}(X_{\bar
 	K},\Lambda) \ar[r] & H^{i}(X_{\bar K}\smallsetminus \bar Z,\Lambda)
 }
        \end{equation}
  vanishes,
thereby  giving the asserted containment.

  For the statement of equality when $\Lambda$ denotes $\rat$ or $\mathbb Q_\ell$ and when $K$ is perfect,
 	since Jannsen \cite{jannsenseattle} only asserts this for fields of characteristic $0$,
 as de Jong's 
 results were not available at the time,
 here we
 reproduce the argument of \cite[p.265--6]{jannsenseattle} to show how the
 argument can be extended to fields of positive characteristic. 
 Consider a closed embedding $\iota: Z\hookrightarrow X_{\bar K}$ with $\dim Z = d_X-\nu$ and use the theory of alterations to produce a diagram
  $$\xymatrix{f: Z' \ar@{-{>>}}[r]^\pi & Z \ar@{^{(}->}[r]^\iota&X_{\bar K}}$$
  with $ Z'$ smooth of pure dimension $d_X-\nu$. We get a  commutative diagram (using $\ell$-adic
                homology)
  $$
  \xymatrix{
   H^{i-2\nu}(Z',\rat_\ell(-\nu)) \ar[r] \ar[d]_{\simeq} \ar@/^2pc/[rr]^{f_*}&
   H^i_{Z}(X_{\bar K},\rat_\ell) \ar[r] \ar[d] & H^i(X_{\bar K},\rat_\ell)
   \ar[d]_{\simeq}\\
   H_{2d_X-i}( Z',\rat_\ell(d_X)) \ar[r]^{\pi_*}& H_{2d_X-i}(Z,\rat_\ell(d_X))
   \ar[r]^{\iota_*}& H_{2d_X-i}(X_{\bar K},\rat_\ell(d_X))
  }
  $$
  and we then argue using weights (see \cite[\S 6]{jannsenthesis}).
  To utilize weights, we must assume that $K$ is finitely generated over its
  prime field\,; however, since all the varieties are of finite type, they are all
  defined over a base field $K'\subseteq \bar K$ that is finitely generated over its
  prime field, and we may work over $K'$, and then base change to $\bar K$.  In other
  words, we may assume
that $K$ is finitely generated over its prime
  field and that $X$, $Z$, $\Gamma$ and $\pi$ are defined over $K$.
  Now, since $H^i(X_{\bar K},\mathbb Q_\ell)$ is pure of weight $i$, the image of
  $\iota_*$ equals the image of $W_iH_{2d_X-i}(Z_{\bar K},\mathbb Q_\ell(d_X))$.  On
  the other hand, it is shown in \cite[Rem.~7.7]{jannsenthesis} that  $H_{2d_X-i}(
  Z'_{\bar K},\mathbb Q_\ell(d_X))$ surjects onto this space \emph{via}~$\pi_*$.
 \end{proof}

 \subsection{$p$-adic coniveau filtrations}\label{SS:pcon}
If $\characteristic(K)=p>0$ and if $K$ is perfect, we also allow $\Lambda$ to variously denote $\ww_n(K)$, $\ww(K)$ or $\kk(K)$, in which case $H^\bullet(-,\Lambda)$ denotes a crystalline cohomology group.  (See \S \ref{S:p-adicConv} for our notations concerning $p$-adic cohomology theories in characteristic $p$.)
One defines 
$$
\coniveau^\nu H^i(X_{\bar K},\Lambda) \ \ \text{ and } \ \ \coniveau'^\nu H^i(X_{\bar K},\Lambda)
$$ 
exactly as in~\eqref{eq:N} (note that in this setting \eqref{eq:Nker} is not well-behaved) and~\eqref{eq:N'}, respectively.  Since we will only need homology for smooth projective varieties $X$, we simply define $H_i(X,\Lambda) = H^{2d_X-i}(X,\Lambda(d_X))$.  With this notation we then define the correspondence niveau filtration $\widetilde{\coniveau}$ exactly as in~\eqref{eq:Ntilde}.

In contrast to the crystalline cohomology groups, the groups $H^i(X_{\bar K},\rat_p)$ no longer have a useful theory of weights or Poincar\'e duality.  Since $H^{i}(X_{\bar K},\integ_p)$ can be recovered as $(H^{i}(X/\ww))^F$, the suitably-defined $F$-invariants of $H^{i}(X/\ww)$ \cite[\S I.3]{gros85}, we simply \emph{define} the $p$-adic coniveau filtrations $\mathcal N = \widetilde{\coniveau}, \coniveau, \coniveau'$ by
\begin{align*}
  {\mathcal N}^{\nu} H^{i}(X_{\bar K},\integ_p) &= ({\mathcal N}^{\nu}
  H^{i}(X/\ww))^F\\
  {\mathcal N}^{\nu} H^{i}(X_{\bar K},\integ_p/p^r) &= ({\mathcal N}^{\nu}
  H^{i}(X/\ww_r))^F \\
  {\mathcal N}^\nu H^i(X_{\bar K}, \rat_p/\integ_p) &= \varinjlim {\mathcal N}^\nu H^i(X_{\bar K}, \integ_p/p^r).
\end{align*}
Then Proposition~\ref{P:Jannsen} holds in this context, too\,:

\begin{projannsenbis}
    Let $X$ be a smooth projective variety over a perfect field  $K$. 
  Suppose that $\Lambda$ is one of $\integ/l^r\integ$, $\integ_l$, 
  $\rat_l$, or $\mathbb Q_l/\mathbb Z_l$, or that $\characteristic(K)>p$ and $\Lambda$ is one of $\ww_r(K)$, $\ww(K)$ or $\kk(K)$, 
  or that $K = \cx$ and $\Lambda$ is one of $\mathbb Z$, $\mathbb Q$ or $\mathbb Q/\mathbb Z$.
  Then there are natural inclusions
  $$
  {\widetilde \coniveau}^\nu H^i(X_{\bar K},\Lambda)\subseteq   {\coniveau'}^\nu H^i(X_{\bar K},\Lambda)\subseteq \coniveau^\nu H^i(X_{\bar
   K},\Lambda).
  $$
    In case $(\nu,i) = (n,2n)$ or $(n-1,2n-1)$,
   the first inclusion is an equality.
  If $\Lambda$ is a field of characteristic zero, then the second inclusion is an equality, and if in addition Grothendieck's Lefschetz standard conjecture holds, then the first inclusion is an equality.
\end{projannsenbis}

\begin{proof}
  It only remains to prove the assertions for the various $p$-adic coefficients in characteristic $p>0$. If $\Lambda = \ww_n(K)$, $\ww(K)$ or $\kk(K)$, the argument is identical\,; the key point in the case $\Lambda = \kk(K)$ is that like \'etale cohomology, rigid cohomology has a good theory of weights and Poincar\'e duality. (The second map in~\eqref{eq:N'N} simply needs to be replaced with $H^i(X,\Lambda) \to H^i(X,\Lambda)/H^i_Z(X,\Lambda)$, which makes sense in crystalline cohomology.) The results for $\Lambda = \integ_p/p^n$, $\integ_p$, $\rat_p$ or $\rat_p/\integ_p$ follow by taking $F$-invariants.
\end{proof}

 \section{The image of the $l$-adic Bloch map and the coniveau filtration}\label{S:Proof-P-3.1}

Again, for brevity, in this section, the ring of coefficients $\Lambda$ denotes either $\integ_l$, $\rat_l$, or $\mathbb Q_l/\mathbb Z_l$.     
 The subscript $\tau_\Lambda$ on $ H^\bullet(-,\Lambda)_{\tau_\Lambda}$  indicates that when $\Lambda =\mathbb Z_l$, we take the quotient by the torsion subgroup.  
For each of these $\Lambda$ and for $M= \operatorname{CH}^n(X_{\bar K})$ or  $\operatorname{A}^n(X_{\bar K})$, we have the corresponding groups $M_\Lambda$\,:   $M_{\mathbb Z_l} := T_l M$, 
$M_{\mathbb Q_l} := V_l M$ and $M_{\mathbb Q_l/\mathbb Z_l}:= M[l^\infty]$.

\medskip
  
  We now consider the $l$-adic Bloch map constructed by Suwa~\cite{Suwa} in the case $l \neq \characteristic(K)$ and by Gros--Suwa~\cite{grossuwaAJ} in the case $l = \characteristic(K)$\,; see the appendix for a review of these maps, especially  \S \ref{SS:DSuwa} and  \S\ref{SS:pbloch}.
  Our main result  in this section extends a result of Suwa~\cite[Prop.~5.2]{Suwa}, originally stated for $\Lambda = \rat_\ell$ with $\ell \neq \characteristic(K)$, and gives a preliminary description of the image of the $l$-adic Bloch map 
  restricted to
  the Tate module of algebraically trivial cycles in terms of the correspondence coniveau filtration~\eqref{eq:N'}\,:

  \begin{pro}[{\cite[Prop.~5.2]{Suwa}}]\label{P:ImBloch-Pre}
  	Let $X$ be a smooth projective variety over a perfect field $K$, and let $l$ be a prime number. The image of the composition
  	\begin{equation}\label{E:CoNiTh3'}
  	\xymatrix{
  		\operatorname{A}^n(X_{\bar K})_\Lambda \ar@{^(->}[r]&
  		\operatorname{CH}^n(X_{\bar K})_\Lambda  \ar[r]^<>(0.5){\lambda^n}&H^{2n-1}(X_{\bar
  			K},\Lambda(n))_{\tau_\Lambda}\\
  	}
  	\end{equation}
       	contains  ${\widetilde \coniveau}^{n-1} H^{2n-1}(X_{\bar K},\Lambda(n))_{\tau_\Lambda}$, with equality if $\Lambda=\mathbb Q_l/\mathbb Z_l$ or $\mathbb Q_l$.  
   	Recall that the subscript $\tau_\Lambda$    indicates that when $\Lambda =\mathbb Z_l$, we take the quotient by the torsion subgroup.  
  \end{pro}
  
  \begin{rem}
Note that in the case where $\Lambda = \mathbb Q_l/\mathbb Z_l$, this implies that ${\widetilde \coniveau}^{n-1} H^{2n-1}(X_{\bar K},\mathbb Q_l/\mathbb Z_l(n))\subseteq  H^{2n-1}(X_{\bar K},\mathbb Z_l(n))\otimes _{\mathbb Z_l} \mathbb Q_l/\mathbb Z_l \subseteq  H^{2n-1}(X_{\bar K},\mathbb Q_l/\mathbb Z_l(n))$\,; see \S \ref{S:propertiesA}.    \end{rem}

  We split the proof of Proposition~\ref{P:ImBloch-Pre} in two, depending on whether the prime $l$ is invertible in~$K$.
  Key to the proof is the following lemma of Suwa \cite[Lem.~3.2]{Suwa} originally stated for primes $l \neq \characteristic(K)$\,; see  Remark~\ref{R:S3.2p} below, where we verify that Suwa's argument works even when $l = \characteristic(K)$.

  \begin{lem}[{\cite[Lem.~3.2]{Suwa}}]\label{L:Suwa-3.2}
  	Let $C$ be a smooth
  	projective irreducible curve  (resp.~abelian variety) over a field $K$ and let $\Gamma\in
  	\operatorname{CH}^n(C\times_K X)$. 
  	Let $\alpha\in \operatorname{A}^n(X_{\bar K})[l^\nu]$, 
  	and assume there exists  $\beta\in \operatorname{A}_0(C_{\bar
  		K})$ such that $\Gamma_*\beta = \alpha$.  Then there exists $\gamma\in
  	\operatorname{A}_0(C_{\bar K})[l^\mu]$ for some  $\mu\ge \nu$
  	such that $\lambda^n(\Gamma_*\gamma)= \lambda^n(\alpha)$. \qed
  \end{lem}

  \begin{exa}\label{E:Suwa-3.2-ell}
  	In Lemma \ref{L:Suwa-3.2}, it may be that  one must take $\mu >\nu$.  For instance, take $X=C=E$ to be an elliptic curve over $K$, let $\Gamma=\Gamma_f$ be the graph of the multiplication by $\ell$ map $f:C\to X$, and let $\alpha\in \operatorname{A}^1(X_{\bar K})[\ell]$.  Then there does not exist $\gamma\in
  	\operatorname{A}^1(C_{\bar K})[\ell]$ 
  	such that $\lambda^1(\Gamma_*\gamma)= \lambda^1(\alpha)$.
  \end{exa}

\subsection{The image of the $\ell$-adic Bloch map} Here we review the argument of Suwa~\cite{Suwa} yielding the proof of Proposition~\ref{P:ImBloch-Pre} in the case $\Lambda = \rat_{\ell}$ and extend it to other rings of coefficients.

\begin{proof}[Proof of Proposition~\ref{P:ImBloch-Pre}, prime-to-$\characteristic(K)$]
    Suppose $l =\ell\not = \characteristic(K)$.  
  	 Consider the diagram
  \begin{equation}\label{E:con-com-diag}
  \xymatrix{
   \displaystyle \bigoplus_{\Gamma:Z\vdash X_{\bar K}}
   \operatorname{A}_0(Z)_\Lambda  \ar@{->}[d]^{\Gamma_*} \ar[r]^<>(0.5)= &
   \displaystyle \bigoplus_{\Gamma:Z\vdash X_{\bar K}}
   \operatorname{CH}_0(Z)_\Lambda \ar[d]^{\Gamma_*}
   \ar[r]^<>(0.5){\lambda_0}_{\simeq\ \ }&\displaystyle\bigoplus_{\Gamma:Z\vdash X_{\bar K}} H_1(Z,\Lambda)_{\tau_\Lambda} \ar[d]^{\Gamma_*}\\
   \operatorname{A}^n(X_{\bar K})_\Lambda \ar@{^(->}[r]&
   \operatorname{CH}^n(X_{\bar K})_\Lambda \ar[r]^<>(0.5){\lambda^n}&H^{2n-1}(X_{\bar
    K},\Lambda (n))_{\tau_\Lambda} \\
  }
  \end{equation}
  where the direct sums are over all smooth projective varieties $Z$ over $\bar K$ and all correspondences $\Gamma \in \operatorname{CH}^{n}(Z\times_{\bar K}X_{\bar K})$.
  The image of the right vertical arrow  is by definition
  ${\widetilde \coniveau}^{n-1}H^{2n-1}(X_{\bar K},\Lambda  (n))_{\tau_\Lambda} $.
   We conclude from commutativity of the diagram that the image of the bottom row of~\eqref{E:con-com-diag} contains
  ${{\widetilde \coniveau}}^{n-1}H^{2n-1}(X_{\bar K},\Lambda  (n))$.
  
We now show that the inclusion is an equality in case $\Lambda = \rat_{\ell}/\integ_{\ell}$.  
By Lemma~\ref{L:Suwa-3.2}, for any $\alpha\in
  \operatorname{A}^n(X_{\bar K})[\ell^\infty]$, there exists an element $ \gamma\in
  \bigoplus_{\Gamma:Z\vdash X} \operatorname{A}_0(Z_{\bar K})[\ell^\infty]$ such that
  $\lambda^n(\Gamma_*\gamma)=\lambda^n(\alpha)$.
It readily follows from a diagram chase that the image of the bottom row of~\eqref{E:con-com-diag} is contained in   ${\widetilde \coniveau}^{n-1}H^{2n-1}(X_{\bar K},\mathbb Q_\ell/\mathbb Z_\ell(n))$.

Finally we show equality in the case $\Lambda = \mathbb Q_\ell$. 
  Taking Tate modules in~\eqref{E:con-com-diag} with $\mathbb Q_\ell/\mathbb Z_\ell$-coefficients, and using the previous case, we have that the image of the bottom row of~\eqref{E:con-com-diag} with $\mathbb Z_\ell$-coefficients is
   $T_\ell  {{\widetilde \coniveau}}^{n-1} H^{2n-1}(X,\rat_\ell /\integ_\ell (n))$. 
Since $H^{1}( Z,\mathbb Q_\ell/\mathbb Z_\ell(1))$ is a divisible group for any smooth projective variety $Z$ over $\bar K$,
and since an increasing chain of divisible abelian $\ell$-torsion subgroups of a finite corank abelian $\ell$-torsion group is stationary \cite[Lem.~1.2]{Suwa},  then by adding connected components to $Z$, one can conclude
  there exist a smooth projective curve $Z$ over $\bar K$ and a correspondence $\Gamma \in \operatorname{CH}^n(Z\times_{\bar K}X_{\bar K})$ such that 
 \begin{align*}
 {{\widetilde \coniveau}}^{n-1} H^{2n-1}(X_{\bar K},\mathbb Z_\ell(n)) &= 
 \operatorname{im} \Big( \Gamma_*: H^{1}( Z,\mathbb Z_\ell(1)) \to H^{2n-1}(X_{\bar
  K},\mathbb Z_\ell (n))\Big)\\
 {{\widetilde \coniveau}}^{n-1} H^{2n-1}(X_{\bar K},\mathbb Q_\ell/\mathbb Z_\ell(n)) &= 
 \operatorname{im} \Big( \Gamma_*: H^{1}( Z,\mathbb Q_\ell/\mathbb Z_\ell(1)) \to H^{2n-1}(X_{\bar
  K},\mathbb Q_\ell/\mathbb Z_\ell(n))\Big).
\end{align*}
 Then the image of 
 the bottom row of~\eqref{E:con-com-diag} with $\mathbb Z_\ell$-coefficients is  
\begin{align*}
T_\ell  {{\widetilde \coniveau}}^{n-1} H^{2n-1}(X,\rat_\ell /\integ_\ell (n)) &= T_\ell \Gamma_* H_1(Z,
\rat_\ell/\integ_\ell) \\
&\supseteq \Gamma_* T_\ell H_1(Z,\rat_\ell/\integ_\ell) \\
&= \Gamma_* H_1(Z,\mathbb Z_\ell) \\
&= {{\widetilde \coniveau}}^{n-1} H^{2n-1}(X_{\bar K},\mathbb Z_\ell(n))_\tau .
\end{align*}
From Lemma \ref{L:push-pull-QQ},
the inclusion above has torsion cokernel.  This gives the assertion with $\mathbb Q_\ell$-coefficients.
  \end{proof}

\subsection{The image of the $p$-adic Bloch map}

Now let $K$ be a perfect field of characteristic $p>0$.  Here we secure Proposition~\ref{P:ImBloch-Pre} for $p$-adic cohomology groups.

\begin{rem}\label{R:S3.2p}
Lemma \ref{L:Suwa-3.2} (\cite[Lem.~3.2]{Suwa}) holds as well for $p$-power
torsion.  The proof is essentially identical.  
Briefly, suppose $K$ is finite\,; then $\A^1(C_{\bar K}) = \pic^0_{C/K}(\bar
K)$ is a torsion group.  Write the order of $\beta$ as $p^m M$ with
$M$ relatively prime to $p$, and choose $N$ with $NM \equiv 1 \bmod
p^\nu$.  Then $\gamma := NM\beta$ has order $p^m$, and $\Gamma_*\gamma
= NM\alpha = \alpha$.
To account for arbitrary $K$, we may assume that $K$ is the perfection
of a field $K_0$ which is finitely generated over $\ff_p$, and then spread
out the data $X$, $\Gamma$, $C$, $\alpha$ and $\beta$ to an irreducible scheme $S$
of finite type over $\ff_p$, with function field $K_0$\,; let $k$ be the
algebraic closure of $\ff_p$ in $K_0$.  Consider the
finite flat group scheme $G := \pic^0_{C/S}[p^\nu] \to S$, and let $Q
\in S$ be a closed point.  Because the residue field of $Q$ is finite,
possibly after base-change by a finite extension of $k$, there exists
some $\gamma_Q \in G_Q(k) = \A^1(C_{Q})$ such that
$\Gamma_{Q,*}\gamma_Q = \alpha_Q$.  Possibly after replacing $S$ with
an open neighborhood of $Q$, there exists a finite surjective $T \to
S$ such that $G_T \to T$ admits a section $\gamma\in G(T)$ passing through
some pre-image of $\gamma_Q$ in $G_Q \times_S T$ \cite[14.5.10]{egaIV3}\,; 
the generic fiber of $\gamma$ is the sought-for class.
\end{rem}

\begin{proof}[Proof of Proposition~\ref{P:ImBloch-Pre}, $\characteristic(K)$-torsion]
 The assertion of the inclusion in Proposition~\ref{P:ImBloch-Pre} with $\Lambda = \mathbb Q_p/\mathbb Z_p$ (for arbitrary $n$) follows immediately from diagram~\eqref{E:con-com-diag}.  For the opposite inclusion, the argument is identical.
For the inclusion of Proposition~\ref{P:ImBloch-Pre} with $\mathbb Z_p$-coefficients, one uses diagram~\eqref{E:con-com-diag} with $\Lambda= \mathbb W$, and the fact that taking $F$-invariants commutes with push-forward \cite[Cor.~I.3.2.7]{gros85}.  Tensoring with $\mathbb Q_p$ gives the inclusion with $\mathbb Q_p$-coefficients.  One then argues identically that the inclusion with $\mathbb Z_p$-coefficients has torsion quotient.   Indeed, taking Tate modules of 
\eqref{E:con-com-diag} with $\mathbb Q_p/\mathbb Z_p$-coefficients, one sees that the image of the bottom row of~\eqref{E:con-com-diag} with $\mathbb Z_p$-coefficients is ${{\widetilde \coniveau}}^{n-1} H^{2n-1}(X_{\bar K},\rat_p/\integ_p(n))$.
 As before, since $H^{1}( Z,\mathbb Q_p/\mathbb Z_p(1))$ is a divisible group for any smooth projective variety $Z$ over $\bar K$,
	and since an increasing chain of divisible abelian $p$-torsion subgroups of a finite corank abelian $p$-torsion group is stationary \cite[Lem.~1.2]{Suwa},  then by adding connected components to $Z$, one can conclude 
 there is a 
smooth surface $Z/\bar K$, possibly disconnected but of finite type, equipped
with a morphism $f:Z \to X_{\bar K}$ such that $f_* H^1(Z,\rat_p/\integ_p(1)) =
{{\widetilde \coniveau}}^{n-1} H^{2n-1}(X_{\bar K},\rat_p/\integ_p(n))$ and $f_* H^1(Z/\ww(1)) =
{{\widetilde \coniveau}}^{n-1} H^{2n-1}(X_{\bar K}/\ww(n))$.  We then have 
\begin{align*}
T_p {{\widetilde \coniveau}}^{n-1} H^{2n-1}(X_{\bar K},\rat_p/\integ_p(n)) &= T_p f_* H^1(Z,
\rat_p/\integ_p(1)) \\
&\supseteq f_* T_p H^1(Z,\rat_p/\integ_p(1)) \\
&= f_* H^1(Z,\integ_p(1)) \\
&= f_* (H^1(Z/\ww(1))^F) \\
&= (f_* H^1(Z/\ww(1)))^F\text{ by \cite[Cor.~I.3.2.7]{gros85}}\\
&= ({{\widetilde \coniveau}}^{n-1} H^{2n-1}(X_{\bar K}/\ww(n)))^F = {{\widetilde \coniveau}}^{n-1} H^{2n-1}(X_{\bar K},\integ_p(n))_\tau.
\end{align*}
Since the inclusion has torsion cokernel (see Lemma \ref{L:push-pull-QQ} and Remark~\ref{R:L:p-p-QQp}), we are done.
\end{proof}

\section{Decomposition of the diagonal, algebraic representatives, and miniversal cycles}

In \cite{ACMVdiag} we consider decomposition of the diagonal and algebraic representatives in detail.  These ideas also come into play in the proofs of Theorems \ref{T:main}, \ref{T:co=H} and \ref{T:model}, and so we review these notions in this section.  We refer the reader to~\cite{ACMVdiag} for more details.

 \subsection{Decomposition of the diagonal}\label{SS:dec} 

  The aim of this subsection is to fix the notation for decomposition of the diagonal, and to recall the existence of decompositions of the diagonal for various flavors of rational varieties.

 \smallskip  
Let $K$ be a field and let $\ell$ be a prime not equal to $\operatorname{char}K$.
  Let $R$ be a commutative ring, let $X$ be a smooth projective variety over a field $K$ and let $W_1$ and $W_2$ be two closed subschemes of $X$ not containing any component of $X$ .
 A cycle class $Z \in \operatorname{CH}^{d_X}(X\times_KX)\otimes_{\integ} R$
  is said to admit a \emph{decomposition of  type $(W_1,W_2)$} 
 if
  \begin{equation*}\label{E:DefDcp}
  Z =Z_1+Z_2\in \operatorname{CH}^{d_X}(X\times _K X)\otimes_\integ R ,
  \end{equation*}
  where $Z_1\in \operatorname{CH}^{d_X}(X\times _K X)\otimes_\integ R$ is supported on  $W_1\times_KX$ and $Z_2\in \operatorname{CH}^{d_X}(X\times _K X)\otimes_\integ R$ is supported on $X\times_K W_2$. 
If $Z = \Delta_X$ and if one can choose $W_2$ with $\dim W_2 = 0$, we simply say that $\Delta_X\in \operatorname{CH}^d(X\times_KX)\otimes R$ has a Chow decomposition.
\medskip

  \begin{dfn}[Universal support of $\operatorname{CH}_0\otimes R$]
  	\label{D:unisupport}
  	Let $X$ be a smooth projective variety over a field~$K$. We say that $\operatorname{CH}_0(X)\otimes R$ is \emph{universally supported} in dimension $d$ if there exists a closed subscheme $W_2\subseteq X$ of dimension $\leq d$ such that $\operatorname{CH}_0(X)\otimes R$ is universally supported on $W_2$, \emph{i.e.}, if the push-forward map $\chow_0((W_2)_L)\otimes_{\integ} R \to \chow_0(X_L)\otimes_{\integ} R$ is surjective for all field extensions $L/K$.
  \end{dfn}

  The following proposition is classical and goes back to Bloch and Srinivas~\cite{BS}, and relates decomposition of the diagonal to the universal support of $\operatorname{CH}_0$\,:

  \begin{pro}[{Bloch--Srinivas \cite{BS}}]\label{P:BS}
The diagonal $\Delta_X \in \operatorname{CH}^{d_X}(X\times_KX)\otimes_{\integ} R$ of a smooth projective variety $X$ over $K$ admits a decomposition of type $(W_1,W_2)$
  if and only if
 $\chow_0(X)\otimes_{\integ} R$ is universally supported on $W_2$.
 In particular, taking $R=\integ$, the diagonal $\Delta_X \in \operatorname{CH}^{d_X}(X\times_K X)$ of a smooth projective,
  stably rational, variety $X$ admits an decomposition of type $(W_1, P)$ for any choice of
   $K$-point $P\in X(K)$.
   \qed
    \end{pro}

\begin{rem}\label{R:nishimura} The existence of a $K$-point on a stably rational variety over $K$ is
   ensured by the Lang--Nishimura theorem \cite{nishimura55}\,; see
   \cite[Prop.~A.6]{reichsteinyoussin} for a modern treatment.
 \end{rem}

\begin{rem}[Varieties admitting decompositions of the diagonal] \label{R:stablyrat}
	Let $X$ and $Y$ be smooth projective varieties over $K$ of respective dimension $d_Y \leq d_X$. If $X$ and $Y$ are \emph{stably rationally equivalent}, \emph{i.e.}, if there exists a nonnegative integer $n$ such that $Y\times_K \mathbb{P}_K^n$ is birational to $X\times_K \mathbb{P}_K^{n+d_X-d_Y}$, then $\operatorname{CH}_0(X)$ is universally supported in dimension $d_Y$. In particular, if $X$ is stably rational, then $\chow_0(X)$ is universally supported in dimension $0$, and in fact on a point by the Lang--Nishimura theorem.  Moreover, if $X$ is only assumed to be geometrically  rationally chain connected, then $\chow_0(X)\otimes_{\mathbb Z}{\mathbb Q}$ is universally supported in dimension $0$ (see \emph{e.g.}~\cite[Rem.~2.8]{ACMVdiag}).  
\end{rem}

\begin{nota}[Decomposition of the diagonal and alterations]  \label{notations}
	Let $X$ be a pure-dimensional smooth projective variety over a perfect field $K$ of characteristic  exponent $p$.
Suppose we have a cycle class $Z_1 + Z_2$ in $\chow^{d_X}(X\times_K X)$ with $Z_1$ supported on $W_1\times_K X$ and $Z_2$ supported on $X\times_K W_2$ with  $W_1$ and $W_2$ two closed subschemes of $X$ not containing any component of $X$ such that $\dim W_1 \le n_1$ and $\dim W_2\le n_2$. By \cite{temkin17}, there exist alterations $\widetilde{W}_1 \to W_1$ and $\widetilde{W}_2 \to W_2$  of degree some power of $p$ such that $\widetilde{W}_1$ and $\widetilde{W}_2$ are smooth projective over $K$. The cycle classes $Z_1$ and $Z_2$, seen as self-correspondences on $X$, factor up to inverting $p$ through $\widetilde{W}_1$ and $\widetilde{W}_2$, respectively.
Precisely, there exists a nonnegative integer $e$ such that
 $$p^eZ_1  = r_1\circ s_1 \qquad \text{and} \qquad p^eZ_2 = r_2 \circ s_2 \qquad  \text{in} \ \chow^{d_X}(X\times_K X)$$
  for some $s_1 \in \chow^{d_X}(X\times_K \widetilde{W}_1)$,  $r_1 \in
  \chow_{d_X}(\widetilde{W}_1\times_K X)$,  $s_2 \in
  \chow_{d_X}(X\times_K \widetilde{W}_2)$ and $r_2 \in
  \chow^{d_X}(\widetilde{W}_2\times_K X)$. Note that by replacing each
  component of $\widetilde{W}_1$ and $\widetilde{W}_2$ with a product
  with projective space of an appropriate dimension, we may assume
  that  $\widetilde{W}_1$ and $\widetilde{W}_2$ are of pure dimension
  $n_1$ and $n_2$, respectively. We refer to \cite[\S3.1]{ACMVdiag}
  for more details. 
Since resolution of singularities exists for threefolds over a perfect field~\cite{CPRes2}, if each $\dim W_i \le 3$, then we may take
$e=0$. 
\end{nota}

 \subsection{Surjective regular homomorphisms and algebraic representatives}\label{SS:AlgRep} 
The aim of this subsection is to fix notation for algebraic representatives.  
 We start by reviewing the definition of an algebraic representative
 (\emph{i.e.},  \cite[Def.~1.6.1]{murre83} or \cite[2.5]{samuelequivalence}).
 Let $X$ be a smooth projective variety over a perfect field $K$ and let $n$ be a nonnegative integer.
 For a smooth separated scheme $T$ of finite type over $K$, we define $\mathscr{A}_{X/K}^n(T)$ to be the abelian group consisting of those cycle classes $Z\in \operatorname{CH}^n(T\times _K X)$ such that for every $t\in T(\bar K)$ the Gysin fiber $Z_t$ is algebraically trivial. For $Z\in\mathscr{A}_{X_{\bar K}/\bar K}^n(T)$ denote by $w_{Z}:T(\bar K)\to \operatorname{A}^n(X_{\bar K})$ the map defined by $w_Z(t)=Z_t$.  
 
 Given an abelian variety $A/\bar K$, a \emph{regular homomorphism} (in codimension $n$)  $$\xymatrix{\phi:\operatorname{A}^n(X_{\bar K})\ar[r] & A(\bar K)}$$ is a homomorphism of groups such that for every $Z\in\mathscr{A}_{X/K}^n(T)$ the composition
 $$
 \xymatrix{
  T(\bar K) \ar[r]^{w_Z}& \operatorname{A}^n(X_{\bar K}) \ar[r]^\phi& A(\bar K)
 }
 $$
 is induced by a morphism of varieties $\psi_Z: T_{\bar K} \to A$.  An \emph{algebraic representative} (in codimension~$n$) is a regular homomorphism 
 $$
 \phi^n_{X_{\bar K}}:\operatorname{A}^n(X_{\bar K})\longrightarrow \operatorname{Ab}^n_{X_{\bar K}/\bar K}(\bar K)
 $$ 
 that is initial among all regular homomorphisms. 
   For $n=1$, an algebraic representative is given by $(\operatorname{Pic}^0_{X_{\bar K}/\bar K})_{\operatorname{red}}$ together with the Abel--Jacobi map.  For $n=d_X$, an algebraic representative is given by the Albanese variety and the Albanese map.
   For $n=2$, it is a result of Murre \cite[Thm.~A]{murre83} that there exists an algebraic representative for $X_{\bar K}$, which in the case $K=\mathbb C$ is the algebraic intermediate Jacobian $J^3_a(X)$\,; \emph{i.e.}, the image of the  Abel--Jacobi map restricted to algebraically trivial cycle classes.

 The main result of \cite{ACMVdcg} (see also \cite{ACMVfunctor}) is that if there exists an algebraic representative  $\phi^n_{X_{\bar K}}:\operatorname{A}^n(X_{\bar K})\to \operatorname{Ab}^n_{X_{\bar K}/\bar K}(\bar K)$, then  $\operatorname{Ab}^n_{X_{\bar K}/\bar K}$ admits a canonical model over $K$, denoted $ \operatorname{Ab}^n_{X_{K}/K}$, such that $\phi^n_{X_{\bar K}}$ is $\operatorname{Gal}(\bar K/K)$-equivariant and such  that for any  $Z\in\mathscr{A}_{X_{K}/K}^n(T)$ the morphism $\psi_{Z_{\bar K}}:T_{\bar K}\to A_{\bar K}$ descends to a morphism $\psi_Z:T\to A$ of $K$-schemes.  In particular, the algebraic representative $\operatorname{Ab}^2_{X_{\bar K}/\bar K}$ of \cite{murre83} admits a canonical model over $K$, denoted $ \operatorname{Ab}^2_{X_{K}/K}$.  In the case $K\subseteq \mathbb C$, the abelian variety $\operatorname{Ab}^2_{X/K}$  is the distinguished model $J^3_{a,X/K}$ of the algebraic intermediate Jacobian, as defined in  \cite{ACMVdmij}. \medskip

We include the following lemma for clarity\,; we also note that it is clear from the definitions that an algebraic representative $ \phi^n_{X_{\bar K}}:\operatorname{A}^n(X_{\bar K})\to \operatorname{Ab}^n_{X_{\bar K}/\bar K}(\bar K)$ is a \emph{surjective} regular homomorphism.

\begin{lem}\label{L:SRHom-facts}
Let $\phi:\operatorname{A}^n(X_{\bar K})\to A(\bar K)$ be a surjective regular homomorphism.  

\begin{enumerate}
\item Let $l$ be prime. Then\,:
\begin{enumerate}
\item   $\phi[l^\infty]: \operatorname{A}^n(X_{\bar
  K})[l^\infty]\to A[l^\infty]$ is surjective. 
\item The following are equivalent\,:
\begin{enumerate}
\item $\phi[l^\infty]: \operatorname{A}^n(X_{\bar
  K})[l^\infty]\to A[l^\infty]$ is an isomorphism.
\item $\phi[l^\infty]: \operatorname{A}^n(X_{\bar
  K})[l^\infty]\to A[l^\infty]$ is an inclusion.
\item $\phi[l^\nu]: \operatorname{A}^n(X_{\bar K})[l^\nu]\to
  A[l^\nu]$ is an inclusion for all natural numbers $\nu$.
  \item $\phi[l]: \operatorname{A}^n(X_{\bar K})[l] \to
    A[l]$
    is an inclusion.
\end{enumerate}

\item If any of the equivalent conditions in (1)(b) hold, then $T_l \phi: T_l \operatorname{A}^n(X_{\bar K})\to T_l A$ is an isomorphism.
\end{enumerate}

\item There exists a natural number $e$ (independent of $l$) such
  that for all natural numbers $\nu$ the image of the map
  $\phi[l^{\nu+e}]: \operatorname{A}^n(X_{\bar K})[l^{\nu+e}]\to
  A[l^{\nu+e}]$ contains $A[l^\nu]\subseteq A[l^{\nu+e}]$.

\item For all but finitely many primes $l$ the map $\phi[l^\nu]:
  \operatorname{A}^n(X_{\bar K})[l^\nu]\to A[l^\nu]$ is surjective.
\end{enumerate}
\end{lem}

\begin{proof}
(1)(a) is \cite[Rem.~3.3]{ACMVdmij}.  (Even though it is only claimed
  there for $K$ of characteristic zero, the arguments of the
  references cited there are valid for arbitrary torsion in arbitrary
  characteristic.)  The equivalence of (1)(b)(i) and (1)(b)(ii) is obvious.   
To show the equivalence of (1)(b)(ii) and (1)(b)(iii) we argue as follows.  
For a group $G$, there is an inclusion $G[l^\nu]\subseteq
G[l^{\nu+1}]$, so that in our situation, we have
$\operatorname{A}^n(X_{\bar K})[l^\nu]\subseteq
\operatorname{A}^n(X_{\bar K})[l^\infty]$ and $A[l^\nu]\subseteq
A[l^\infty]$.  The equivalence of (1)(b)(ii) and (1)(b)(iii) then
follows by a diagram chase\,; the equivalence of (1)(b)(iii) and
(1)(b)(iv) is elementary.   One obtains (1)(c) by applying the Tate module to the isomorphism (1)(b)(i). Item (3) follows from Item (2), which in turn is  \cite[Rem.~3.3]{ACMVdmij}.
\end{proof}

\begin{rem}
It is worth noting that if there exists a regular homomorphism $\phi:\operatorname{A}^n(X_{\bar K})\to A(\bar K)$ such that $\phi[\ell^\infty]$ is injective, then there is an algebraic representative in codimension-$n$\,; this follows directly from Saito's criterion (\cite[Prop.~2.1]{murre83}).
\end{rem}

\subsection{Miniversal cycles and miniversal cycles of minimal degree}\label{SS:mini}
Let 
$X$ be a smooth projective variety over a field $K$ and let $\phi:\operatorname{A}^n(X_{\bar K}) \to A(\bar K)$ be a regular homomorphism. A \emph{miniversal cycle} for $\phi$ is a cycle $Z\in\mathscr{A}_{X_{K}/K}^n(A)$ such that the homomorphism $\psi_Z: A \to A$ is given by multiplication by $r$ for some natural number~$r$, which we call the \emph{degree} of the cycle. A miniversal cycle is called \emph{universal} if $\psi_Z : A \to A$ is given by the identity. In case $\phi$ is an algebraic representative for codimension-$n$ cycles on $X$, we call a universal cycle for $\phi$ a universal cycle in codimension-$n$ for~$X$.  Clearly there is a minimal $r$ such that there exists a miniversal cycle of degree $r$\,; taking linear combinations of miniversal cycles,  one can see this minimum is achieved by the GCD of all of the degrees of miniversal cycles.

If $K$ is algebraically closed, it is a classical and crucial fact~\cite[1.6.2 \& 1.6.3]{murre83} that a miniversal cycle exists  if and only if $\phi$ is surjective\,; this also holds without any restrictions on the field~$K$ by~\cite[Lem.~4.7]{ACMVfunctor}. In particular, since an algebraic representative is always a surjective regular homomorphism  \cite[Prop.~5.1]{ACMVfunctor}, it always admits a miniversal cycle. However, the existence of a universal cycle is restrictive.  For example, both over the complex numbers \cite{voisinUniv} and over a field of characteristic at least three \cite{ACMVdiag}, the standard desingularization of the very general double quartic solid with 7 nodes does not admit a universal cycle in codimension-$2$.

 \subsection{Decomposition of the diagonal and algebraic representatives}\label{SS:Murre}

We now recall a result due to \cite{murre83} and \cite{BS}\,:

\begin{pro}[{Murre \cite{murre83}, Bloch--Srinivas \cite{BS}}]\label{P:phi}
	Let $X$ be a smooth projective variety over a perfect field $K$ of characteristic exponent $p$.
	\begin{enumerate}
		\item If $\operatorname{char}(K)=0$, then
		\begin{equation}\label{E:Murre0Tors}
		\phi^2_{X_{\bar K}}[\ell^\infty] :  \operatorname{A}^2(X_{\bar K})[\ell^\infty]
		\longrightarrow \mathrm{Ab}^2_{X/K}[\ell^\infty](\bar K).
		\end{equation}
		is an isomorphism of $\operatorname{Gal}(K)$-modules for all prime numbers $\ell$.
		
		\item   Assume that the diagonal $\Delta_{X_{\bar K}}\in
		\operatorname{CH}^{d_X}(X_{\bar K}\times_{\bar K} X_{\bar K})\otimes \rat$  admits a
		decomposition
		of type $(W_1,W_2)$ with  $\dim W_2\le 1$. Then
		$$ \phi^2_{X_{\bar K}} : \operatorname{A}^2(X_{\bar K})
		\longrightarrow \mathrm{Ab}^2_{X/K}(\bar K)$$ is an isomorphism of $\operatorname{Gal}(K)$-modules.
		
		\item   Assume that the diagonal $N\Delta_{X_{\bar K}}\in
		\operatorname{CH}^{d_X}(X_{\bar K}\times_{\bar K} X_{\bar K})$  admits a
		decomposition
		of type $(W_1,W_2)$ with $\dim W_2\le 2$ for some positive integer $N$. Then
		 			$$V_l \phi^2_{X_{\bar K}} : V_l \operatorname{A}^2(X_{\bar K})
		\longrightarrow V_l \mathrm{Ab}^2_{X/K}$$ is an isomorphism of $\operatorname{Gal}(K)$-modules for all primes $l$, and 
		
		$$T_\ell \phi^2_{X_{\bar K}} : T_\ell \operatorname{A}^2(X_{\bar K})
		\longrightarrow T_\ell \mathrm{Ab}^2_{X/K}$$ is an isomorphism of $\operatorname{Gal}(K)$-modules for all prime numbers $\ell$ not dividing
		$Np$.
		
		\item  In the setting of (3), further assume that $p\ge 2$,  resolution of singularities holds in dimensions $< d_X$,  and $p\nmid N$. Then
		\[
		T_p \phi^2_{X_{\bar K}}: T_p \A^2(X_{\bar K}) \longrightarrow T_p \Ab^2_{X/K}
		\]
		is an isomorphism of $\gal(K)$-modules.
	\end{enumerate}
\end{pro}
 \begin{proof} 
 	First recall from \cite{ACMVdcg} that $\phi^2_{X_{\bar K}}$ is $\operatorname{Gal}(K)$-equivariant. Item (2) is \cite[Thm.~1(i)]{BS}, while 
Item (1) reduces via \cite{ACMVdcg} to the case $K=\cx$, which is covered by \cite[Thm.~10.3]{murre83}.

  We now prove Item (3) and assume that $\operatorname{char}(K) = p >0$.
With Notation~\ref{notations},
 we have a commutative diagram with composition of horizontal arrows being multiplication by $Np^e$\,:
 \begin{equation}\label{E:diagfact}
\xymatrix{ \operatorname{A}^2(X_{\bar K}) \ar[r]^{r_1^*\oplus r_2^*\qquad } \ar[d]^{\phi^2_X} & \operatorname{A}^1(\widetilde{W}_1) \oplus  \operatorname{A}^2(\widetilde{W}_2) \ar[r]^{\qquad s_1^*+s_2^*} \ar[d]^{\phi^1_{\widetilde{W}_1}\oplus \phi^2_{\widetilde{W}_2}}&
	\operatorname{A}^2(X_{\bar K}) \ar[d]^{\phi^2_X} \\
	\mathrm{Ab}^2_{X/K}(\bar K)  \ar[r] & \operatorname{Pic}^0_{\widetilde{W}_1}(\bar K) \oplus  \operatorname{Alb}_{\widetilde{W}_2}(\bar K) \ar[r] & 	\mathrm{Ab}^2_{X/K}(\bar K).
}
 \end{equation}
where the bottom horizontal arrows are the $\bar K$-homomorphisms induced by the universal property of algebraic representatives. (Note that since $\dim \widetilde{W}_2=2$ we have identified $\phi^2_{\widetilde{W}_2}$ with the Albanese map.) Since $\phi^1$ is an isomorphism and since the Albanese morphism is an isomorphism on torsion by Rojtman~\cite{bloch79,grossuwaAJ,milne82}, a simple diagram chase establishes that $\phi^2_X : \operatorname{A}^2(X_{\bar K})
\to \mathrm{Ab}^2_{X}(\bar K)$ is an isomorphism on prime-to-$Np^e$ torsion. It follows that $\phi^2_X$ is an isomorphism on $l$-primary torsion for all primes $l$ not dividing $Np^e$. The statement about $T_\ell \phi^2_{X_{\bar K}}$ then ensues by passing to the inverse limit. Alternately, since the middle vertical arrow is an isomorphism on torsion, it is an isomorphism on Tate modules. By applying $T_l$ to \eqref{E:diagfact}, and since $T_l\operatorname{A}^2(X_{\bar K})$ and $T_l	\mathrm{Ab}^2_{X/K}$ are finite free $\integ_l$-modules ($T_l \lambda^2 : T_l\operatorname{A}^2(X_{\bar K}) \hookrightarrow T_l	\mathrm{Ab}^2_{X/K}$ is injective, Proposition~\ref{P:M-9.2}), we directly see that $T_l\phi^2_{{X_{\bar K}}}$ is an isomorphism for $l\nmid Np$ and we also see, after tensoring with $\rat_l$ that $V_l\phi^2_{{X_{\bar K}}}$ is an isomorphism for all primes $l$.

For (4), it suffices to observe that, if $W_1$ and $W_2$ admit resolutions of singularities (which is the case if $d_X\leq 4$ by \cite{CPRes2}), then we may take $e=0$ above.
 \end{proof}
 
  \begin{rem}
By rigidity, the same results in Proposition~\ref{P:phi} hold if the separable closure $\bar K$ is replaced with the algebraic closure $\bar K^a$.  
\end{rem}
 
 \begin{pro}[{Decomposition of the diagonal and miniversal cycle classes}] 
 	\label{P:decmini}
 	Let $X$ be a smooth projective variety over a field $K$ that is either finite or algebraically closed,
 	and let $N$ be a natural number.
Assume that  $N\Delta_{X_{\bar K}}\in
\operatorname{CH}^{d_X}(X\times_{K} X)$  admits a
decomposition
of type $(W_1,W_2)$ with $\dim W_2\le 1$.  Then $\operatorname{Ab}^2_{X/K}$ admits a miniversal cycle of degree $p^eN$ for some nonnegative integer $e$ that may be chosen to be zero if $\dim X \leq 4$. 

In particular, if $\dim X\leq 4$ and if $\operatorname{CH}_0(X)$ is universally supported in dimension~1, then $\operatorname{Ab}^2_{X/K}$ admits a universal cycle.
\end{pro}

 \begin{proof} Similarly to \eqref{E:diagfact}, we have with Notation~\ref{notations} a commutative diagram with composition of horizontal arrows being multiplication by $Np^e$ (where, due to resolution of singularities in dimensions $<4$, $e$ can be chosen to be zero if $\dim X\leq 4$)\,:
 	\begin{equation*}
 	\xymatrix{ \operatorname{A}^2(X_{\bar K}) \ar[r]^{r_1^* } \ar[d]^{\phi^2_X} & \operatorname{A}^1(\widetilde{W}_1) \ar[r]^{ s_1^*} \ar[d]^{\phi^1_{\widetilde{W}_1}}&
 		\operatorname{A}^2(X_{\bar K}) \ar[d]^{\phi^2_X} \\
 		\mathrm{Ab}^2_{X/K}(\bar K)  \ar[r]^{r_1^* } & \operatorname{Pic}^0_{\widetilde{W}_1}(\bar K)  \ar[r]^{ s_1^*} & 	\mathrm{Ab}^2_{X/K}(\bar K),
 	}
 	\end{equation*}
 	where $r_1^* : \operatorname{Ab}^2_{X/K} \to (\operatorname{Pic}^0_{\widetilde
 		W_1/K})_{\operatorname{red}}$ and $s_1^* : (\operatorname{Pic}^0_{\widetilde
 		W_1/K})_{\operatorname{red}} \to  \operatorname{Ab}^2_{X/K}$ denote the $K$-homomorphisms induced by the correspondences $r_1$ and $s_1$.
 		Since we are assuming $K$ to be either finite or algebraic closed, the Abel--Jacobi map $\phi^1_{\widetilde{W}_1}$ admits 
 		a
 	universal divisor
 	$\widetilde {\mathcal D} \in
 	\mathscr A^1_{\widetilde W_1/K}((\operatorname{Pic}^0_{\widetilde
 		W_1/K})_{\operatorname{red}})$ (see, \emph{e.g.}, \cite[\S 7.1]{ACMVfunctor}), meaning that the
 	induced morphism $$\psi_{\widetilde{\mathcal D}} : \ (\operatorname{Pic}^0_{\widetilde
 		W_1/K})_{\operatorname{red}}\to (\operatorname{Pic}^0_{\widetilde
 		W_1/K})_{\operatorname{red}}$$ is the identity. 
 	It is then clear that the homomorphism associated to
 	the cycle class
 	$$
 	Z:=s_1^*\circ
 	\widetilde {\mathcal
 		D} \circ r_1^* \quad \in \mathscr A^n_{X/K}(\operatorname{Ab}^n_{X/K}).
 	$$
 is given by
 	$Np^e\operatorname{Id}_{\operatorname{Ab}^2_{X/K}}$.
\end{proof}

\section{Miniversal cycles and the image of the second $l$-adic Bloch map}\label{SS:proofThm1}

In this section, we prove our main Theorem~\ref{T:mainBody}. As an immediate consequence of the existence of an algebraic representative for codimension-2 cycles (see \S \ref{SS:AlgRep}), we obtain proofs of Theorem~\ref{T:main} and of Theorem~\ref{T:main-uni}.
We start with a lemma that parallels Lemma~\ref{L:Suwa-3.2}.

\begin{lem} \label{L:mainBody}
	Let $X$ be a smooth projective variety over an algebraically closed field $K=\bar K^a=\bar K$ and let $l$ be a prime. 
Let $ \phi:\operatorname{A}^n(X_{\bar K})\to A(\bar K)$, be a surjective regular homomorphism, 
and 	let  $\Gamma \in \mathscr A^{n}_{X/K}(A)$ be a miniversal cycle of degree $r$.
If 
	$$
	\phi [l^\infty]: \operatorname{A}^n(X_{\bar K})[l^\infty]\to A[l^\infty]
	$$
	is an inclusion,
then
	\begin{equation*} 
l^{v_l(r)}  \cdot T_l\operatorname{A}^n(X_{\bar K}) \subseteq \operatorname{im}\big(\Gamma_* : 	T_l \operatorname{A}_0(A) \to T_l \operatorname{A}^n(X_{\bar K})\big),
	\end{equation*}
	where   $v_l(r)$ is the $l$-adic valuation of $r$.
\end{lem}

\begin{proof}
 By the definition of a miniversal cycle and its degree, the composition
	$$
	\xymatrix{A(\bar K) \ar[r]&
		\operatorname{A}_0(A) \ar[r]^<>(0.5){\Gamma_*}&  \operatorname{A}^n(X_{\bar K})\ar[r]^<>(0.5){\phi} &  A(\bar K)
	}
	$$ is multiplication by $r$. Here, the map $A(\bar K) \to
	\operatorname{A}_0(A)$ is the map of sets 
	$a\mapsto [a]-[0]$.
	Recall the general fact about abelian varieties, due to Beauville,  that the map of sets $A(\bar K) \to \operatorname{A}_0(A_{\bar K}), a \mapsto [a]-[0]$ is an isomorphism on torsion
	(see \cite[Lem.~3.3]{ACMVdmij} for references, and recall that Beauville's argument works for arbitrary torsion in arbitrary characteristic). Therefore, restricting to $l$-primary torsion, we get a composition of homomorphisms
	$$
	\xymatrix{A[l^\infty] \ar[r]_{\simeq \quad \ } &
		\operatorname{A}_0(A)[l^\infty] \ar[r]^{  \Gamma_*}&  \operatorname{A}^n(X_{\bar K})[l^\infty] \ar[r]^<>(0.5){\phi[l^\infty]}_{\quad \simeq} &   A[l^\infty]
	}
	$$ 
	which is given by multiplication by $r$. Here $\phi [l^\infty]$ is an isomorphism due to
	Lemma~\ref{L:SRHom-facts}.
	Passing to the inverse limit, we obtain a composition of homomorphisms
	$$
	\xymatrix{T_l A \ar[r]_{\simeq \quad \ } &
		T_l \operatorname{A}_0(A) \ar[r]^{  \Gamma_*}& T_l \operatorname{A}^n(X_{\bar K})\ar[r]^<>(0.5){T_l\phi^n}_{\quad \simeq} &  T_l A
	}
	$$ 
which is given by multiplication by $r$. It immediately follows that $l^{v_l(r)}  T_l \operatorname{A}^n(X_{\bar K})$ lies in the image of~$\Gamma_*$.
\end{proof}

 \begin{teo} \label{T:mainBody}
 	  Let $X$ be a smooth projective variety over an algebraically closed field $K=\bar K^a=\bar K$ and let $l$ be a prime. 
 Let $ \phi:\operatorname{A}^n(X_{\bar K})\to A(\bar K)$, be a surjective regular homomorphism, and 
let  $\Gamma \in \mathscr A^{n}_{X/K}(A)$ be a miniversal cycle of  minimal degree $r$ (see~\S \ref{SS:mini}).
Then the morphisms
 \begin{equation*}
\xymatrix{
T_l A \ar[r]^<>(0.5){\Gamma_*}& T_l	\operatorname{A}^n(X_{\bar K})  \ar[r]^<>(0.5){T_l\lambda^n}&H^{2n-1}(X_{\bar
		K},\integ_l(n))_{\tau}
}
\end{equation*}
induce inclusions
$$
 \operatorname{im}(T_l \lambda^n \circ \Gamma_*)\subseteq  {\widetilde \coniveau}^{n-1} H^{2n-1}(X_{\bar K},\mathbb Z_l(n))_\tau 
  \subseteq \operatorname{im}(T_l\lambda^n).
$$

Moreover, if $
	\phi [l^\infty]: \operatorname{A}^n(X_{\bar K})[l^\infty]\to A[l^\infty]
	$
	is an inclusion, then in addition we have 
$$
 l^{v_l(r)} \operatorname{im}(T_l\lambda^n) \subseteq  \operatorname{im}(T_l \lambda^n \circ \Gamma_*).
 $$
In other words, if $\phi[l^\infty]$ is an inclusion, then  $\operatorname{im}(T_l\lambda^n) $ is an extension of $ {\widetilde \coniveau}^{n-1} H^{2n-1}(X_{\bar K},\mathbb Z_l(n))_\tau $ by a finite $l$-primary torsion group killed by multiplication by $l^{v_l(r)}$.  In particular, if $l$ does not divide $r$, then
  \begin{equation*} 
{\widetilde \coniveau}^{n-1} H^{2n-1}(X_{\bar K},\mathbb Z_l(n))_\tau =
  \operatorname{im}(T_l\lambda^n).
\end{equation*}
 \end{teo}

\begin{proof}
   The inclusion  ${{\widetilde \coniveau}}^{n-1} H^{2n-1}(X_{\bar K},\mathbb Z_l(n))_\tau \subseteq \operatorname{im}(T_l\lambda^n)$ is Proposition~\ref{P:ImBloch-Pre} (due to Suwa). 

Let us now show $ \operatorname{im}(T_l \lambda^n \circ \Gamma_*)\subseteq  {\widetilde \coniveau}^{n-1} H^{2n-1}(X_{\bar K},\mathbb Z_l(n))_\tau$.
For that purpose, consider the commutative diagram
 \begin{equation}\label{E:con-com-diag2}
\xymatrix{
T_l	\operatorname{A}_0(A_{\bar K})  \ar[rd]_{\Gamma_*} \ar@{^(->}[r] &	\displaystyle \bigoplus_{\Gamma':Z\vdash X}
T_l	\operatorname{A}_0(Z_{\bar K})  \ar@{->}[d]^{\Gamma'_*} \ar[r]^<>(0.5)= &
	\displaystyle \bigoplus_{\Gamma':Z\vdash X}
	T_l \operatorname{CH}_0(Z_{\bar K}) \ar[d]^{\Gamma'_*}
	\ar[r]^<>(0.5){T_l\lambda_0}_{\simeq }&\displaystyle\bigoplus_{\Gamma':Z\vdash X} H_1(Z_{\bar K},\integ_l)_\tau \ar[d]^{\Gamma'_*}\\
&T_l	\operatorname{A}^n(X_{\bar K}) \ar@{^(->}[r]&
T_l	\operatorname{CH}^n(X_{\bar K}) \ar[r]^<>(0.5){T_l\lambda^n}&H^{2n-1}(X_{\bar
		K},\integ_l(n))_{\tau}\\
}
\end{equation}
where the direct sums run through all smooth projective varieties $Z$ over $K$ and all correspondences $\Gamma' \in \operatorname{CH}^{d_X-n+1}(Z\times_K X)$. By definition, the image of the right vertical arrow consists of ${\widetilde \coniveau}^{n-1} H^{2n-1}(X_{\bar K},\mathbb Z_l(n))_\tau $, completing the proof via a diagram chase. 

Finally, under the assumption that $\phi [l^\infty]: \operatorname{A}^n(X_{\bar K})[l^\infty]\to A[l^\infty]$ is an isomorphism, the assertion $ l^{v_l(r)} \operatorname{im}(T_l\lambda^n) \subseteq  \operatorname{im}(T_l \lambda \circ \Gamma_*)$ follows from Lemma \ref{L:mainBody}.
\end{proof}

\begin{proof}[Proof of Theorems \ref{T:main} and \ref{T:main-uni}]
	Recall from \S \ref{SS:AlgRep} that an algebraic representative for codimension-2 cycles $ \phi^2_{X_{\bar K}}:\operatorname{A}^2(X_{\bar K})\to\operatorname{Ab}^2_{X_{\bar K}/\bar K}(\bar K)$ always exists.
	Both Theorems \ref{T:main} and \ref{T:main-uni} are then a special case of Theorem~\ref{T:mainBody}.
\end{proof}

Since 
an algebraic representative always exists for codimension-$2$ cycles (see \S \ref{SS:AlgRep}), in order to prove Theorem~\ref{T:mainBody} unconditionally for the algebraic representative in codimension-$2$, it suffices to  show the standard assumption holds.  The following lemma allows us to reduce to assuming the standard assumption holds for varieties over finite fields\,:

\begin{lem}[Standard assumption and generization] \label{L:StAs-spec}
  Let $S$ be the spectrum of a discrete valuation ring with generic
  point $\eta = \spec K$ and closed point $\circ = \spec \kappa$.  Let
  $X/S$ be a smooth projective scheme, and let $\Gamma \in \mathscr A^2_{X_\eta/K}(\operatorname{Ab}^2_{X_\eta/K})$ be a miniversal cycle of minimal degree $r$.  For all primes $\ell \nmid r\cdot \operatorname{char}(K)$, if $X_\circ$ satisfies the standard assumption at $\ell$ (\emph{i.e.}, $\phi^2_{X_\circ/\kappa}[\ell^\infty]$ is an isomorphism), then $X_\eta$ satisfies the standard assumption at $\ell$ (\emph{i.e.}, $\phi^2_{X_\eta/K}[\ell^\infty]$ is an isomorphism).
\end{lem}

\begin{proof}
By \cite[Thm.~8.3]{ACMVfunctor}
we have  $(\Ab^2_{X/S})_\eta \iso
  \Ab^2_{X_\eta/\eta}$.  
  Let $\Gamma_{X/S} \in \mathcal A^2_{X/S}(\Ab^2_{X/S})$ be a
  miniversal cycle of minimal degree $r$ induced by the one in the assumption of the lemma (see \cite[Lem.~4.7]{ACMVfunctor}). Its specialization induces a group homomorphism
  $w_{\Gamma_{X/S},\circ}: \Ab^2_{X/S}(\bar \kappa) \to
  \A^2(X_{\circ,\bar\kappa})$, and thus a homomorphism
  $\psi_{\Gamma_{X/S,\circ}}: (\Ab^2_{X/S})_\circ \to \Ab^2_{X_\circ/\circ}$.

 On $\ell$-primary torsion, we have a
  commutative diagram
  \begin{equation}
    \label{D:specializeA2}
  \xymatrix{
    \Ab^2_{X/S}[\ell^\infty](\bar K) \ar@{->}[r]^{\simeq} \ar[d]^{w_{\Gamma_{X/S}}}&
    (\Ab^2_{X/S})_\circ[\ell^\infty](\bar\kappa)
    \ar[d]^{w_{\Gamma_{X/S},\circ}}
    \ar@/^5pc/[dd]^{\psi_{\Gamma_{X/S},\circ}} \\
    \A^2(X_{\bar K})[\ell^\infty] \ar@{^{(}->}[r] & \A^2(X_0)[\ell^\infty]
    \ar@{^{(}->}[d]^{\phi^2_{X_\circ/\bar\kappa}[\ell^\infty]} \\
    & \Ab^2_{X_\circ/\kappa}[\ell^\infty](\bar\kappa)
  }
  \end{equation}
 Both the top and bottom horizontal 
  arrows are the specialization maps\,; the fact that the specialization map on torsion cycle classes is injective in codimension~$2$ follows from the fact, due to  Merkurjev--Suslin~\cite{MS} (see also Propositions~\ref{P:M-9.2} and~\ref{P:M-9.2bis}), that the second Bloch map is an inclusion.  Choose a prime $\ell\ne \operatorname{char}(K)$
  relatively prime to $r$.  Then $w_{\Gamma_{X/S}}$ and  is injective, and by commutativity, this implies $w_{\Gamma_{X/S,\circ}}$ is injective. 
   Then since we assume that $\phi^2_{X_\circ/\kappa}[\ell^\infty]$ is injective, it follows that all arrows in~\eqref{D:specializeA2} are injective.

  We now
  complete diagram~\eqref{D:specializeA2} by introducing a second copy
  of the isogeny $\psi_{\Gamma_{X/S,\circ}}$.  Note that the bottom
  square does \emph{not} commute, and the outer rectangle fails to
  commute by a factor of~$r$\,:
    \begin{equation}
    \label{D:specializeinjectA2}
  \xymatrix{
    \Ab^2_{X/S}[\ell^\infty](\bar K) \ar@{->}[r]^\simeq \ar[d]^{w_{\Gamma_{X/S}}}&
    (\Ab^2_{X/S})_\circ[\ell^\infty](\bar\kappa)
    \ar[d]^{w_{\Gamma_{X/S},\circ}}  \ar@/^5pc/[dd]^{\psi_{\Gamma_{X/S},\circ}} \\
    \ar @{}[dr] |{\not\circlearrowright}
    \A^2(X_{\bar K})[\ell^\infty] \ar@{^{(}->}[r]
    \ar[d]^{\phi^2_{X_\eta/\bar K}} & \A^2(X_0)[\ell^\infty]
    \ar[d]^{\phi_{X_\circ/\bar\kappa}} \\
    (\Ab^2_{X/S})[\ell^\infty](\bar K)  \ar[r] \ar[d]^\iso
    & \Ab^2_{X_\circ/\kappa}[\ell^\infty](\bar\kappa)\\
    \Ab^2_{X/S}[\ell^\infty](\bar\kappa) \ar[ur]_{\psi_{\Gamma_{X/S},\circ}}
  }
    \end{equation}
 If $\ell \nmid r$, then the injectivity of
 $\phi^2_{X_\circ/\bar\kappa}$ implies the injectivity of
 $\phi^2_{X_\eta/\bar K}$.
\end{proof}

\begin{pro}\label{P:finitefield}
 Assume that for any finite field $\mathbb F$ and any smooth
 projective variety $Y$ of dimension $d$ over $\overline {\mathbb F}$ we have that
 for all  primes $\ell\ne \operatorname{char}(\mathbb F)$, 
 $$\phi^2_{Y_{\bar {\mathbb F}}}[\ell^\infty]:
 \operatorname{A}^2(Y_{\bar {\mathbb F}})[\ell^\infty]\to
 \operatorname{Ab}^2_{Y_{\bar {\mathbb F}}/\bar {\mathbb F}}(\bar
              {\mathbb F})[\ell^\infty]$$ is an
              isomorphism\,;  \emph{i.e.}, assume the standard assumption holds for varieties of dimension $d$ over finite fields.
              
Let $X$ be a smooth projective variety of dimension $d$ over an algebraically closed field  $K=\overline K^a=\bar K$, and let $\Gamma\in \mathscr A^2_{X/\bar K}(\operatorname{Ab}^2_{X_{\bar K}/\bar  K})$ be a miniversal cycle of minimal degree $r$.  
Then for all primes $\ell \nmid r\cdot \operatorname{char}(K)$, we have 
  \begin{equation*} 
{\widetilde \coniveau}^{1} H^{3}(X_{\bar K},\mathbb Z_\ell(1))_\tau =
  \operatorname{im}(T_\ell\lambda^2: T_\ell	\operatorname{A}^2(X_{\bar K}) \hookrightarrow H^{3}(X_{\bar
		K},\integ_\ell(2))_{\tau}) .
\end{equation*}
\end{pro}

\begin{proof}
  By Proposition~\ref{P:phi}, we need only treat the case where $K$ has characteristic $p>0$.
  From Theorem~\ref{T:mainBody} we only need to establish that $X$ satisfies the standard assumptions for $\ell \nmid r \cdot \operatorname{char}(K)$.
  Since $X$ is of finite type over $K$, we may and do replace $K$ with
  a field of finite transcendence degree~$n$ over the prime field
  $\ff_p$.  By spreading out and then taking successive hypersurface sections,
  we may conclude from our hypothesis on finite fields and Lemma \ref{L:StAs-spec}.

  We provide details for the sake of completeness.
  Spread $X$ to a smooth scheme over $\spec R$, where $R$ is
  a smooth $\ff_p$-algebra of Krull dimension $n$ with $\operatorname{Frac}(R) = K$.  Let $D\subset \spec(R)$ be a prime divisor with
  generic point $\eta_D$.  Then $D$ defines a discrete valuation on
  $K$, whose valuation ring $R_{\eta_D} \supset R$ has residue field
  isomorphic to the function field of $D$.  By iterating this construction we obtain a
  sequence of ring surjections $R= R_n \twoheadrightarrow R_{n-1} \dots \twoheadrightarrow R_0$
  where $\dim R_j = j$, and for $j\ge 1$ the field $K_j := \operatorname{Frac}(R_j)$ admits a
  discrete valuation with valuation ring $S_j \supset R_j$ and residue field isomorphic to $K_{j-1}$.  By hypothesis, $X\times_{\spec
    R} \spec R_0$ satsifies
  the standard assumption at $\ell$.  By repeated invocation of
  Lemma~\ref{L:StAs-spec} for $X\times_{\spec R} \spec S_j$ with $j = 1, 2, \cdots, n$, we find
  that $X_K$ satisfies the standard assumption at $\ell$.
\end{proof}

\section{Decomposition of the diagonal and the image of the second $\ell$-adic Bloch map}
\label{SS:Bloch}
 
The aim of this section is to establish Theorem~\ref{T:co=H}. First, we have the following proposition that extends \cite[Prop.~2.3(ii)]{BenWittClGr} to the $\ell=p$ case.

\begin{pro}[{\cite[Prop.~2.3(ii)]{BenWittClGr}}]\label{P:lambdaBW}
	Let $X$ be a smooth projective variety over a perfect field $K$.
Assume that $N\Delta_{X_{\bar K}}\in
\operatorname{CH}^{d_X}(X_{\bar K}\times_{\bar K} X_{\bar K})$  admits a
decomposition
of type $(W_1,W_2)$ with $\dim W_2\le 1$.  Then, for all primes $l$, the
 second Bloch map
	 	$$\lambda^2 : \operatorname{A}^2(X_{\bar K})[l^\infty] \longrightarrow
	H^3(X_{\bar K},\integ_l(2))\otimes_{\mathbb Z_l}\mathbb Q_l/\mathbb Z_l$$
and  the 
 second $l$-adic Bloch map
	 	$$T_l \lambda^2 : T_l \operatorname{A}^2(X_{\bar K}) \longrightarrow
	H^3(X_{\bar K},\integ_l(2))_\tau$$
	are isomorphism of $\operatorname{Gal}(K)$-modules.
\end{pro}

\begin{proof}
The following  argument is due to Benoist--Wittenberg  for $l\neq \operatorname{char}(K)$. We check that it holds for $l=\characteristic(K)$ as well.   For $X$ smooth and projective, we have a diagram with exact row~\eqref{E:BlResA}\,:
	\[
	\xymatrix{
		&& \operatorname{CH}^2(X_{\bar K})[l^\infty]\ar@{^{(}->}[d]^{\lambda^2} \ar@{-->}[rd] \\
		0 \ar[r] & H^3(X_{\bar K},\integ_l(2))\otimes \rat_l/\integ_l \ar[r] & H^3(X_{\bar K},\rat_l/\integ_l(2))   \ar[r] & H^4(X_{\bar K},\integ_l(2))
	}
	\]
	where the dashed arrow is, up to sign, the cycle class map (\cite[Cor.~4]{CTSS83}, \cite[Prop.~III.1.16 and Prop.~III.1.21]{grossuwaAJ}).  Since algebraically trivial cycles are homologically trivial, it follows that the image of $\A^2(X_{\bar K})[l^\infty]$ under $\lambda^2$ is contained in $H^3(X_{\bar K},\mathbb Z_l)\otimes_{\mathbb Z_l}\mathbb Q_l/\mathbb Z_{l}\subseteq H^3(X_{\bar K},\mathbb Q_l/\mathbb Z_l)$.  In particular, the cokernel of $\lambda^2 : \A^2(X_{\bar K})[l^\infty] \to H^3(X_{\bar K},\mathbb Z_l)\otimes_{\mathbb Z_l}\mathbb Q_l/\mathbb Z_{l}$ is divisible.
	
	Now suppose $N\Delta_{X_{\bar K}}\in
\operatorname{CH}^{d_X}(X_{\bar K}\times_{\bar K} X_{\bar K})$  admits a
	decomposition
	of type $(W_1,W_2)$ with \linebreak $\dim W_2\le 1$.  
	With Notation~\ref{notations} we obtain by the naturality of the Bloch map (Proposition~\ref{P:correspondences})  a commutative diagram
	\begin{equation}\label{E:diagfact3}
	\xymatrix{
	\operatorname{A}^2(X_{\bar K})[l^\infty] \ar[r]^{r_1^* } \ar[d]^{\lambda^2} & \operatorname{A}^1(\widetilde{W}_1)[l^\infty]  \ar[r]^{ s_1^*} \ar[d]_{\simeq}^{\lambda^1}&
		\operatorname{A}^2(X_{\bar K})[l^\infty] \ar[d]^{\lambda^2} \\
		H^3(X_{\bar K},\integ_l(2))\otimes \rat_{l}/\integ_{l}  \ar[r]^{r_1^*} & H^1(\widetilde{W}_1, \integ_l(1))\otimes \rat_{l}/\integ_{l} \ar[r]^{ s_1^*}   & 	H^3(X_{\bar K},\integ_l(2))\otimes \rat_{l}/\integ_{l}.
	}
	\end{equation}
	Note there is no $\widetilde W_2$ term above for reasons of codimension.  
	The middle vertical arrow in~\eqref{E:diagfact3} is an isomorphism
 	 by
	Proposition~\ref{P:KummerA}, while the
	composition of the horizontal arrows in~\eqref{E:diagfact3}  is multiplication by
	$Np^e$.
	It follows that
	 $\operatorname{coker}\lambda^2$ is torsion, annihilated by $Np^e$, and consequently that this cokernel is trivial, \emph{i.e.}, $\lambda^2 : \A^2(X_{\bar K})[l^\infty] \to H^3(X_{\bar K},\integ_l(2))\otimes_{\integ_l} \rat_l/\integ_l$ is an isomorphism.

	Taking the Tate module of this isomorphism gives the result for the $l$-adic Bloch map, 
	since $H^3(X_{\bar K},\mathbb Z_l)$ is a finitely generated $\mathbb Z_l$-module, and so it is elementary to check that there is an identification $T_l( H^3(X_{\bar K},\mathbb Z_l)\otimes_{\mathbb Z_l}\mathbb Q_l/\mathbb Z_{l})\simeq H^3(X_{\bar K},\mathbb Z_l)_\tau$.
 \end{proof}

The following proposition establishes Theorem~\ref{T:co=H}.

\begin{pro}[{Theorem~\ref{T:co=H}}]\label{P:lambda}
	Let $X$ be a smooth projective variety over a perfect field
	$K$ of characteristic exponent $p$ and let $N$ be a natural number.
	Assume that  $N\Delta_{X_{\bar K}}\in
	\operatorname{CH}^{d_X}(X_{\bar K}\times_{\bar K}X_{\bar K})$  admits a decomposition
	of type $(W_1,W_2)$.
	\begin{enumerate}
		\item Assume $\dim W_2\le
		2$. Suppose $l$ is a prime such that $l\nmid N$, and such that either $l\ne \characteristic(K)$ or  resolution of singularities holds in dimensions $< d_X$.
		Then
		the inclusion 
		${\widetilde \coniveau}^1	H^3(X_{\bar K},\integ_l(2)) \subseteq 	H^3(X_{\bar K},\integ_l(2))$ is an equality,
		the inclusion $T_l \operatorname{A}^2(X_{\bar K})
		\hookrightarrow T_l \operatorname{CH}^2(X_{\bar K})$ is an equality,
		and the
		second $l$-adic Bloch map
		$$T_l \lambda^2 : T_l \operatorname{CH}^2(X_{\bar K}) \longrightarrow
		H^3(X_{\bar K},\integ_l(2))_\tau$$ is an isomorphism of
		$\operatorname{Gal}(K)$-modules. 
		\item Assume $\dim W_2\le
		1$. Let $l$ be any prime.
		Then
		the inclusion $T_l \operatorname{A}^2(X_{\bar K})
		\hookrightarrow T_l \operatorname{CH}^2(X_{\bar K})$ is an equality,
		and the
		second $l$-adic Bloch map
		$$T_l \lambda^2 : T_l \operatorname{CH}^2(X_{\bar K}) \longrightarrow
		H^3(X_{\bar K},\integ_l(2))_\tau$$ is an isomorphism of
		$\operatorname{Gal}(K)$-modules. Moreover, if $l\nmid N$ and if  either $l\ne \characteristic(K)$ or  resolution of singularities holds in dimensions $< d_X$, then $H^3(X_{\bar K},\integ_l)$ is
		torsion-free.
	\end{enumerate} 
\end{pro}

\begin{proof} 
	We first assume that $l=\ell \not = p$.
  	That $T_\ell \lambda^2$ is a morphism of  $\operatorname{Gal}(K)$-modules is Proposition~\ref{P:BlGal} and that $T_\ell \lambda^2$ is injective in general is Proposition~\ref{P:M-9.2}.
	
Concerning item (1), with Notation~\ref{notations} we obtain by the naturality of the $\ell$-adic Bloch map (Proposition~\ref{P:correspondences})  a commutative diagram
	\begin{equation}\label{E:diagfact2}
	\xymatrix{ T_\ell\operatorname{A}^2(X_{\bar K}) \ar[r]^{r_1^*\oplus r_2^*\qquad  \quad } \ar[d]^{T_\ell\lambda^2_X} & T_\ell\operatorname{A}^1(\widetilde{W}_1) \oplus  T_\ell \operatorname{A}^2(\widetilde{W}_2) \ar[r]^{\qquad \quad s_1^*+s_2^*} \ar[d]_{\simeq}^{T_\ell\lambda^1_{\widetilde{W}_1}\oplus T_\ell \lambda^2_{\widetilde{W}_2}}&
		T_\ell	\operatorname{A}^2(X_{\bar K}) \ar[d]^{T_\ell\lambda^2_X} \\
		H^3(X_{\bar K},\integ_\ell(2))_\tau  \ar[r]^{r_1^*\oplus r_2^*\qquad  \qquad} & H^1(\widetilde{W}_1, \integ_\ell(1))\oplus H^3(\widetilde{W}_2, \integ_\ell(2))_\tau \ar[r]^{\qquad \qquad   s_1^*+s_2^*\ \ }   & 	H^3(X_{\bar K},\integ_\ell(2))_\tau.
	}
	\end{equation}
	The middle vertical arrow in~\eqref{E:diagfact2} is an isomorphism by
	Propositions~\ref{P:Kummer} and~\ref{P:Rojtman}, while the
	composition of the horizontal arrows in~\eqref{E:diagfact2}  is multiplication by
	$Np^e$. In particular, the latter are bijective if $\ell$ does
	not divide $Np^e$. 
	A diagram chase then 
	establishes the surjectivity of $T_\ell \lambda^2_X$ restricted to algebraically trivial cycles, and hence the bijectivity of $T_\ell \operatorname{A}^2(X_{\bar K})
	\hookrightarrow T_\ell \operatorname{CH}^2(X_{\bar K})$ and of
	$T_\ell \lambda^2 : T_\ell \operatorname{CH}^2(X_{\bar K}) \longrightarrow
	H^3(X_{\bar K},\integ_\ell(2))_\tau$.
	Finally, since the composition 
	$$\xymatrix{ 
		H^3(X_{\bar K},\integ_\ell(2))  \ar[r]^{r_1^*\oplus r_2^*\qquad  \qquad} & H^1(\widetilde{W}_1, \integ_\ell(1))\oplus H^3(\widetilde{W}_2, \integ_\ell(2)) \ar[r]^{\qquad \qquad   s_1^*+s_2^*\ \ }   & 	H^3(X_{\bar K},\integ_\ell(2))
	}$$
	is multiplication by $Np^e$ and since $\dim \widetilde{W}_2 \leq 2$, we obtain from the equality $\widetilde \coniveau^1H^3 = \coniveau'^1H^3$ of Proposition~\ref{P:Jannsen},  for $\ell \nmid Np^e$, the inclusion   $ 	H^3(X_{\bar K},\integ_\ell(2)) \subseteq {\widetilde \coniveau}^1	H^3(X_{\bar K},\integ_\ell(2))$.
	
	In case (2), by Proposition~\ref{P:lambdaBW}, it suffices to see that $H^3(X_{\bar K},\integ_\ell(2))$ is torsion-free for $\ell$ not dividing $Np^e$. This follows simply from the factorization of the multiplication by $Np^e$ map as
	$$ \xymatrix{
		H^3(X_{\bar K},\integ_\ell(2))  \ar[r]^{r_1^*  } & H^1(\widetilde{W}_1, \integ_\ell(1)) \ar[r]^{   s_1^*}   & 	H^3(X_{\bar K},\integ_\ell(2))
	}$$
	and the fact that  $H^1(\widetilde{W}_1, \integ_\ell(1))$ is torsion-free.

	Now suppose $l = \characteristic(K) = p>0$.  Bearing in mind the properties of the $p$-adic Bloch map summarized in \S \ref{S:properties}, we see that the composition of
	the horizontal arrows in~\eqref{E:diagfact2} is again multiplication
	by $Np^e$.  If resolution of singularities holds in dimension at most
	$d_X-1$, then we may take $e=0$.  Under this hypothesis, if $p\nmid N$, we
	again see that $T_p \lambda^2$ is an isomorphism of
	$\gal(K)$-modules.
\end{proof} 

\begin{rem}
Note that Proposition~\ref{P:lambda}(2), together with Theorem~\ref{T:mainBody}, implies that $\operatorname{CH}_0(X_{\bar K})\otimes \rat$ is universally supported in dimension~1, then the primes $\ell$ for which ${{\widetilde \coniveau}}^{1} H^{3}(X_{\bar K},\mathbb Z_\ell(2))_\tau \subseteq  H^{3}(X_{\bar K},\mathbb Z_\ell(2))_\tau$ might fail to be an equality are the primes dividing the minimal degree of a miniversal cycle. Due to Proposition~\ref{P:decmini}, this is in this case \emph{a priori} finer than the conclusion of Proposition~\ref{P:lambda}(1). 
\end{rem}

 \section{Modeling cohomology via correspondences}\label{S:application}

In this section, we prove Theorems~\ref{T:co=H} and \ref{T:model}. The starting point is that a geometrically rationally chain connected variety (resp.~stably rational variety) has \emph{universally trivial} Chow group of zero-cycles with $\rat$-coefficients (resp.~$\integ$-coefficients)\,; see \S \ref{SS:dec} and specifically Remark~\ref{R:stablyrat}. We then combine the existence of the $\ell$-adic Bloch map with the existence of an  algebraic representative for codimension-2 cycles to establish Proposition~\ref{P:lambda} (which implies Theorem~\ref{T:co=H}) and the main Theorem~\ref{T:coniFil} (which implies Theorem~\ref{T:model}).
Along the way we establish related results concerning the third $\ell$-adic cohomology group of uniruled threefolds (Proposition~\ref{P:easy}).

\subsection{Modeling $\mathbb Q_\ell$-cohomology via correspondences} \label{SS:MazurQproof}
The aim of this section is to show Mazur's Questions~\ref{Q:MazurQ} and~\ref{Q:MazurQMot}, which are  with $\rat_\ell$-coefficients, can be easily answered positively under some assumption on the coniveau of $H^{2n-1}(X_{\bar K},\rat_\ell(n))$.
The following Proposition extends \cite[Thm.~2.1(d)]{ACMVdcg} in the positive characteristic case. Note that it applies to smooth projective geometrically uniruled threefolds.

\begin{pro}\label{P:easy}
	Let $X$ be a smooth projective variety over a perfect field $K$.
	Assume $H^{2n-1}(X_{\bar K},\rat_{\ell_0}(n)) = \coniveau^{n-1}H^{2n-1}(X_{\bar K},\rat_{\ell_0}(n))$ for some prime  $\ell_0\ne \operatorname{char}(K)$  and for some integer $n$ such that $2n-1 \leq d_X$.
	Then there exist an abelian variety $A$ over~$K$ and a cycle class $\Gamma \in\chow^n(A
	\times_K X)$ such that
	the induced morphism
  \begin{equation}
  \label{eq:easy}
  \xymatrix{\Gamma_* : V_l A \ar[r] & H^{2n-1}(X_{\bar K}, \rat_l(n))}
  \end{equation} is an
  isomorphism of $\operatorname{Gal}(K)$-modules for all primes $l$.
  (In particular, we have $H^{2n-1}(X_{\bar K},\rat_l(n)) = \coniveau^{n-1}H^{2n-1}(X_{\bar K},\rat_l(n))$.)

  Moreover, if $K$ has positive characteristic, then $\Gamma$ induces an isomorphism of F-isocrystals
\[
  \xymatrix{\Gamma_* :
  H^{2 d_A-1}(A/\mathbb K)(d_A) \ar[r] & 
  H^{2n-1}(X/\mathbb K)(n)}.
\]
\end{pro}

\begin{proof} By Proposition~\ref{P:Jannsen}, there is a
smooth projective variety $W$ over $K$ of dimension $d_X-n+1$ and a $K$-morphism $f : W \to X$ inducing a surjection $$f_* : H^1(W_{\bar K},\rat_{\ell_0}(1)) \twoheadrightarrow H^{2n-1}(X_{\bar K}, \rat_{\ell_0}(n)).$$
Let $Z_{W_{\bar K}} \in \chow^1((\mathrm{Pic}^0_{W_{\bar K}})_{\mathrm{red}} \times_{\bar K} W_{\bar K})$ be the universal divisor on $W_{\bar K}$\,; it induces an isomorphism of $\integ_{\ell_0}$-modules $T_{\ell_0} \mathrm{Pic}^0_{W_{\bar K}} \to H^1(W_{\bar K},\integ_{\ell_0}(1))$. Let $L/K$ be a finite field extension over which $Z_{W_{\bar K}}$ is defined.
 By pushing forward, we obtain a cycle $Z_W \in  \chow^1((\mathrm{Pic}^0_{W})_{\mathrm{red}} \times_{K} W_{K})$ inducing an isomorphism $V_{\ell_0} \mathrm{Pic}^0_{W} \to H^1(W_{\bar K},\rat_{\ell_0}(1))$ of $\operatorname{Gal}(K)$-modules. Let us set $B:= (\mathrm{Pic}^0_{W})_{\mathrm{red}}$.
 Composing $f_*$ with $Z_W$ we obtain a correspondence $\gamma \in \chow^n(B \times_K X)$ inducing a surjection $$\gamma_{*} : V_{\ell_0} B \twoheadrightarrow H^{2n-1}(X_{\bar K}, \rat_{\ell_0}(n))$$ of $\operatorname{Gal}(K)$-modules.
 Consider now the cycle $\Delta_* (c_1(\mathcal{O}_X(1))^{d-2n+1}) \in \chow^{2d-2n+1}(X\times_K X)$, where $\Delta : X \hookrightarrow X\times_K X$ is the diagonal embedding.
 By the Hard Lefschetz Theorem, this cycle induces an isomorphism $L: H^{2n-1}(X_{\bar K}, \rat_{\ell_0}(n)) \to H^{2d-2n+1}(X_{\bar K}, \rat_{\ell_0}(d-n+1))$ of $\operatorname{Gal}(K)$-modules and we obtain a homomorphism
  \begin{equation*}
 \xymatrix{
 V_{\ell_0} B    \ar@{->>}[r]^{\gamma_*\qquad \quad}  &H^{2n-1}(X_{\bar K}, \rat_{\ell_0}(n)) \ar[r]^{L\qquad\ }_{\simeq \qquad\ } & H^{2d-2n+1}(X_{\bar K}, \rat_{\ell_0}(d-n+1)) \ar@{^(->}[r]^{\qquad \qquad \quad  \gamma^*} & (V_{\ell_0} B)^\vee
 }
 \end{equation*}
induced by  the correspondence $\gamma^*\circ L\circ \gamma_* \in \chow^1(B\times_K B)$. In particular, the above homomorphism, which is  $\operatorname{Gal}(K)$-equivariant, is induced by a $K$-homomorphism $\varphi : B\to B^\vee$. It is clear that $\ker \gamma_* = \ker \varphi_*$. By Poincar\'e reducibility, there exist an abelian variety $A$ and $\psi \in \Hom (A,B)\otimes \rat$ such that $\gamma_* \circ \psi_* : V_{\ell_0} A \to  H^{2n-1}(X_{\bar K}, \rat_{\ell_0}(n))$ is an isomorphism.
In addition, there exists an idempotent $\theta \in \Hom(B^\vee,
B^\vee)\otimes \rat$ with image $A$ such that $\theta \circ \varphi =
\varphi$. Setting $\Gamma = \gamma\circ \psi$, it follows from the independence of $\ell$ of the
$\ell$-adic Betti numbers that $\Gamma_* : V_\ell A \to
H^{2n-1}(X_{\bar K}, \rat_\ell(n))$ is an isomorphism \emph{for all}
primes $\ell \ne \operatorname{char}(K)$.

Now suppose $\characteristic(K)=p>0$.  Since $\Gamma^* \circ \Gamma_*: V_\ell A \to V_\ell A$ is an isomorphism, this same cycle induces an
  automorphism of the $F$-isocrystal
  $H^1(A/\mathbb K)$ (\cite{katzmessing}, after
  a spread and specialization argument to reduce to $K$ finite).   Because crystalline and
  $\ell$-adic Betti numbers coincide,
  \[
  \xymatrix{\Gamma_* :
  H^{2 d_A-1}(A/\mathbb K)(d_A) \ar[r] & 
  H^{2n-1}(X/\mathbb K)(n)}
  \]
  is an isomorphism
   of crystals. Taking $F$-invariants shows that \eqref{eq:easy} holds for $l=p$, too.
  \end{proof}
  
\begin{rem}[$2n-1>d_X$]\label{R:easy}
We note that using the hard Lefschetz theorem, with the notation and assumptions of Proposition~\ref{P:easy},  there also exists 
  a cycle class $\Gamma' \in\chow^{d_X-n}(A
	\times_K X)$ such that for all primes $l$
  \begin{equation*}
  \label{eq:easyRem}
  \xymatrix{\Gamma'_* : V_l
	A \ar[r] &H^{2d_X-2n+1}(X_{\bar K},\rat_l(d_X-n+1))}
  \end{equation*} is an
  isomorphism of $\operatorname{Gal}(K)$-modules.
\end{rem}

In case $n=2$ and under the assumption that $V_l \phi^2_{X_{\bar K}} : V_l \operatorname{A}^2(X_{\bar K})
\longrightarrow V_l \mathrm{Ab}^2_{X/K}$ is an isomorphism for all primes $l$,
 one can make Question~\ref{Q:MazurQMot} more precise and ask whether there exists a correspondence $\Gamma \in \operatorname{CH}^2(\operatorname{Ab}^2_{X/K} \times_K X)\otimes \rat$ inducing for all primes $l$ the canonical identifications~\eqref{E:canQ}. We provide a positive answer for geometrically uniruled threefolds:

\begin{pro}\label{P:mazuruniruled}
  Let $X$ be a smooth projective variety over a perfect field $K$ and 
assume $\operatorname{CH}_0(X_{\bar K})\otimes \rat$ is universally supported in dimension 2, \emph{e.g.}\ $X$ is a geometrically uniruled threefold. (In particular, due to Proposition~\ref{P:phi}, $V_l \phi^2_{X_{\bar K}} : V_l \operatorname{A}^2(X_{\bar K})
\longrightarrow V_l \mathrm{Ab}^2_{X/K}$ is an isomorphism for all primes $l$.) 

Then there exists a correspondence $\Gamma \in \operatorname{CH}^2(\operatorname{Ab}^2_{X/K} \times_K X)\otimes \rat$ inducing for all primes $l$ the canonical identifications \eqref{E:canQ}
\begin{equation*} 
\xymatrix{
	V_l \operatorname{Ab}^2_{X/K} \ar@/_1.5pc/[rrrrr]^{\simeq}_{\Gamma_*} \ar[rr]^<>(0.5){(V_l\phi^2_{X_{\bar K}/\bar K})^{-1}}_\simeq &&V_l \operatorname{A}^2(X_{\bar K})\ar@{^(->}[r]&V_l
	\operatorname{CH}^2(X_{\bar K}) \ar@{^(->}[rr]^<>(0.5){V_l\lambda^2}&&H^{3}(X_{\bar
		K},\mathbb Q_l(2)).
}
\end{equation*}
\end{pro}
\begin{proof} First we note that the assumption that $\operatorname{CH}_0(X_{\bar K})\otimes \rat$ is universally supported in dimension 2 implies that
$\operatorname{im} (V_l \lambda^2) = \coniveau^1H^3(X_{\bar K},\rat_{l}(2))$ for all primes $l$. Together with Proposition~\ref{P:ImBloch-Pre}, this implies that $V_l\lambda^2$ is an isomorphism for all $l$, so that \eqref{E:canQ} is an isomorphism for all $l$. Second, let $Z\in \operatorname{CH}^2(\operatorname{Ab}^2_{X/K} \times_K X)$ be a miniversal cycle and let $r$ denote its degree. Then, by \cite[Cor.~11.8]{ACMVdiag} with $\rat_l$-coefficients, we have that $Z_* = r\cdot V_l\lambda^2 \circ (V_l\phi^2_{X_{\bar K}/\bar K})^{-1}$ for all primes $l$, and it follows that $\Gamma =\frac{1}{r} Z$ induces \eqref{E:canQ} for all primes $l$.
\end{proof}

  \subsection{Modeling $\mathbb Z_\ell$-cohomology via correspondences\,: Theorem~\ref{T:model}}
  \label{SS:correspmodel}
 Combining Proposition~\ref{P:phi} with Proposition~\ref{P:lambda}, we see that,
 under the assumption that $\Delta_{X_{\bar K}}\in
 \operatorname{CH}^{d_X}(X_{\bar K}\times_{\bar K} X_{\bar K})$  admits a
 decomposition
 of type $(W_1,W_2)$ with $\dim W_2\le 2$, we obtain for all primes $\ell \neq \operatorname{char}(K)$ canonical isomorphisms
 \begin{equation}\label{D:correspmodel}
   \begin{CD}
 T_\ell \mathrm{Ab}^2_{X/K} @>{(T_\ell \phi^2_{X_{\bar K}})^{-1}}>\simeq> T_\ell
 \operatorname{A}^2(X_{\bar K}) @>>\simeq>  T_\ell
 \operatorname{CH}^2(X_{\bar K}) @>{T_\ell \lambda^2}>\simeq> H^3(X_{\bar
  K},\integ_\ell(2))_\tau.
 \end{CD}
 \end{equation}
 Under the further assumption that $\dim W_2\le 1$, $H^3(X_{\bar
   K},\integ_\ell(2))$ is torsion-free by Proposition~\ref{P:lambda}
 and the main result below establishes that the isomorphisms~\eqref{D:correspmodel} are induced by a correspondence defined over
 $K$. In view of Proposition~\ref{P:BS} and Remark~\ref{R:stablyrat},
 the theorem below establishes Theorem~\ref{T:model}.

 \begin{teo}[Theorem~\ref{T:model}]\label{T:coniFil}
  Let $X$ be a smooth projective variety over a field $K$ of characteristic exponent~$p$, which is assumed to
  be either finite or algebraically closed.
  Assume that there is a natural number~$N$ such that 
  $N\Delta_{X}\in \operatorname{CH}^{d_X}(X\times_K X)$  admits a decomposition
  of type $(W_1,W_2)$ with  $\dim W_2\le 1$.
   Then there  exists
  a  correspondence $\Gamma\in \operatorname{CH}^{2}(\operatorname{Ab}^2_{X/K}\times _K
  X)$ 
  inducing for all primes $l$ not dividing $Np$  isomorphisms
  \begin{equation}\label{E:model1pf}
  \begin{CD}
  \Gamma_*:\ T_l \operatorname{Ab}^2_{X/K} @>\simeq>> H^{3}(X_{\bar K}, \mathbb Z_l(2))
  \end{CD}
\end{equation}
  of $\operatorname{Gal}(K)$-modules. If $p\ge 2$, if $p\nmid N$,  and if resolution of singularities holds in dimensions $< d_X$, then \eqref{E:model1pf} holds with $l=p$.

  Finally, if $ \operatorname{char}(K)=0$, the correspondence $\Gamma$ induces an
  isomorphism
   \begin{equation}\label{E:model2pf}
  \begin{CD}
  \Gamma_*:\ H_1((\operatorname{Ab}^2_{X/K})_\cx,\integ)\tensor\integ\big[\frac{1}{N}\big]  @>\simeq>> H^{3}(X_{\cx}, \mathbb Z(2))\tensor\integ\big[\frac{1}{N}\big].
  \end{CD}
  \end{equation}
 \end{teo}

 \begin{proof} First we focus on~\eqref{E:model1pf}.
 With Notation~\ref{notations}, we  have the diagram
  \begin{equation}\label{E:diagfact'}
\xymatrix{ \operatorname{A}^2(X_{\bar K}) \ar[r]^{r_1^*\oplus r_2^*\qquad } \ar[d]^{\phi^2_X} & \operatorname{A}^1(\widetilde{W}_1) \oplus  \operatorname{A}^2(\widetilde{W}_2) \ar[r]^{\qquad s_1^*+s_2^*} \ar[d]^{\phi^1_{\widetilde{W}_1}\oplus \phi^2_{\widetilde{W}_2}}&
	\operatorname{A}^2(X_{\bar K}) \ar[d]^{\phi^2_X} \\
	\mathrm{Ab}^2_{X/K}(\bar K)  \ar[r] & \operatorname{Pic}^0_{\widetilde{W}_1}(\bar K) \oplus  \operatorname{Alb}_{\widetilde{W}_2}(\bar K) \ar[r] & 	\mathrm{Ab}^2_{X/K}(\bar K).
}
 \end{equation}
 	Assuming $\dim W_2 \leq 1$, the diagram~\eqref{E:diagfact'} takes the simpler form
$$
\xymatrix{ \operatorname{A}^2(X_{\bar K}) \ar[r]^{r_1^* } \ar[d]^{\phi^2_X} & \operatorname{A}^1(\widetilde{W}_1) \ar[r]^{s_1^*} \ar[d]^{\phi^1_{\widetilde{W}_1}}_{\simeq} &
	\operatorname{A}^2(X_{\bar K}) \ar[d]^{\phi^2_X} \\
	\mathrm{Ab}^2_{X/K}(\bar K)  \ar[r]^{g} & \operatorname{Pic}^0_{\widetilde{W}_1}(\bar K)  \ar[r]^{f} & 	\mathrm{Ab}^2_{X/K}(\bar K)}
$$
where the composition of the horizontal arrows is multiplication by $Np^e$. Note that the homomorphisms $f$ and $g$ are in fact induced by $K$-homomorphisms
$	\mathrm{Ab}^2_{X/K}\to(\operatorname{Pic}^0_{\widetilde{W}_1})_{\mathrm{red}}  $ and $(\operatorname{Pic}^0_{\widetilde{W}_1})_{\mathrm{red}}  \to 	\mathrm{Ab}^2_{X/K}$ by \cite{murre83} in the case $K$ algebraically closed and by  the main result of \cite{ACMVdcg} in the case $K$ perfect.
By Proposition~\ref{P:lambda} with Proposition~\ref{P:phi}, we get for all $\ell$ not dividing $Np$ a commutative diagram
$$
\xymatrix{ H^3(X_{\bar K},\integ_\ell(2)) \ar[r]^{r_1^* }  & H^1(\widetilde{W}_1,\integ_\ell(1)) \ar[r]^{s_1^*}  &
	H^3(X_{\bar K},\integ_\ell(2)) \\
	T_\ell \operatorname{A}^2(X_{\bar K}) \ar[r]^{r_1^* }  \ar[u]_{T_\ell\lambda^2}^{\simeq}& T_\ell \operatorname{A}^1(\widetilde{W}_1) \ar[r]^{s_1^*} \ar[u]_{T_\ell\lambda^1}^{\simeq} &
	T_\ell \operatorname{A}^2(X_{\bar K})\ar[u]_{T_\ell\lambda^2}^{\simeq} \\
	T_\ell	\mathrm{Ab}^2_{X/K}  \ar[r]^{g_*} \ar[u]_{(T_\ell\phi^2_X)^{-1}}^{\simeq}&T_\ell \operatorname{Pic}^0_{\widetilde{W}_1} \ar[r]^{f_*} \ar[u]_{(T_\ell\phi^1_{\widetilde{W}_1})^{-1}}^{\simeq} & T_\ell	\mathrm{Ab}^2_{X/K}   \ar[u]_{(T_\ell\phi^2_X)^{-1}}^{\simeq}
}
$$
where the vertical arrows are isomorphisms and the composition of the horizontal arrows is multiplication by $Np^e$.

Now, the condition
that $K$ be either
finite or algebraically closed ensures that $\widetilde{W}_1$ admits a universal divisor $Z_{\widetilde{W}_1} \in \chow^1( \widetilde{W}_1 \times_K (\operatorname{Pic}^0_{\widetilde{W}_1})_{\mathrm{red}})$, meaning that the $K$-homomorphism $(\operatorname{Pic}^0_{\widetilde{W}_1})_{\mathrm{red}}  \to (\operatorname{Pic}^0_{\widetilde{W}_1})_{\mathrm{red}}$ induced by $Z_{\widetilde{W}_1}$ is the identity. In addition, the homomorphism
$$T_\ell\lambda^1 \circ (T_\ell\phi^1_{\widetilde{W}_1})^{-1} : T_\ell \operatorname{Pic}^0_{\widetilde{W}_1} \to H^1(\widetilde{W}_1,\integ_\ell(1))$$
 coincides with the action of $Z_{\widetilde{W}_1}$\,; this follows from Kummer theory, see \emph{e.g.}~\cite[\S 11]{ACMVdiag}.
We then define a codimension-2 cycle $Z \in \chow^2( X\times_K \mathrm{Ab}^2_{X/K} )$ as the composition $^{t}s_1 \circ {}^tZ_{\widetilde{W}_1} \circ g$. By a simple diagram chase, its induced action
$$Z_* = s_1^* \circ Z_{\widetilde{W}_1}^* \circ g_* : T_\ell	\mathrm{Ab}^2_{X/K} \to H^3(X_{\bar K},\integ_\ell(2))$$
is equal to $Np^e\, T_\ell\lambda^2 \circ (T_\ell\phi^2_{X})^{-1}$, which is an isomorphism.

If $K$ has positive characteristic $p>0$ and if resolution of singularities holds in dimension $< d_X$, then we may take $e=0$.  Consequently, under this hypothesis, the same argument shows that if $p\nmid N$, then \eqref{E:model1pf} is an isomorphism for $l=p$, too.

Finally for the case $K\subseteq \mathbb C$ and~\eqref{E:model2pf}, one uses essentially the same argument, but  with the canonical identification of the cohomology modulo torsion of a smooth complex projective variety with the first homology of the intermediate Jacobian.
 \end{proof}

\section{The image of the $\ell$-adic Bloch map in characteristic $0$}\label{S:char0}

In this  section we show that we can model the third integral cohomology with an abelian variety for any smooth projective variety liftable to a smooth projective rationally chain connected variety in characteristic $0$. We start with the following proposition, which essentially shows that the strategy of the proof of \cite[Thm.~B]{ACMVdcg} works with $\mathbb Z_\ell$-coefficients for rationally chain connected varieties.  
\begin{pro}\label{P:Ch0UniRul}
Let $X$ be a smooth projective variety over a field $K\subseteq \mathbb C$ such  that $\coniveau^1H^3(X_\cx, \rat) = H^3(X_\cx,\rat)$ (\emph{e.g.}, 
 $X$  is geometrically rationally chain connected, or $\dim X=3$ and $X$ is  geometrically uniruled).
 Then for any prime $\ell$ the morphisms
$$
T_\ell\phi^2_{X_{\bar K}/\bar K}:\ T_\ell\operatorname{A}^2(X_{\bar K}) \to  T_\ell\operatorname{Ab}^2_{X/K} \ \ \ \text{ and } \ \ \ 
 T_\ell\lambda^2: T_\ell\operatorname{A}^2(X_{\bar K})\to H^3(X_{\bar K},\mathbb Z_\ell(2))_\tau
$$
are isomorphisms, so that the composition
\begin{equation}\label{E:Ch0UniRul}
\xymatrix{
T_\ell\operatorname{Ab}^2_{X_K /K } \ar[r]^{(T_\ell \phi^2_{X_{\bar K }/\bar K})^{-1}}& T_\ell \operatorname{A}^2(X_{\bar K})\ar[r]^<>(0.5){T_\ell \lambda^2}& H^3(X_{\bar K},\mathbb Z_\ell(2))_\tau
}
\end{equation}
is an isomorphism.
\end{pro}

\begin{proof}
The fact due to Murre that $\phi^2_{X_{\bar K}/\bar K}[\ell^\infty]:\ \operatorname{A}^2(X_{\bar K})[\ell^\infty] \to  \operatorname{Ab}^2_{X/K}[\ell^\infty] $ is an isomorphism was explained in Proposition~\ref{P:phi}(1)\,; now one simply applies Tate modules to get that $T_\ell\phi^2_{X_{\bar K}/\bar K}$ is an isomorphism.  

We now consider the Bloch map. 
The first observation is that $\phi^2_{X_{\mathbb C}/\mathbb C}$ agrees with the Abel--Jacobi map $AJ:\operatorname{A}^2(X_{\mathbb C})\to J^3_a(X_{\mathbb C})$. 
Thus, applying Tate modules to the Abel--Jacobi map, we have that $T_\ell \phi^2_{X_{\mathbb C}/\mathbb C}= T_\ell AJ$, so that, by the above, $T_\ell AJ$ is an isomorphism.
 At the same time, the assumption $\coniveau^1H^3(X_\cx, \rat) = H^3(X_\cx,\rat)$ implies that the Abel--Jacobi map $AJ:\operatorname{A}^2(X_{\mathbb C})\to J^3(X_{\mathbb C})$ to the full intermediate Jacobian  is surjective (\emph{e.g.}, \cite[Thm.~12.22]{voisinI})\,; \emph{i.e.}, $J^3_a(X_{\mathbb C})= J^3(X_{\mathbb C})=H^{1,2}(X_{\mathbb C})/H^3(X_{\mathbb C},\mathbb Z)_\tau$.  We conclude using the fact that $T_\ell AJ= T_\ell\lambda^2$  (see Remark~\ref{R:AJ-ell-adic}).
\end{proof}

\begin{rem} Recall that for a smooth projective variety $X$ the condition $\coniveau^1H^3(X_\cx, \rat) = H^3(X_\cx,\rat)$ is implied by 
the diagonal $\Delta_{X_{\mathbb C}}\in
  	\operatorname{CH}^{d_X}(X_{\mathbb C}\times_{\mathbb C} X_{\mathbb C})_\rat$  admitting  a
  	decomposition
  	of type $(W_1,W_2)$ with $\dim W_2\le 2$.  The diagonal of $X_{\mathbb C}$ has such a decomposition if $X$ is geometrically rationally chain connected, or $\dim X=3$ and $X$ is geometrically uniruled, since in these cases $\operatorname{CH}_0(X_{\bar K})_{\mathbb Q}$ is universally supported on a surface (see Proposition~\ref{P:BS} and Remark~\ref{R:stablyrat}).  
\end{rem}

\begin{rem}
We note here that Proposition~\ref{P:Ch0UniRul} strengthens Theorem~\ref{T:MazurQ} in the case $n=2$ in that it allows for $\mathbb Z_\ell$-coefficients\,; on the other hand, Proposition~\ref{P:Ch0UniRul} is weaker than Theorem~\ref{T:MazurQ} in the sense that it does not provide a correspondence giving the isomorphism~\eqref{E:Ch0UniRul}.
Similarly, Proposition~\ref{P:Ch0UniRul} strengthens Theorem~\ref{T:coniFil} in the case $\operatorname{char}(K)=0$ in the sense that it gives isomorphisms for all primes without assuming a decomposition of the diagonal with $\integ$-coefficients, and it allows for a weaker form of the decomposition (see the previous remark).  
For example, in characteristic $0$, Theorem~\ref{T:coniFil}  implies  \eqref{E:Ch0UniRul} is an isomorphism for all primes if $X$ is geometrically stably rational, whereas  Proposition~\ref{P:Ch0UniRul} implies the same if $X$ is just assumed to be geometrically rationally chain connected (or even a geometrically uniruled threefold).
On the other hand, Proposition~\ref{P:Ch0UniRul} is weaker than Theorem~\ref{T:coniFil} in the sense that it does not provide a correspondence giving the isomorphism~\eqref{E:Ch0UniRul}.
\end{rem}

We now use Proposition~\ref{P:Ch0UniRul} to prove the following result on varieties liftable to characteristic $0$\,:

\begin{cor}\label{C:Ch0UniRul}
 Let $X_\circ$ be a smooth  projective  variety over a field $\kappa$, and suppose that $X_\circ$ lifts to a smooth projective variety
 $X_\eta$ over a field $K$ of characteristic $0$\,;  \emph{i.e.}, there is a DVR with spectrum   $S$,  generic
  point $\eta = \spec K$ with $\operatorname{char}(K)=0$,  closed point $\circ = \spec \kappa$, and 
 a smooth projective scheme $X/S$, with special fiber over $\circ$ equal to $X_\circ$, and generic fiber over $\eta$ equal to  $X_\eta$.
 
 If $\coniveau^1H^3((X_\eta)_\cx, \rat) = H^3((X_\eta)_\cx,\rat)$, \emph{e.g.}, if $X_\eta$ is a geometrically rationally chain connected variety, 
 or a smooth geometrically uniruled threefold,  
 then for any prime $\ell \ne \operatorname{char}(\kappa)$ the morphism
$$
 T_\ell \lambda^2: T_\ell \operatorname{A}^2(X_{\bar \kappa})\to H^3(X_{\bar \kappa},\mathbb Z_\ell(2))_\tau
$$
is an  isomorphism.
          \end{cor}

\begin{proof}
We consider the diagram
$$
\xymatrix{
T_\ell \operatorname{A}^2(X_{\bar K}) \ar@{^(->}[r]^{T_\ell \lambda^2} \ar[d]& H^3(X_{\bar K},\mathbb Z_\ell)_\tau \ar[d]^{\simeq}\\
T_\ell \operatorname{A}^2(X_{\bar \kappa })  \ar@{^(->}[r]^{T_\ell \lambda^2} & H^3(X_{\bar \kappa}, \mathbb Z_\ell)_\tau
}
$$
The vertical arrows are specialization maps.  The result holds by commutativity and by Proposition~\ref{P:Ch0UniRul}. 
\end{proof}

 \ifHideApp
\end{document}
\fi

\appendix

 \section{A review of the $l$-adic Bloch map}
 The aim of this appendix is to review the original construction~\cite{bloch79} of the Bloch map on $\ell$-primary torsion and show it in fact yields a direct construction of the $\ell$-adic Bloch map. We also review Suwa's construction~\cite{Suwa} of the $\ell$-adic Bloch map and show it coincides with our direct construction. For future referencing purposes we list in \S \ref{S:properties} the properties of the $\ell$-adic Bloch map that can be directly derived from the corresponding properties of the original Bloch map via Suwa's construction. In addition, the construction of Gros--Suwa~\cite{grossuwaAJ} of the $p$-adic Bloch map is briefly reviewed.  Finally, we study the restriction of the Bloch map to the subgroup of algebraically trivial cycles.

\subsection{Conventions for $\ell$-adic  and $p$-adic cohomology}

\subsubsection{$\ell$-adic cohomology} 

We fix a variety $X$ over a field $K$, and consider the \'etale
cohomology groups of $X_{\bar K}:= X\times_K \bar K$ with values in
prime-to-$\characteristic(K)$ torsion sheaves.  
For each $n$, $r$,  and $\nu$, we use the convention
\begin{align*}
H^n(X_{\bar K},\mathbb Z/\ell^\nu \mathbb Z(r)):= H^n_{\operatorname{\acute{e}t}}(X_{\bar K},\mmu_{\ell^\nu}^{\otimes r}),
\end{align*}
the \'etale cohomology of the \'etale sheaf $\mmu_{\ell^\nu}$ of $\ell^\nu$-roots of unity.  
Note that since 
\'etale cohomology is invariant under purely inseparable extensions, we can replace $\bar K$ with $\bar K^a$ throughout this subsection.

There are maps
$$
H^n(X_{\bar K},\integ/\ell^\nu\mathbb Z(r)) \ra H^n(X_{\bar
	K},\integ/\ell^{\nu+1}\mathbb Z(r))
$$
induced from the natural map $ \integ /\ell^{\nu}\mathbb Z \hookrightarrow
\integ /\ell^{\nu+1} \mathbb Z, [x]\mapsto [\ell x]$, or more precisely, from the natural map $\mmu_{\ell^\nu}\hookrightarrow \mmu_{\ell^{\nu+1}}$, $\zeta \mapsto \zeta^\ell$, 
as well as maps
$$
H^n(X_{\bar K},\integ /\ell^{\nu+1}\mathbb Z(r)) \ra H^n(X_{\bar K},\integ
/\ell^{\nu}\mathbb Z(r))
$$
induced from the natural quotient map $ \integ /\ell^{\nu+1}\mathbb Z
\twoheadrightarrow \integ /\ell^{\nu}  \mathbb Z$, or more precisely the natural map $\mmu_{\ell^{\nu+1}}\twoheadrightarrow \mmu_{\ell^{\nu+1}}/\mmu_\ell \simeq  \mmu_{\ell^\nu}$ 
The $\ell$-adic cohomology groups of $X$ are defined as follows\,:
\begin{align*}
H^n(X_{\bar K},\integ_\ell(r))&:=\varprojlim_{\nu} H^n(X_{\bar K},\integ
/\ell^{\nu}\mathbb Z(r))\\
H^n(X_{\bar K},\mathbb Q_\ell(r))&:= H^n(X_{\bar
	K},\integ_\ell(r))\otimes_{\mathbb Z_\ell}\mathbb Q_\ell.
\end{align*}
The cohomology groups of $X$ with $\ell$-torsion coefficients are defined as follows\,:
\begin{align*}
H^n(X_{\bar K},\mathbb Q_\ell/\integ_\ell(r))&:=\varinjlim_{\nu} H^n(X_{\bar K},\integ
/\ell^{\nu}\mathbb Z(r)).
\end{align*}
 We denote by 
\begin{align*}
\mathbb Z_\ell(r)&:=\varprojlim \mmu_{\ell^\nu}^{\otimes r}\\ 
\mathbb Q_\ell(r)&:=\mathbb Z_\ell(r)\otimes_{\mathbb Z_\ell}\mathbb Q_{\ell}\\
\mathbb Q_\ell/\mathbb Z_\ell(r)&:=\varinjlim \mmu_{\ell^\nu}^{\otimes r}
\end{align*}
and we can obtain the various twists in cohomology as\,:
\begin{align*}
H^n(X_{\bar K},\integ_\ell(r))&= H^n(X_{\bar K},\integ_\ell)\otimes_{\mathbb Z_\ell} \mathbb Z_{\ell}(r)\\
H^n(X_{\bar K},\mathbb Q_\ell(r))&=  H^n(X_{\bar K},\mathbb Q_\ell)\otimes_{\mathbb Q_\ell}\mathbb Q_\ell(r)\\
H^n(X_{\bar K},\mathbb Q_\ell/\mathbb Z_\ell(r))&=  H^n(X_{\bar K},\mathbb Q_\ell/\mathbb Z_\ell)\otimes_{\mathbb Q_\ell/\mathbb Z_\ell}\mathbb Q_\ell/\mathbb Z_\ell(r)
\end{align*} 

For the sake of completeness, we recall the following basic fact and its proof (see also \emph{e.g.}\  \cite[Ch.~III, Rmk.~3.6(d)]{milneetale})\,:

\begin{pro}
	\label{P:direct} Viewing $\mathbb Q_\ell/\mathbb Z_\ell$ as a torsion \'etale sheaf on $X$, 
	there is a natural isomorphism
	\begin{align*}
	H^n_{\operatorname{\acute{e}t}}(X_{\bar K},\rat_\ell/\integ_\ell)\otimes_{\mathbb Q_\ell/\mathbb Z_\ell} \mathbb Q_\ell/\mathbb Z_\ell(r)&= H^n(X_{\bar K},\rat_\ell/\integ_\ell(r)) \ \ (:=\varinjlim_{\nu} H^n(X_{\bar
		K},\integ /\ell^{\nu}\mathbb Z(r))).
	\end{align*}
\end{pro}

\begin{proof}

	The Snake Lemma applied to
	$$
	\xymatrix{
		0 \ar[r]&\mathbb Z_\ell \ar[r]^{\cdot \ell^\nu} \ar@{^(->}[d]& \mathbb Z_\ell
		\ar[r] \ar@{^(->}[d]& \mathbb Z/\ell^\nu\mathbb Z  \ar[r] \ar[d]& 0\\
		0 \ar[r]&\mathbb Q_\ell \ar[r]^{\cdot \ell^\nu}& \mathbb Q_\ell \ar[r]& 0
		\ar[r]& 0\\
	}
	$$
	gives a short exact sequence of \'etale sheaves
	$$
	\xymatrix{
		0 \ar[r] & \mathbb Z/\ell^\nu\mathbb Z \ar[r]&\mathbb Q_\ell/\mathbb Z_\ell
		\ar[r]^{\cdot \ell^\nu}& \mathbb Q_\ell/\mathbb Z_\ell \ar[r]& 0.\\
	}
	$$
	The map on the left can be written explicitly as $[x]\mapsto [x/\ell^\nu]$,
	where we are viewing $\mathbb Z\subseteq \mathbb Z_\ell$ in the natural way.
	This gives a long exact sequence in cohomology
	\begin{equation}\label{E:2.6.1.1}
	\xymatrix{
		\cdots  \ar[r] & H^n(X_{\bar K},\mathbb Z_\ell/\ell^\nu\mathbb Z_\ell)
		\ar[r]&H^n(X_{\bar K},\mathbb Q_\ell/\mathbb Z_\ell) \ar[r]^{\cdot \ell^\nu}&
		H^n(X_{\bar K},\mathbb Q_\ell/\mathbb Z_\ell) \ar[r]& \cdots \\
	}
	\end{equation}
	In fact we have a  diagram
	$$
	\xymatrix{
		0 \ar[r] & \mathbb Z/\ell^\nu\mathbb Z \ar[r] \ar[d]&\mathbb Q_\ell/\mathbb
		Z_\ell \ar[r]^{\cdot \ell^\nu} \ar@{=}[d]& \mathbb Q_\ell/\mathbb Z_\ell \ar[r]
		\ar[d]^{\cdot \ell}& 0\\
		0 \ar[r] & \mathbb Z/\ell^{\nu+1}\mathbb Z \ar[r]&\mathbb Q_\ell/\mathbb Z_\ell
		\ar[r]^{\cdot \ell^{\nu+1}}& \mathbb Q_\ell/\mathbb Z_\ell \ar[r]& 0\\
	}
	$$
	which is commutative since given $[x]\in \mathbb Z/\ell^\nu \mathbb Z$ we have
	$[x/\ell^\nu]=[\ell x / \ell^{\nu+1}]$.   Taking direct limits is exact, and so
	we obtain an exact sequence
	$$
	\xymatrix{
		\cdots \ar[r]& 0 \ar[r] & \displaystyle \varinjlim_\nu H^n(X_{\bar K},\mathbb
		Z_\ell/\ell^\nu\mathbb Z_\ell)  \ar[r]&H^n(X_{\bar K},\mathbb Q_\ell/\mathbb
		Z_\ell) \ar[r]&0  \ar[r]& \cdots ,
	}
	$$
	thereby settling the proposition.
\end{proof}

\subsubsection{$p$-adic cohomology}\label{S:p-adicConv}

For $K$ perfect of characteristic $p>0$, let $\ww(K)$ be the ring of Witt
vectors of $K$, with field of fractions $\kk(K)$.
For $X/K$, we
adopt the standard notation \cite[\S I.3.1]{grossuwaAJ}, \cite[\S 1]{milne86}
\begin{align*}
H^n(X, \integ/p^\nu \mathbb Z(r)) &:=
H^{n-r}_{\operatorname{\acute{e}t}}(X, \ww_\nu \Omega^r_{X,\log})\\
H^n(X_{\bar K},\integ_p(r))&:=\varprojlim_{\nu} H^n(X_{\bar K},\integ
/p^{\nu}\mathbb Z(r))\\
H^n(X_{\bar K},\mathbb Q_p(r))&:= H^n(X_{\bar
	K},\integ_p(r))\otimes_{\mathbb Z_p}\mathbb Q_p\\
H^n(X_{\bar K},\rat_p/\integ_p(r))&:=\varinjlim_{\nu} H^n(X_{\bar K},\integ
/p^{\nu}\mathbb Z(r))
\end{align*}
With
these conventions,   $H^n_{\operatorname{\acute{e}t}}(X_{\bar K}, \rat_p(r)):=\varprojlim_{\nu} H^n_{\operatorname{\acute{e}t}}(X_{\bar K},\mmu_{p^\nu}^{\otimes r})\otimes_{\mathbb Z_p}\mathbb Q_p$ and
$H^n(X, \rat_p(r))$ coincide \cite[(5.2.1)]{illusiedRW}, \cite[Prop.\ 1.15]{milne86}, but if $n>1$ then the corresponding statement for
integral coefficients need not hold.

For $X/K$ smooth and projective, we let $H^n(X,\ww(K)(r)):=H^n_{\operatorname{cris}}(X/\ww(K))(r)$  denote the crystalline cohomology group, and let $H^n(X,\kk(K)(r)):=
H^n(X,\ww(K)(r))\otimes_{\ww(K)}\kk(K)$.  This group may also be
computed as the rigid cohomology group
$H^n_{\operatorname{rig}}(X/\kk(K))(r)$.  Note that if we set $\ww_\nu(K) := \ww(K)/p^\nu \ww(K)$, then $H^n(X,\ww_\nu(K)(r)) = H^n_{\operatorname{cris}}(X/\ww_\nu(K)(r))$.

\subsection{The $\ell$-adic Bloch map}\label{S:construction}

In  \cite{bloch79}, Bloch constructed a map $$\lambda^n :
\chow^n(X_{\bar K^a})[\ell^\infty] \to H^{2n-1}(X_{\bar K^a},\rat_\ell/\integ_\ell(n))$$ for smooth projective varieties over a field $K$.  In this section, we review his construction, showing how it also defines a map 
$T_\ell\lambda^n  :
T_\ell \chow^n(X_{\bar K^a}) \to H^{2n-1}(X_{\bar K^a},\integ_\ell(n))_\tau$ on Tate modules.

\subsubsection{The Abel--Jacobi map on torsion}\label{S:AJ-tors}
We start by recalling the definition of the Abel--Jacobi map on torsion.  This is rather elementary from the definition of the Abel--Jacobi map, but gives some motivation for Bloch's approach to his algebraic construction of the  map, as well as some motivation for our interest in what we call the $\ell^\nu$-Bloch maps (Definition~\ref{D:finiteBlochmap}).  We also explain in Remark~\ref{R:AJ-ell-adic} that the $\ell$-adic Abel--Jacobi map~\eqref{E:an-ell_ad_Bloch} agrees with the $\ell$-adic Bloch map on homologically trivial cycle classes.

For a complex projective variety $X$ one has the Abel--Jacobi map
 $$
\operatorname{CH}^n(X)_{\operatorname{hom}}\stackrel{AJ}{\longrightarrow} J^{2n-1}(X)=F^n\backslash H^{2n-1}(X,\mathbb C)/H^{2n-1}(X,\mathbb Z)_\tau
$$
from the group of homologically trivial cycle classes of codimension-$n$ to the $(2n-1)$-st intermediate Jacobian $J^{2n-1}(X)$. Here, $F^\bullet$ refers to the Hodge filtration on $H^{2n-1}(X,\mathbb C)$\,; in the situation above and in terms of the Hodge decomposition we have $F^n H^{2n-1}(X,\mathbb C) = H^{2n-1,0}(X) \oplus H^{2n-2,1}(X) \oplus \dots \oplus H^{n,n-1}(X)$.
We can identify the torsion
$J^{2n-1}(X)[\ell^\nu]$ as follows.  For any complex torus $A=V/\Lambda$, we have $A[\ell^\nu]=\frac{1}{\ell^\nu}\Lambda/\Lambda$.  If we consider the commutative diagram of short exact sequences, 
$$
\xymatrix{
	0 \ar[r] & \Lambda \ar[r]  \ar[d]^{\cdot \ell^\nu}& \Lambda_{\mathbb Q} \ar[r] \ar[d]^{\cdot \ell^\nu}_\simeq & \Lambda_\mathbb Q/\Lambda \ar[r]  \ar[d]^{\cdot \ell^\nu}& 0\\
	0 \ar[r] & \Lambda \ar[r] & \Lambda_{\mathbb Q}\ar[r] & \Lambda_{\mathbb Q}/\Lambda \ar[r] & \cdots \\
}
$$
the snake lemma gives an identification $A[\ell^\nu]=\Lambda /\ell^\nu\Lambda$.  
In our situation with $A=J^{2n-1}(X)$, we have $\Lambda=H^{2n-1}(X,\mathbb Z)_\tau$.  In other words, $J^{2n-1}(X)[\ell^\nu]=H^{2n-1}(X,\mathbb Z)_\tau/\ell^\nu H^{2n-1}(X,\mathbb Z)_\tau$.  We then consider the diagram
$$
\xymatrix{
	0 \ar[r] & H^{2n-1}(X,\mathbb Z)_{\operatorname{tors}} \ar[r]  \ar[d]^{\cdot \ell^\nu}& H^{2n-1}(X,\mathbb Z) \ar[r] \ar[d]^{\cdot \ell^\nu} &  H^{2n-1}(X,\mathbb Z)_\tau \ar[r]  \ar[d]^{\cdot \ell^\nu}& 0\\
	0 \ar[r] & H^{2n-1}(X,\mathbb Z)_{\operatorname{tors}} \ar[r]  & H^{2n-1}(X,\mathbb Z) \ar[r] &  H^{2n-1}(X,\mathbb Z)_\tau \ar[r]  & 0\\
}
$$
For brevity, we denote $\delta^{an}_{\ell^\nu}= \frac{H^{2n-1}(X,\mathbb Z)_{\operatorname{tors}}}{\ell^\nu H^{2n-1}(X,\mathbb Z)_{\operatorname{tors}}}$ the cokernel of the vertical map on the left.
The snake lemma, and the long exact sequence in cohomology associated to $0\to \mathbb Z\stackrel{\cdot \ell^\nu}{\to} \mathbb Z\to \mathbb Z/\ell^\nu \mathbb Z\to 0$,  together give a diagram
$$
\xymatrix{
	0 \ar[r] & \delta^{an}_{\ell^\nu} \ar[r]  \ar@{=}[d]^{}& \frac{H^{2n-1}(X,\mathbb Z)}{\ell^\nu H^{2n-1}(X,\mathbb Z)} \ar[r] \ar@{^(->}[d]^{} &  \frac{H^{2n-1}(X,\mathbb Z)_\tau}{\ell^\nu H^{2n-1}(X,\mathbb Z)_\tau} \ar[r]  \ar[d]^{}& 0\\
	0 \ar[r] &   \delta^{an}_{\ell^\nu}\ar[r]  & H^{2n-1}(X,\mathbb Z/\ell^\nu \mathbb Z) \ar[r] &  H^{2n-1}(X,\mathbb Z/\ell^\nu \mathbb Z)/\delta^{an}_{\ell^\nu}  \ar[r]  & 0\\
}
$$
Thus we obtain maps

\begin{equation}\label{E:an-ell_nu_Bloch}
\xymatrix{
\operatorname{CH}^n(X)_{\operatorname{hom}}[\ell^\nu] \ar[rr]^{AJ[\ell^\nu]}&& J^{2n-1}(X)[\ell^\nu] \ar[r] & H^{2n-1}(X,\mathbb Z/\ell^\nu \mathbb Z)/\delta^{an}_{\ell^\nu}.
}
\end{equation}

It is easy to see that for sufficiently large $\nu$, we have $\delta^{an}_{\ell^\nu}=H^{2n-1}(X,\mathbb Z)_{\ell\operatorname{-tors}}$, and via the isomorphism $H^{2n-1}(X,\mathbb Z_\ell)= H^{2n-1}(X,\mathbb Z)\otimes_{\mathbb Z}\mathbb Z_\ell$, we have that $H^{2n-1}(X,\mathbb Z)_{\ell\operatorname{-tors}}=H^{2n-1}(X,\mathbb Z_\ell)_{\operatorname{tors}}$.   It follows that $\varprojlim \delta^{an}_{\ell^\nu}=H^{2n-1}(X,\mathbb Z_\ell)_{\operatorname{tors}}$.  

We claim now that $\varprojlim H^{2n-1}(X,\mathbb Z/\ell^\nu \mathbb Z)/\delta_{\ell^\nu}^{an}= H^{2n-1}(X,\mathbb Z_\ell)_\tau$.  For this we consider the short exact sequence $0\to \delta^{an}_{\ell^\nu}\to H^{2n-1}(X,\mathbb Z/\ell^\nu \mathbb Z) \to H^{2n-1}(X,\mathbb Z/\ell^\nu \mathbb Z)/\delta_{\ell^\nu}^{an}\to 0$,  and use the fact that since the $\delta^{an}_{\ell^\nu}$ are finite, we have $\varprojlim^1\delta_{\ell^\nu}^{an}=0$.  

Taking the inverse limit of the maps~\eqref{E:an-ell_nu_Bloch}, we therefore obtain a map

\begin{equation}\label{E:an-ell_ad_Bloch}
\xymatrix{
T_\ell \operatorname{CH}^n(X)_{\operatorname{hom}}\ar[rr]^{T_\ell AJ}&& T_\ell J^{2n-1}(X) \ar[r] & H^{2n-1}(X,\mathbb Z_\ell)_\tau
}
\end{equation}
and then tensoring with $-\otimes_{\mathbb Z_\ell}\mathbb Q_\ell$, we obtain a map

\begin{equation}\label{E:an-ell_ad_BlochV}
\xymatrix{
V_\ell \operatorname{CH}^n(X)_{\operatorname{hom}}\ar[rr]^{V_\ell AJ}&& V_\ell J^{2n-1}(X) \ar[r] & H^{2n-1}(X,\mathbb Q_\ell).
}
\end{equation}

The next claim is that $
\varinjlim \delta_{\ell^\nu}^{an}=0$.  This also follows from the fact that 
for sufficiently large $\nu$, we have $\delta^{an}_{\ell^\nu}=H^{2n-1}(X,\mathbb Z)_{\ell\operatorname{-tors}}$, since the latter group is finite, and is therefore killed by multiplication by $\ell^N$ for some sufficiently large $N$.  As a consequence, taking the direct limit in~\eqref{E:an-ell_nu_Bloch} we obtain a map

\begin{equation}\label{E:an-Bloch}
\xymatrix{
\operatorname{CH}^n(X)_{\operatorname{hom}}[\ell^\infty] \ar[rr]^{AJ[\ell^\infty]}&& J^{2n-1}(X)[\ell^\infty] \ar[r] & H^{2n-1}(X,\mathbb Q_\ell/\mathbb Z_\ell).
}
\end{equation}

\subsubsection{Bloch's preliminaries}  The set-up in \cite{bloch79} is from  the
paper \cite{BlochOgus}.
It is described in \cite{bloch79} in the following way.  One sets $\mathbf
H^q(\mmu_{\ell^\nu}^{\otimes n})$ to be the Zariski sheaf on $X_{\bar K^a}$
associated to the pre-sheaf $U\mapsto
H^q_{\operatorname{\acute{e}t}}(U,\mmu_{\ell^\nu}^{\otimes n})$.    In other
words, this is the derived push-forward of the sheaf $\mmu_{\ell^\nu}^{\otimes
	n}$ on $(X_{\bar K^a})_{\operatorname{\acute{e}t}}$ \emph{via} the morphism of sites
$\pi:(X_{\bar K^a})_{\operatorname{\acute{e}t}}\to (X_{\bar K^a})_{\operatorname{zar}}$, from the \'etale site to the Zariski site\,:
$$
\mathbf H^q(\mmu_{\ell^\nu}^{\otimes n}):=R^q\pi_*\mmu^{\otimes n}_{\ell^\nu}.
$$
Consider the composition of morphisms of sites $(X_{\bar
	K^a})_{\operatorname{\acute{e}t}}\to (X_{\bar K^a})_{\operatorname{zar}}\to
\operatorname{Spec}\bar K^a$.  The Leray spectral sequence is
$$
E^{p,q}_2=H^p_{\operatorname{zar}}(X_{\bar K^a},\mathbf
H^q(\mmu_{\ell^\nu}^{\otimes n}))\implies
H^{p+q}_{\operatorname{\acute{e}t}}(X_{\bar K^a},\mmu_{\ell^\nu}^{\otimes n}).
$$
The main tool is the existence of a particular  flasque resolution of $\mathbf
H^q(\mmu_{\ell^\nu}^{\otimes n})$ \cite[(1.3)]{bloch79},
$$
0\to \mathbf H^q(\mmu_{\ell^\nu}^{\otimes n})\to F^0\to F^1\to \cdots,
$$
which of course computes  $H^p_{\operatorname{zar}}(X_{\bar K^a},\mathbf H^q(\mmu_{\ell^\nu}^{\otimes
	n}))$ in the $p$-th place.
This resolution has two nice properties.  First, it turns out to be easy to read
off from the resolution that $$H_{\operatorname{zar}}^p(X_{\bar K^a},\mathbf
H^q(\mmu_{\ell^\nu}^{\otimes n}))=0$$ for $p>q$, and consequently, from the
shape of the spectral sequence, one obtains so-called boundary maps for the
spectral sequence  \cite[Cor.~1.4]{bloch79}
\begin{equation}\label{E:BM-E1}
H^{n-1}_{\operatorname{zar}}(X_{\bar K^a},\mathbf H^n(\mmu_{\ell^\nu}^{\otimes
	n}))\to H^{2n-1}_{\operatorname{\acute{e}t}}(X_{\bar K^a},\mmu_{\ell^\nu}^{\otimes
	n}).
\end{equation}
Second, the precise description of the flasque resolution in  \cite[(1.3)]{bloch79} shows that the group  $H^{n-1}_{\operatorname{zar}}(X_{\bar K^a}, \mathbf
H^n(\mmu_{\ell^\nu}^{\otimes n}))$ on the left in~\eqref{E:BM-E1} is the
cohomology of the complex \cite[Cor.~1.5]{bloch79}
$$
\bigoplus_{W^{n-2}\subseteq X_{\bar K^a}}H^2_{\operatorname{Gal}}(Q(W),\mmu_{\ell^\nu}^{\otimes 2})\to
\bigoplus_{V^{n-1}\subseteq X_{\bar K^a}} Q(V)^*/Q(V)^{*
	\ell^\nu}\stackrel{\partial_{\ell^\nu}}{\to} \bigoplus_{T^n\subseteq X_{\bar K^a}}\mathbb Z/\ell^\nu \mathbb Z,
$$
where the sums are taken over irreducible subvarieties of the indicated
\emph{codimensions}, and $Q(-)$ indicates the function field.   The map
$\partial_{\ell^\nu}$ is obtained from the standard exact sequence below after
reduction modulo $\ell^\nu$\,:
\begin{equation}\label{E:BM-E2}
\bigoplus_{V^{n-1}\subseteq X_{\bar K^a}}Q(V)^* \stackrel{\partial}{\to}
\bigoplus_{T^n\subseteq X_{\bar K^a}}\mathbb Z\to \operatorname{CH}^n(X_{\bar K^a})\to 0,
\end{equation}
where $\partial$ sends a rational function and to its divisor of zeros
and poles on $V$, and then one pushes forward \emph{via} the inclusion $V\subseteq X_{\bar K^a}$ to cycles on $X_{\bar K^a}$.

In particular, we have surjections
\begin{equation}\label{E:BM-E3}
\ker \partial_{\ell^\nu}\twoheadrightarrow H^{n-1}_{\operatorname{zar}}(X_{\bar K^a}, \mathbf H^n(\mmu_{\ell^\nu}^{\otimes n})).
\end{equation}

\subsubsection{The $\ell^\nu$-Bloch map}\label{S:ellnuBlochMap}

To get the construction started, we simply consider the commutative diagram of
short exact sequences  \emph{of groups} \cite[(2.1)]{bloch79}\,:
\begin{equation}\label{E:BM-3.5}
\xymatrix@R=2em@C=1em{
	0\ar@{->}[r]&\displaystyle \bigoplus_{V^{n-1}\subseteq X_{\bar K^a}} Q(V)^*/\bar
	K^{a*} \ar@{->}[r]^<>(1.0){(-)^{\ell^\nu}}\ar@{->}[d]^{\partial}&\displaystyle
	\bigoplus_{V^{n-1}\subseteq X_{\bar K^a}} Q(V)^*/\bar K^{a*} \ar@{->}[r]
	\ar@{->}[d]^{\partial}& \displaystyle \bigoplus_{V^{n-1}\subseteq X_{\bar K^a}}
	Q(V)^*/Q(V)^{*\ell^\nu}\ar@{->}[r] \ar@{->}[d]^{\partial_{\ell^\nu}}& 0\\
	0\ar@{->}[r]&\displaystyle \bigoplus_{T^{n}\subseteq X_{\bar K^a}} \mathbb Z
	\ar@{->}[r]^{\ell^\nu}&\displaystyle \bigoplus_{T^{n}\subseteq X_{\bar K^a}}
	\mathbb Z \ar@{->}[r] & \displaystyle \bigoplus_{T^{n}\subseteq X_{\bar K^a}}
	\mathbb Z/\ell^\nu \mathbb Z \ar@{->}[r]& 0\\
}
\end{equation}
The snake lemma yields \emph{via}~\eqref{E:BM-E2} the long exact sequence
$$
0\to \ker \partial \stackrel{\ell^\nu}{\to}\ker\partial\to \ker
\partial_{\ell^\nu}\to \operatorname{CH}^n(X_{\bar K^a})\stackrel{\ell^\nu}{\to
}\operatorname{CH}^n(X_{\bar K^a})\to \operatorname{CH}^n(X_{\bar K^a})/\ell^\nu\operatorname{CH}^n(X_{\bar K^a})\to 0.
$$
We obtain a diagram where the top row is a short exact sequence
\cite[(2.2)]{bloch79}\,:
\begin{equation}\label{E:BM-E4}
\xymatrix{
	0\ar@{->}[r]&\displaystyle \left(\frac{\ker \partial}{\ell^\nu \ker \partial}
	\right) \ar@{->}[r] \ar@/_2pc/@{->}[rdd]_{\rho_{\ell^\nu}}  &\ker
	\partial_{\ell^\nu}
	\ar@{->}[r]\ar@{->>}[d]^{\eqref{E:BM-E3}}&\operatorname{CH}^n(X_{\bar K^a})[\ell^\nu]\ar@{->}[r]&0\\
	&&H^{n-1}_{\operatorname{zar}}(X_{\bar K^a},\mathbf H^n(\mmu_{\ell^\nu}^{\otimes
		n}))\ar@{->}[d]^{\eqref{E:BM-E1}}&&\\
	&&H^{2n-1}_{\operatorname{\acute{e}t}}(X_{\bar K^a},\mmu_{\ell^\nu}^{\otimes
		n})&&\\
}
\end{equation}
where $\rho_{\ell^\nu}$ is defined as the indicated composition in the diagram.
In fact, we find it convenient to define $\delta_{\ell^\nu}$ to be the image of
$\rho_{\ell^\nu}$, to obtain the diagram\,:
\begin{equation}\label{E:BM-E4d}
\xymatrix{
	0\ar@{->}[r]&\displaystyle \left( \frac{\ker \partial}{\ell^\nu \ker
		\partial}\right) \ar@{->}[r] \ar@/_2pc/@{->}[rdd]_{\rho_{\ell^\nu}}
	\ar@{->>}[dd]&\ker \partial_{\ell^\nu}
	\ar@{->}[r]\ar@{->>}[d]^{\eqref{E:BM-E3}}&\operatorname{CH}^n(X_{\bar K^a})[\ell^\nu]\ar@{->}[r] \ar[dd]&0\\
	&&H^{n-1}_{\operatorname{zar}}(X_{\bar K^a},\mathbf H^n(\mmu_{\ell^\nu}^{\otimes
		n}))\ar@{->}[d]^{\eqref{E:BM-E1}}&&\\
	0\ar[r]&\delta_{\ell^\nu} \ar[r]&H^{2n-1}_{\operatorname{\acute{e}t}}(X_{\bar K^a},\mmu_{\ell^\nu}^{\otimes n}) \ar[r]
	&H^{2n-1}_{\operatorname{\acute{e}t}}(X_{\bar K^a},\mmu_{\ell^\nu}^{\otimes
		n})/\delta_{\ell^\nu} \ar[r] &0\\
}
\end{equation}

We now give a name to the vertical arrow on the right\,; this map is used tacitly
by Bloch in many places, and it will be convenient for us to give this a name\,:

\begin{dfn}[{The $\ell^\nu$-Bloch map}] \label{D:finiteBlochmap}
	The map
	\[
	\begin{CD}
	\chow^n(X_{\bar K^a})[\ell^\nu] @>{\lambda^n[\ell^\nu]}>>
	H^{2n-1}_{\operatorname{\acute{e}t}}(X_{\bar K^a},\mmu_{\ell^\nu}^{\otimes
		n})/\delta_{\ell^\nu}
	\end{CD}
	\]
	which is the negative of the map
	defined from~\eqref{E:BM-E4d} is  the \emph{$\ell^\nu$-Bloch map in
		codimension $n$}.
\end{dfn}

\begin{rem}
	As explained on \cite[p.112]{bloch79}, the choice of the negative sign is for
	compatibility, in the case $n=1$, with the natural map coming from the Kummer
	sequence, as in Proposition~\ref{P:Kummer}.
\end{rem}

\subsubsection{Bloch's Key Lemma}
Consider as in \cite[(2.3)]{bloch79} the map
\begin{equation}\label{E:BM-rho}
\xymatrix{\rho:\ker \partial\ar[r]& H^{2n-1}_{}(X_{\bar K^a},\mathbb
	Z_\ell(n))}
\end{equation}
defined as the composition
$$
\begin{CD}
\rho: \ker \partial @>>> \varprojlim \left(\ker \partial/\ell^\nu \ker
\partial\right) @>{\varprojlim_\nu \rho_{\ell^\nu}}>>
H^{2n-1}_{}(X_{\bar K^a},\mathbb Z_\ell(n)).
\end{CD}
$$
The following lemma, whose proof uses the Weil conjectures (\emph{via} specialization
to finite fields), is key to constructing the Bloch map on $\chow^n(X_{\bar
	K^a})[\ell^\infty]$ and the $\ell$-adic Bloch map on $T_\ell\chow^n(X_{\bar K^a})$.

\begin{lem}[Bloch's Key Lemma {\cite[Lem.~2.4]{bloch79}}] \label{L:BlochKey}
	The image of
	$\rho$
	is torsion.  \qed
\end{lem}

What is left tacit by Bloch, but is used in his construction of the Bloch map,
is that Lemma~\ref{L:BlochKey} implies\,:

\begin{lem}\label{L:Bl-p112-adic}
	The image of the map
	$$
	\varprojlim \rho_{\ell^\nu}:\varprojlim \left(\ker \partial/\ell^\nu \ker
	\partial\right)\to H^{2n-1}_{}(X_{\bar K^a},\mathbb
	Z_\ell(n))
	$$
	is torsion.
\end{lem}

\begin{proof}
	Since $\ker \partial\subseteq \varprojlim \left(\ker \partial/\ell^\nu \ker
	\partial\right)$ is a dense subset, and  the map  $\varprojlim \rho_{\ell^\nu}$
	on completions is continuous, the image of $\varprojlim \left(\ker
	\partial/\ell^\nu \ker \partial\right)$ (\emph{i.e.}, the image of $\varprojlim
	\rho_{\ell^\nu}$) is contained in the closure of the image of $\ker \partial$
	(\emph{i.e.}, the image of $\rho$)\,; the  image  of the closure of a set is contained in
	the closure of the image.  By Bloch's Key Lemma~\ref{L:BlochKey}, the image of
	$\rho$ is contained in  $\operatorname{Tors}
	H^{2n-1}_{\operatorname{\acute{e}t}}(X_{\bar K^a},\mathbb Z_\ell(n))$.
	Now use that $\operatorname{Tors} H^{2n-1}_{\operatorname{\acute{e}t}}(X_{\bar  K^a},\mathbb Z_\ell(n))\subseteq  H^{2n-1}_{\operatorname{\acute{e}t}}(X_{\bar K^a},\mathbb Z_\ell(n))$ is closed. Indeed, find  $N = \ell^r$ which kills the
	torsion\,; multiplication by $N$ is continuous, and the torsion is the inverse
	image of $0$ under this continuous map.
\end{proof}

\subsubsection{The $\ell$-adic Bloch map}

The $\ell$-adic Bloch map is defined using the inverse limit of
\eqref{E:BM-E4d}. We obtain a diagram
\begin{small}
\begin{equation}\label{E:ell-addBM}
\xymatrix{
	0\ar@{->}[r]&\displaystyle \varprojlim \left(\frac{\ker \partial}{\ell^\nu \ker
		\partial}\right) \ar@{->}[r] \ar@/_2pc/@{->}[rdd]_{\varprojlim_\nu
		\rho_{\ell^\nu}}   \ar@{->}[dd]&\varprojlim \ker \partial_{\ell^\nu}
	\ar@{->}[r]\ar@{->}[d]^{\eqref{E:BM-E3}}&T_\ell \operatorname{CH}^n(X_{\bar K^a})\ar@{->}[r] \ar[dd]&0\\
	&&\varprojlim  H^{n-1}_{\operatorname{zar}}(X_{\bar K^a},\mathbf
	H^n(\mmu_{\ell^\nu}^{\otimes n}))\ar@{->}[d]^{\eqref{E:BM-E1}}&&\\
	0\ar[r]&\varprojlim \delta_{\ell^\nu} \ar[r]
	\ar@{->}[d]&H^{2n-1}_{}(X_{\bar K^a},\mathbb Z_\ell(n))
	\ar[r] \ar@{->}[d] &H^{2n-1}_{}(X_{\bar K^a}, \mathbb
	Z_\ell(n))/ \varprojlim\delta_{\ell^\nu} \ar[r]  \ar@{->}[d] &0\\
	&(\varprojlim \delta_{\ell^\nu})_\tau
	\ar[r]&H^{2n-1}_{}(X_{\bar K^a},\mathbb Z_\ell(n))_\tau
	\ar[r] &H^{2n-1}_{}(X_{\bar K^a}, \mathbb Z_\ell(n))_\tau
	/ (\varprojlim\delta_{\ell^\nu})_\tau \ar[r] &0\\
}
\end{equation}
\end{small}
The top row  remains short exact after taking the inverse limit, since
$\varprojlim^1 \left(\ker \partial/\ell^\nu \ker \partial\right)=0$.  Indeed,
the system $ \left(\ker \partial/\ell^\nu \ker \partial\right)$ is clearly
surjective, by virtue of the fact that the terms are defined by quotients of an
increasing chain of subgroups.  The  $\delta_{\ell^\nu}$, being contained in
$H^{2n-1}_{\operatorname{\acute{e}t}}(X_{\bar K^a},\mmu_{\ell^\nu}^{\otimes n})$,
are finite, and thus the middle row, which is obtained as the inverse limit of
the bottom row of~\eqref{E:BM-E4d}, remains a short exact sequence, as well.  The bottom row is obtained from the middle row by taking the quotient by torsion subgroups in the left and middle entries.

Lemma \ref{L:Bl-p112-adic} and the commutativity of the diagram~\eqref{E:ell-addBM} yield
\begin{equation}\label{E:VPLimtau=0}
(\varprojlim \delta_{\ell^\nu})_\tau=0,
\end{equation}
allowing us to define the $\ell$-adic Bloch map\,:

\begin{dfn}[{$\ell$-adic Bloch map}]\label{D:BM-4}
	The map
	\[
	\begin{CD}
	T_\ell\chow^n(X_{\bar K^a}) @>{T_\ell \lambda^n}>> H^{2n-1}(X_{\bar  K^a},\integ_\ell(n))_\tau,
	\end{CD}
	\]
	which is the negative of the map
	defined from~\eqref{E:ell-addBM} and Lemma \ref{L:Bl-p112-adic}, is defined as
	the \emph{$\ell$-adic Bloch map in codimension $n$}.
\end{dfn}

\begin{rem}[$\ell$-adic Bloch map over the separable closure]\label{R:D:ell-Bloch79Sep}
	Using the fact that \'etale cohomology is invariant under purely inseparable extensions and the fact that the prime-to-$\operatorname{char}(K)$ torsion of Chow groups is also invariant under purely inseparable extensions (see \emph{e.g.}, \cite[Lem.~4.10]{ACMVfunctor}), the $\ell$-adic Bloch map over the algebraic closure  induces a well defined map $T_\ell \lambda^n: T_\ell \chow^n(X_{\bar K}) \to  H^{2n-1}(X_{\bar
		K},\mathbb Z_\ell(n))_\tau$.

\end{rem}

\subsubsection{The Bloch map}
Bloch defines his map by considering the direct limit of~\eqref{E:BM-E4d}\,:

\begin{equation}\label{E:BM-E6}
\xymatrix@C=1em{
	0\ar@{->}[r]&\displaystyle \varinjlim \left( \frac{\ker \partial}{\ell^\nu \ker
		\partial}\right) \ar@{->}[r] \ar@/_2pc/@{->}[rdd]_{\varinjlim \rho_{\ell^\nu}}
	\ar@{->>}[dd]& \varinjlim \ker \partial_{\ell^\nu}
	\ar@{->}[r]\ar@{->>}[d]^{\eqref{E:BM-E3}}&\operatorname{CH}^n(X_{\bar
		K^a})[\ell^\infty]\ar@{->}[r] \ar[dd]&0\\
	&&\varinjlim H^{n-1}_{\operatorname{zar}}(X_{\bar K^a},\mathbf
	H^n(\mmu_{\ell^\nu}^{\otimes n}))\ar@{->}[d]^{\eqref{E:BM-E1}}&&\\
	0\ar[r]&\varinjlim \delta_{\ell^\nu}
	\ar[r]&H^{2n-1}_{}(X_{\bar K^a},\mathbb Q_\ell/\mathbb
	Z_\ell(n)) \ar[r] &H^{2n-1}_{}(X_{\bar K^a},\mathbb
	Q_\ell/\mathbb Z_\ell(n))/ \varinjlim \delta_{\ell^\nu} \ar[r] &0\\
}
\end{equation}
Bloch's observation is that \cite[Lem.~2.4]{bloch79} implies the following\,:

\begin{lem}[{\cite[p.112]{bloch79}}]\label{L:Bl-p112}
	The map
	$$
	\varinjlim \rho_{\ell^\nu} : \varinjlim \left(\ker \partial/\ell^\nu \ker
	\partial\right)\to H^{2n-1}_{}(X_{\bar K^a},\mathbb
	Q_\ell/\mathbb Z_\ell(n))
	$$
	is the zero map.
\end{lem}

\begin{proof} To quote Bloch verbatim, the assertion follows
	from Lemma~\ref{L:BlochKey} using the fact that  the image in
	$H^{2n-1}_{\operatorname{\acute{e}t}}(X_{\bar K^a},\mmu_{\ell^\nu}^{\otimes n})$
	of the torsion in $H^{2n-1}_{}(X_{\bar K^a},\mathbb
	Z_\ell(n))$ is zero in
	$H^{2n-1}_{}(X_{\bar K^a},\mathbb Q_\ell/\mathbb
	Z_\ell(n))$.
	
	In a little more detail,
	we consider the commutative diagram
	$$
	\xymatrix{
		\displaystyle \varprojlim \left( \frac{\ker \partial}{\ell^\nu \ker
			\partial}\right)  \ar[r] \ar[d]^{\varprojlim
			\rho_{\ell^\nu}}&\displaystyle\frac{\ker \partial}{\ell^\nu \ker \partial}
		\ar[r] \ar[d]^{\rho_{\ell^\nu}}& \displaystyle \varinjlim \left( \frac{\ker
			\partial}{\ell^\nu \ker \partial}\right) \ar[d]^{\varinjlim \rho_{\ell^\nu}}\\
		H^{2n-1}_{\operatorname{\acute{e}t}}(X_{\bar K},\mathbb Z_\ell(n))       =
		\varprojlim H^{2n-1}_{\operatorname{\acute{e}t}}(X_{\bar
			K},\mmu_{\ell^\nu}^{\otimes n}) \ar[r]&
		H^{2n-1}_{\operatorname{\acute{e}t}}(X_{\bar K},\mmu_{\ell^\nu}^{\otimes n})
		\ar[r]
		& \varinjlim H^{2n-1}_{\operatorname{\acute{e}t}}(X_{\bar
			K},\mmu_{\ell^\nu}^{\otimes n})
	}
	$$
	where the horizontal maps are the natural maps.  To show the direct limit
	$\varinjlim \rho_{\ell^\nu}$ is the zero map, it suffices to show that for all
	$\nu$,  the image of $\frac{\ker \partial}{\ell^\nu \ker \partial}$ in
	$ \varinjlim H^{2n-1}_{\operatorname{\acute{e}t}}(X_{\bar
		K},\mmu_{\ell^\nu}^{\otimes n})$ is zero.
	To this end, let $\alpha\in \frac{\ker \partial}{\ell^\nu \ker \partial}$.
	As we observed before,   $\frac{\ker \partial}{\ell^\nu \ker \partial}$ forms a
	surjective system,  so we may lift $\alpha$ to $\beta\in \varprojlim \frac{\ker
		\partial}{\ell^\nu \ker \partial}$, and then send $\beta$ to $\gamma\in
	H^{2n-1}_{\operatorname{\acute{e}t}}(X_{\bar K},\mathbb Z_\ell(n))$.  By Lemma
	\ref{L:Bl-p112-adic}, $\gamma$ is torsion.  Now we use that the image of
	torsion, under the composition in the bottom row, is zero.
	
	Since this last assertion is not immediately obvious, we sketch a proof here.
	Let $(\alpha_1,\alpha_2,\dots)$ be a torsion element of
	$H^{2n-1}_{\operatorname{\acute{e}t}}(X_{\bar K},\mathbb Z_\ell)$, say of order
	$\ell^r$.   One can show that in order for  $(\alpha_1,\alpha_2,\dots)$ to be a
	consistent system, and also be $\ell^r$ torsion, one must have a $\mathbb
	Z/\ell^r\mathbb Z$ summand of the group 
	$H^{2n-1}_{\operatorname{\acute{e}t}}(X_{\bar K},\mathbb Z/\ell^{\nu}\mathbb
	Z)$ for all sufficiently large $\nu$, with  $\alpha_\nu\in \mathbb
	Z/\ell^r\mathbb Z\subseteq H^{2n-1}_{\operatorname{\acute{e}t}}(X_{\bar
		K},\mathbb Z/\ell^{\nu}\mathbb Z)$.  Now since
	$\ell^r \alpha_\nu=0$, by definition, this means that in the directed system
	for the direct limit the image of $\alpha_\nu$ in
	$H^{2n-1}_{\operatorname{\acute{e}t}}(X_{\bar K},\mathbb Z/\ell^{\nu+r}\mathbb
	Z)$ is zero (at each step, we multiply by $\ell$).  Thus the image of
	$\alpha_\nu$ in $\varinjlim H^{2n-1}_{\operatorname{\acute{e}t}}(X_{\bar
		K},\mathbb Z/\ell^{\nu}\mathbb Z)$ is zero.
\end{proof}

As a consequence of Lemma \ref{L:Bl-p112}, we can define the Bloch map\,; we also refer to \cite{CT} where a construction of the Bloch map is discussed and to the recent \cite[\S 10]{schreieder} where a construction of the Bloch map that avoids Bloch--Ogus theory~\cite{BlochOgus} is given.

\begin{dfn}[{Bloch map \cite[(2.7)]{bloch79}}]\label{D:Bloch79}
	The map
	\[
	\begin{CD}
	\chow^n(X_{\bar K^a})[\ell^\infty] @>{\lambda^n}>> H^{2n-1}(X_{\bar
		K^a},\rat_\ell/\integ_\ell(n)),
	\end{CD}
	\]
	which is the negative of the map
	defined from~\eqref{E:BM-E6} and Lemma \ref{L:Bl-p112}, is  the
	\emph{Bloch map in codimension $n$}.  In some cases we will write $\lambda^n[\ell^\infty]$ for clarity.
\end{dfn}

\begin{rem}[Bloch map over the separable closure]\label{R:D:Bloch79Sep}
	As in Remark~\ref{R:D:ell-Bloch79Sep}, the $\ell$-adic Bloch map induces a well defined map  $T_\ell \lambda^n: T_\ell \chow^n(X_{\bar K}) \to  H^{2n-1}(X_{\bar
		K},\integ_\ell(n))$.  
\end{rem}

\begin{rem}[Abel--Jacobi map on torsion] \label{R:AJ-tors}
	Bloch shows in \cite[Prop.~3.7]{bloch79} that for a complex projective manifold $X$, the Bloch map (Definition~\ref{D:Bloch79}) agrees with the Abel--Jacobi map on homologically trivial $\ell$-primary torsion~\eqref{E:an-Bloch}.  We explain below, in Remark~\ref{R:AJ-ell-adic}, that  this implies that the $\ell$-adic Bloch map (Definition~\ref{D:BM-4}) agrees with the Abel--Jacobi map on  homologically trivial cycle classes~\eqref{E:an-ell_ad_Bloch}.
\end{rem}

\subsection{Suwa's construction of the $l$-adic Bloch map}  \label{S:Suwa}

In \cite{Suwa}, Suwa has given a construction of the $\ell$-adic Bloch map by simply taking the Tate module associated to the standard Bloch map.  We review the construction here, and show it agrees with the construction given in \S \ref{S:construction}.  This construction is quite convenient in may cases.   Later Gros--Suwa \cite{grossuwaAJ} constructed an extension of the Bloch map to $p$-torsion\,; taking the Tate module gives a $p$-adic Bloch map.  We review this in~\S \ref{SS:pbloch}.  

\subsubsection{Structure of abelian $l$-primary torsion groups} \label{S:struct-l-tors}
Let $M$ be an abelian $l$-primary torsion
group for some prime number $l$.  
The set of divisible elements $M_{\operatorname{div}}$ forms a divisible abelian subgroup, and since divisible abelian groups are injective, we see that $M$ splits as a direct sum $M=M_{\operatorname{div}}\oplus (M/M_{\operatorname{div}})$.  
It is a basic fact (\emph{e.g.}, \cite[Ch. IV]{fuchsIAG})
that every divisible abelian $l$-primary torsion group is a direct sum of factors of the form $\mathbb Q_l/\mathbb Z_l$, so that $M=\left(\bigoplus \mathbb Q_l/\mathbb Z_l \right)\oplus (M/M_{\operatorname{div}})$.  We say that $M$ has \emph{finite corank} if there in an injective homomorphism $M\hookrightarrow (\mathbb Q_l/\mathbb Z_l)^r$ for some integer $r\ge 0$.   It is elementary to show that the following are equivalent\,:
\begin{itemize}
	\item $M$ has finite corank.
	\item $M[l]$ is finite.
	\item $M\simeq (\mathbb Q_l/\mathbb Z_l)^r\oplus A$ for some integer $r\ge 0$ and some finite $l$-primary torsion group $A$.
\end{itemize}

\subsubsection{$\ell$-adic cohomology from cohomology with torsion coefficients}
In this subsection,  we recall a crucial point used in Suwa's construction of the $\ell$-adic Bloch map  (see Proposition~\ref{P:TateQQ/ZZ}).   We start with a general statement about cohomology.

\begin{pro}\label{P:les}
	There is a natural long exact sequence
	\begin{equation}\label{E:P:les}
	\cdots
	\longrightarrow H^{n-1}(X_{\bar K},\mathbb Z_\ell )
	\longrightarrow  H^{n-1}(X_{\bar K},\mathbb Q_\ell) \stackrel{}{\longrightarrow}  H^{n-1}(X_{\bar
		K},\mathbb Q_\ell/\mathbb Z_\ell) \longrightarrow  H^{n}(X_{\bar K},\mathbb Z_\ell) {\longrightarrow}
	\cdots
	\end{equation}
	In particular, if $X$ is proper,
	\begin{equation}\label{E:Finiteness}
	H^{n}(X_{\bar K},\mathbb Q_\ell /\mathbb Z_\ell) \simeq (\mathbb Q_\ell/\mathbb
	Z_\ell)^r \oplus A
	\end{equation}
	for some integer $r$ and some finite $\ell$-primary torsion abelian group $A$.
\end{pro}
\begin{proof}
	Taking the inverse limit over $\mu$ of the long exact sequence associated to the short exact sequence of \'etale sheaves
	\begin{equation}
	\label{E:finitecoeff}
	\begin{CD}
	0 @>>> \mathbb Z/\ell^\mu\mathbb Z @>{\ell^\nu}>> \mathbb Z/\ell^{\mu +\nu}\mathbb Z @>>> \mathbb Z/\ell^\nu\mathbb Z @>>> 0
	\end{CD}
	\end{equation}
	gives a long exact sequence

	\[\xymatrix{\cdots \ar[r] & H^n(X_{\bar K},\integ_\ell) \ar[r]^{\ell^\nu} & H^n(X_{\bar K},\integ_\ell) \ar[r] & H^n(X_{\bar K},\integ/\ell^\nu\integ) \ar[r] & H^{n+1}(X_{\bar K},\integ_\ell) \ar[r] & \cdots}
	\]
	Taking the direct limit over $\nu$ of the system of long exact sequences
	\[\xymatrix{\cdots \ar[r] & H^n(X_{\bar K},\integ_\ell) \ar[r]^{\ell^\nu} \ar@{=}[d]
		& H^n(X_{\bar K},\integ_\ell) \ar[r] \ar[d]^{\ell} & H^n(X_{\bar K},\integ/\ell^\nu\integ) \ar[r] \ar[d]^{\ell}& H^{n+1}(X_{\bar K},\integ_\ell) \ar[r] \ar@{=}[d]& \cdots  \\
		\cdots \ar[r] & H^n(X_{\bar K},\integ_\ell) \ar[r]^{\ell^{\nu+1}}
		& H^n(X_{\bar K},\integ_\ell) \ar[r]  & H^n(X_{\bar K},\integ/\ell^{\nu+1}\integ) \ar[r] & H^{n+1}(X_{\bar K},\integ_\ell) \ar[r] & \cdots}
	\]
	together with Proposition~\ref{P:direct} provides the desired long exact sequence~\eqref{E:P:les}. Alternately, this can be obtained from the short exact sequence of sheaves $0 \to  \mathbb Z_\ell \to  \mathbb Q_\ell \to \mathbb Q_\ell/\mathbb Z_\ell \to 0$ 
	in the pro-\'etale topology\,; see \cite{BhattScholze}.

	The finiteness property~\eqref{E:Finiteness} follows immediately from the finiteness property of $H^n(X_{\bar K},\integ_\ell)$ when $X$ is proper. More elementarily, the finiteness property~\eqref{E:Finiteness} is a consequence of the finiteness of $H_{\operatorname{\acute{e}t}}^n(X_{\bar K},\integ/\ell\integ)$ and the fact (\S\ref{S:struct-l-tors})  that any $\ell$-primary torsion abelian group $M$ such that $M[\ell]$ is finite is isomorphic to  $(\mathbb Q_\ell/\mathbb
	Z_\ell)^r \oplus A$
	for some integer $r$ and some finite abelian group~$A$.
\end{proof}

From the long exact sequence~\eqref{E:2.6.1.1}, we obtain a short exact sequence
\begin{equation}\label{E:2.6.1.2}
\xymatrix{
	0  \ar[r]  & H^{n-1}(X_{\bar K},\mathbb Q_\ell/\mathbb Z_\ell)/\ell^\nu \ar[r]&
	H^n(X_{\bar K},\mathbb Z_\ell/\ell^\nu\mathbb Z_\ell)  \ar[r]&H^n(X_{\bar.
		K},\mathbb Q_\ell/\mathbb Z_\ell)[\ell^\nu]\ar[r]& 0.
}
\end{equation}
Now with the finiteness property~\eqref{E:Finiteness} and the notation therein, $(\mathbb Q_\ell/\mathbb Z_\ell)^r$ is the maximal
divisible subgroup of $H^{n}(X_{\bar K},\mathbb Q_\ell /\mathbb Z_\ell)$, and
$A=H^{n}(X_{\bar K},\mathbb Q_\ell /\mathbb Z_\ell)_{\operatorname{cotors}}$.
As $\mathbb Q_\ell/\mathbb Z_\ell$ is divisible, the $\ell^\nu$-torsion in
$\mathbb Q_\ell/\mathbb Z_\ell$ forms a surjective system.  Thus, since $A$ is
finite, the associated $\varprojlim^1$ for the direct sum vanishes, giving
\cite[(2.6.2)]{Suwa}\,:
\begin{equation}\label{E:2.6.2}
\xymatrix{
	0  \ar[r]  & \displaystyle \varprojlim_\nu \left(H^{n-1}(X_{\bar K},\mathbb
	Q_\ell/\mathbb Z_\ell)/\ell^\nu\right) \ar[r]&  H^n(X_{\bar K},\mathbb Z_\ell)
	\ar[r]&T_\ell H^n(X_{\bar K},\mathbb Q_\ell/\mathbb Z_\ell)\ar[r]& 0.
}
\end{equation}

\begin{pro}\label{P:TateQQ/ZZ}
	Assume  that  $X$ is a proper variety over a field  $K$. Then the  morphism $ H^n(X_{\bar K},\mathbb Z_\ell)
	\to T_\ell H^n(X_{\bar K},\mathbb Q_\ell/\mathbb Z_\ell)$ from~\eqref{E:2.6.2}, obtained as the inverse limit of the morphisms
	$H^n(X_{\bar K},\mathbb Z_\ell/\ell^\nu\mathbb Z_\ell) \to H^n(X_{\bar
		K},\mathbb Q_\ell/\mathbb Z_\ell)[\ell^\nu]$    from~\eqref{E:2.6.1.2},
	factors through the quotient $H^n(X_{\bar K},\mathbb Z_\ell)  \to H^n(X_{\bar K},\mathbb Z_\ell) _\tau$,  giving a natural isomorphism\,:
	\begin{equation}\label{E:TateQQ/ZZ}
	\xymatrix{
		H^{n}_{}(X_{\bar
			K},\mathbb Z_\ell)\ar[rd] \ar[r]   & T_\ell H^{n}_{}(X_{\bar K},\mathbb Q_\ell/\mathbb Z_\ell)  \\
		&H^{n}_{}(X_{\bar
			K},\mathbb Z_\ell)_\tau \ar[u]^\simeq
	}
	\end{equation}
\end{pro}
\begin{proof}
	It is a basic fact (\emph{e.g.}, \cite[Prop.~0.19]{milneADT}) that $T_\ell H^n(X_{\bar
		K},\mathbb Q_\ell/\mathbb Z_\ell)$ is torsion-free.  Thus,   considering~\eqref{E:2.6.2},  we clearly have $
	\displaystyle \varprojlim_\nu \left(H^{n-1}(X_{\bar K},\mathbb Q_\ell/\mathbb
	Z_\ell)/\ell^\nu\right) \supseteq   \operatorname{Tors}H^n(X_{\bar K},\mathbb
	Z_\ell)$.
	
	As asserted in  \cite[(2.6.3)]{Suwa}, we claim we have equality, completing the
	proof.  Specifically\,:
	\begin{align}
	\label{E:(2.6.3)-1}
	\displaystyle \varprojlim_\nu \left(H^{n-1}(X_{\bar K},\mathbb Q_\ell/\mathbb
	Z_\ell)/\ell^\nu\right)&= H^{n-1}(X_{\bar K},\mathbb Q_\ell/\mathbb
	Z_\ell)_{\operatorname{cotors}}\\
	\label{E:(2.6.3)-2}
	&=\operatorname{Tors} H^{n}(X_{\bar K},\mathbb Q_\ell/\mathbb Z_\ell).
	\end{align}
	Here is an explanation of the claim from Michael Spie\ss.
	To prove~\eqref{E:(2.6.3)-1}, we start with~\eqref{E:Finiteness}, that
	$H^{n-1}(X_{\bar K},\mathbb Q_\ell /\mathbb Z_\ell) = (\mathbb Q_\ell/\mathbb
	Z_\ell)^r \oplus A$, with $A$ a  finite $\ell$-primary torsion abelian group $A$.
	Then we consider the short exact sequence
	$$
	0 \to  \ell^\nu H^{n-1}(X_{\bar K},\mathbb Q_\ell/\mathbb Z_\ell) \to
	H^{n-1}(X_{\bar K},\mathbb Q_\ell/\mathbb Z_\ell) \to  H^{n-1}(X_{\bar
		K},\mathbb Q_\ell /\mathbb Z_\ell )/\ell^\nu \to  0.
	$$
	Its limit is
	$$
	0 \to  \varprojlim_\nu \ell^\nu H^{n-1}(X_{\bar K},\mathbb Q_\ell /\mathbb
	Z_\ell ) \to   H^{n-1}(X_{\bar K},\mathbb Q_\ell /\mathbb Z_\ell ) \to
	\varprojlim_\nu
	H^{n-1}(X_{\bar K},\mathbb Q_\ell /\mathbb Z_\ell )/\ell ^\nu  \to 0.
	$$
	To see that this stays short exact in the limit, we argue as above.  More
	precisely, as $\mathbb Q_\ell/\mathbb Z_\ell$ is divisible,   $\ell^\nu (\mathbb
	Q_\ell/\mathbb Z_\ell)$ forms a surjective system.
	We also have that $\ell^\nu A=0$ for $\nu$ sufficiently large.  Thus
	the associated $\varprojlim^1$ for the direct sum vanishes.
	Now, the image of the inclusion is exactly $(\mathbb Q_\ell /\mathbb Z_\ell)^r$,
	so that $\varprojlim
	H^{n-1}(X_{\bar K},\mathbb Q_\ell /\mathbb Z_\ell)/\ell^\nu$ identifies with
	$A$, \emph{i.e.}, with the cotorsion. This completes the proof of~\eqref{E:(2.6.3)-1}.

	For~\eqref{E:(2.6.3)-2}, the long exact sequence
	$$
	\cdots
	\to H^{n-1}(X_{\bar K},\mathbb Z_\ell )
	\to  H^{n-1}(X_{\bar K},\mathbb Q_\ell) \stackrel{i}{\to}  H^{n-1}(X_{\bar
		K},\mathbb Q_\ell/\mathbb Z_\ell) \to  H^{n}(X_{\bar K},\mathbb Z_\ell) {\to}
	H^n(X_{\bar K},\mathbb Q_\ell)\to  \cdots
	$$
	of Proposition~\ref{P:les} provides an identification
	$$
	\operatorname{Tors}H^{n}(X_{\bar K},\mathbb Z_\ell) = H^{n-1}(X_{\bar K},\mathbb
	Q_\ell /\mathbb Z_\ell)/\operatorname{Im}(i).
	$$
	The image of $i$, being the image of a divisible group,  is divisible in
	$H^{n-1}(X_{\bar K},\mathbb Q_\ell/\mathbb Z_\ell)$, and since
	$\operatorname{Tors}H^{n}(X_{\bar K},\mathbb Z_\ell)$ is finite,
	$\operatorname{Im}(i)$ must be the maximal divisible
	subgroup. This means  that we have $H^{n-1}(X_{\bar K},\mathbb Q_\ell/\mathbb
	Z_\ell)/\operatorname{Im}(i) =
	H^{n-1}(X_{\bar K},\mathbb Q_\ell/\mathbb Z_\ell)_{\operatorname{cotors}}$.
\end{proof}

\begin{lem}[Functoriality of~\eqref{E:TateQQ/ZZ}]
	\label{L:functor-can-iso}
	Let $f:X\to Y$ be a morphism of smooth proper varieties over $K$.  Then there are commutative diagrams\,:
	$$
	\xymatrix@C=1em{
		H^{n}(X_{\bar   K},\mathbb Z_\ell (d_Y-d_X))_\tau \ar[r]^<>(0.5)\simeq \ar[d]_{f_*}& T_\ell H^{n}(X_{\bar K},\mathbb Q_\ell /\mathbb Z_\ell(d_X-d_Y)) \ar[d]_{f_*}&&H^{n}(X_{\bar K},\mathbb Z_\ell)_\tau \ar[r]^<>(0.5)\simeq &T_\ell H^{n}(X_{\bar
			K},\mathbb Q_\ell/\mathbb Z_\ell) \\ 
		H^{n+2(d_X-d_Y)}(Y_{\bar   K},\mathbb Z_\ell)_\tau\ar[r]^<>(0.5)\simeq &T_\ell H^{n+2(d_Y-d_X)}(Y_{\bar K},\mathbb Q_\ell/\mathbb Z_\ell)&&H^{n}(Y_{\bar K},\mathbb Z_\ell)_\tau \ar[r]^<>(0.5)\simeq \ar[u]^{f^*}&T_\ell H^{n}(Y_{\bar K},\mathbb Q_\ell/\mathbb Z_\ell) \ar[u]^{f^*}
	}
	$$
	where the horizontal arrows are the isomorphisms from~\eqref{E:TateQQ/ZZ}.    More generally, given a correspondence $\Gamma :X\vdash  Y$, we obtain the corresponding commutative diagrams for $\Gamma_*$ and $\Gamma^*$.  
\end{lem}

\begin{proof}
	The commutativity of the diagrams follows from the definitions. 
\end{proof}

\begin{exa}\label{E:Tell-infty-ne}
	We note that while there are natural inclusions 
	\begin{align}\label{E:push-incl-suwa}
	f_*T_\ell H^n(X_{\bar K},\mathbb Q_\ell/\mathbb Z_\ell)  &\subseteq T_\ell f_*H^n(X_{\bar K},\mathbb Q_\ell/\mathbb Z_\ell),\\ 
	\label{E:pull-incl-suwa}
	f^*T_\ell H^n(Y_{\bar K},\mathbb Q_\ell/\mathbb Z_\ell) &\subseteq   T_\ell f^*H^n(Y_{\bar K},\mathbb Q_\ell/\mathbb Z_\ell),
	\end{align}
	these need not be equalities. 
	For instance, consider the case where $X=Y=E$ is an elliptic curve over $K$, and $f:X\to Y$ is the multiplication by $\ell$ map.  Then the containment $f_*T_\ell H^1(X_{\bar K},\mathbb Q_\ell/\mathbb Z_\ell)  \subseteq T_\ell f_*H^1(X_{\bar K},\mathbb Q_\ell/\mathbb Z_\ell)$ is the containment $(\ell \mathbb Z_\ell)^2\subsetneq (\mathbb Z_\ell)^2$.
	
\end{exa}

In accordance with the example above, the inclusions~\eqref{E:push-incl-suwa} and~\eqref{E:pull-incl-suwa} have torsion co-kernels in the case where $n=1$\,:

\begin{lem}\label{L:push-pull-QQ}
	Let $f:X\to Y$ be a morphism of smooth proper varieties over $K$.  Then we have isomorphisms
	\begin{align*}
	(f_*T_\ell H^1(X_{\bar K},\mathbb Q_\ell/\mathbb Z_\ell))\otimes_{\mathbb Z_\ell}\mathbb Q_\ell  & \stackrel{\simeq}{\longrightarrow} T_\ell f_*H^1(X_{\bar K},\mathbb Q_\ell/\mathbb Z_\ell) \otimes_{\mathbb Z_\ell}\mathbb Q_\ell, \\
	f^*T_\ell H^1(Y_{\bar K},\mathbb Q_\ell/\mathbb Z_\ell) \otimes_{\mathbb Z_\ell}\mathbb Q_\ell &\stackrel{\simeq}{\longrightarrow}    T_\ell f^*H^1(Y_{\bar K},\mathbb Q_\ell/\mathbb Z_\ell) \otimes_{\mathbb Z_\ell}\mathbb Q_\ell .
	\end{align*}
 More generally, given a correspondence $\Gamma :X\vdash  Y$, we obtain the corresponding commutation relations for  $\Gamma_*$ and $\Gamma^*$.  	
\end{lem}

\begin{proof}
	We will establish the first isomorphism regarding the push-forward. The second is similar, as are the cases of correspondences.
	From~\eqref{E:push-incl-suwa}, we only need to show that the morphism is surjective.  
	We consider the morphism $f_*:H^1(X_{\bar K},\mathbb Q_\ell/\mathbb Z_\ell)\to f_*H^1(X_{\bar K},\mathbb Q_\ell/\mathbb Z_\ell)\subseteq  H^{1+2(d_Y-d_X)}(Y_{\bar K},\mathbb Z_\ell/\mathbb Q_\ell)$.  Then, using that $H^1(X_{\bar K},\mathbb Q_\ell/\mathbb Z_\ell)$ is divisible (it is the cohomology of the abelian variety $\operatorname{Pic}^0_{X/K}$), and that the image of a divisible group is divisible,  we apply \cite[Lem.~1.4]{Suwa} that given a surjective homomorphism $\phi:M\to N$ of divisible abelian $\ell$-primary torsion groups, with $N$ of finite corank, the associated map $T_\ell \phi\otimes 1:T_\ell M\otimes_{\mathbb Z_\ell}\mathbb Q_\ell\to T_\ell N\otimes_{\mathbb Z_\ell}\mathbb Q_\ell$ is surjective. 
\end{proof} 

\subsubsection{Suwa's $\ell$-adic Bloch map}\label{SS:DSuwa}

\begin{dfn}[{Suwa's $\ell$-adic Bloch map  \cite[(2.6.5)]{Suwa}}] \label{D:Suwa}
	Assume $X$ is a smooth projective variety over $K$ and $\ell$ is a prime not equal to $\operatorname{char} K$. The \emph{$\ell$-adic Bloch map} for codimension-$n$ cycles is the map
	\[
	\begin{CD}
	T_\ell\chow^n(X_{\bar K}) @>{T_\ell \lambda^n}>> H^{2n-1}(X_{\bar
		K},\integ_\ell(n))_\tau
	\end{CD}
	\]
	obtained by applying $T_\ell$ to the Bloch map
	$ {\lambda^n} : \
	\chow^n(X_{\bar K})[\ell^\infty] \to  H^{2n-1}(X_{\bar
		K},\rat_\ell/\integ_\ell(n))
	$ and making the identification $T_\ell H^{n}_{}(X_{\bar K},\mathbb Q_\ell/\mathbb Z_\ell)=H^{n}_{}(X_{\bar
		K},\mathbb Z_\ell)_\tau$ from Proposition~\ref{P:TateQQ/ZZ}.
\end{dfn}

\subsubsection{The $\ell$-adic Bloch map and Suwa's construction}

Here we show that Suwa's construction of the $\ell$-adic Bloch map agrees with the direct construction.  

\begin{pro}\label{P:comparison}
	Let $X$ be a smooth projective variety over $K$. Then the $\ell$-adic Bloch  maps of Definition~\ref{D:Suwa} and Definition~\ref{D:BM-4} coincide.
\end{pro}

\begin{proof}
	We first observe that we have the following commutative diagram
	\begin{equation*}
	\xymatrix@R=1.5em@C=4em{
		\operatorname{CH}^n(X_{\bar K})[\ell^\nu] \ar[r]^<>(0.5){\lambda^n[\ell^\nu]} \ar[d]  \ar@/_5pc/[dd]& H^{2n-1}(X_{\bar K},\mathbb Z/\ell^\nu \mathbb Z)/\delta_{\ell^\nu} \ar[d] \ar@/^5pc/[dd]\\
		\left(\operatorname{CH}^n(X_{\bar K}\right)[\ell^\infty])[\ell^\nu] \ar[r]^{(\lambda^n[\ell^\infty])[\ell^\nu]} \ar@{^(->}[d]& H^{2n-1}(X_{\bar K},\mathbb Q_\ell/\mathbb Z_\ell)[\ell^\nu] \ar@{^(->}[d] \\
		\operatorname{CH}^n(X_{\bar K})[\ell^\infty] \ar[r]^{\lambda^n[\ell^\infty]}& H^{2n-1}(X_{\bar K},\mathbb Q_\ell/\mathbb Z_\ell)
	}
	\end{equation*}
	The bottom arrow, $\lambda^n[\ell^\infty]=\lambda^n$ is the Bloch map (Definition~\ref{D:Bloch79}), and the bottom square describes  the map induced on the $\ell^\nu$-torsion for the Bloch map.   In any case, the bottom square in the diagram is commutative by definition.
	The top map in the diagram  is the   $\ell^\nu$-Bloch map (Definition~\ref{D:finiteBlochmap}).   The vertical arrows from the top row to the bottom row are the canonical morphisms from the definition of a direct limit.  Thus the outer square is commutative.  It is only left to describe the vertical arrows from the top to the middle row.  These are defined by the fact that $\operatorname{CH}^n(X_{\bar K})[\ell^\nu]$ and $H^i(X_{\bar K},\mathbb Z/\ell^\nu \mathbb Z)/\delta_{\ell^\nu}$ are $\ell^\nu$-torsion groups, so that
	the
	images of the vertical arrows from the top row to the bottom row are contained in the $\ell^\nu$-torsion.
	
	Taking the inverse limit in the top square, we obtain a commutative diagram
	\begin{equation*}
	\xymatrix@R=1.5em@C=4em{
		T_\ell \operatorname{CH}^n(X_{\bar K}) \ar[r]^<>(0.5){T_\ell (\lambda^n[\ell^\nu])} \ar@{=}[d]  & H^i(X_{\bar K},\mathbb Z_\ell)_\tau \ar[d] \\
		T_\ell \operatorname{CH}^n(X_{\bar K})\ar[r]^<>(0.5){T_\ell (\lambda^n[\ell^\infty])} & T_\ell H^i(X_{\bar K},\mathbb Q_\ell/\mathbb Z_\ell)
	}
	\end{equation*}
	where the top row is the definition of $\ell$-adic Bloch map from Definition~\ref{D:BM-4}, while the bottom row is Suwa's definition of the $\ell$-adic Bloch map (Definition~\ref{D:Suwa}).
	The only thing to check is that the vertical arrow in the diagram is the canonical isomorphism from Proposition~\ref{P:TateQQ/ZZ}.  But this is clear from the construction of this isomorphism in Proposition~\ref{P:TateQQ/ZZ} \emph{via} the limit of the maps on the finite levels.
\end{proof}

\begin{rem}[$\ell$-adic Abel--Jacobi map] \label{R:AJ-ell-adic}
	The same argument as in the proof of Proposition~\ref{P:comparison} shows that for a complex projective manifold $X$, the $\ell$-adic Abel--Jacobi map $T_\ell AJ$~\eqref{E:an-ell_ad_Bloch}  is equal to the Tate module of the Abel--Jacobi map on torsion $AJ[\ell^\infty]$~\eqref{E:an-Bloch}\,; \emph{i.e.}, $T_\ell AJ=T_\ell (AJ[\ell^\infty])$.  
	Now using the fact that the Bloch map (Definition~\ref{D:Bloch79}) agrees with the Abel--Jacobi map on homologically trivial $\ell$-primary torsion~\eqref{E:an-Bloch} (\cite[Prop.~3.7]{bloch79}, Remark~\ref{R:AJ-tors}), it follows that the $\ell$-adic Bloch map $T_\ell \lambda^2$ (Definition~\ref{D:BM-4}) agrees with the $\ell$-adic Abel--Jacobi map $T_\ell AJ$  on  homologically trivial cycle classes~\eqref{E:an-ell_ad_Bloch}.
\end{rem}

\subsubsection{Gross--Suwa's $p$-adic Bloch map}
\label{SS:pbloch}

Now suppose that $K$ is a perfect field of characteristic $p>0$, and recall the
notation concerning $p$-adic cohomology groups.  With
these conventions, Gros and Suwa have secured $p$-adic versions of the
results reviewed above.  So, Proposition~\ref{P:direct} holds by
definition\,; the proof of Proposition~\ref{P:les} is valid at $p$, provided one replaces \eqref{E:finitecoeff} with the exact sequence of \'etale sheaves
\[
\xymatrix{
	0 \ar[r] & \ww_\mu \Omega^r_{X,\log} \ar[r] & \ww_{\mu+\nu} \Omega^r_{X,\log} \ar[r] & \ww_\nu \Omega^r_{X,\log} \ar[r] & 0
}
\]
(see also \cite[Prop.~I.4.18]{grossuwaAJ} for~\eqref{E:Finiteness}).  Gros and Suwa construct a group
homomorphism $\lambda^n = \lambda^n_p: \chow^n (X)[p^\infty] \to
H^{2n-1}(X,\rat_p/\integ_p(n))$ \cite[Def.~III.1.25]{grossuwaAJ}, 
and, as
in Definition~\ref{D:Suwa}, by applying the Tate module 
 obtain a map
  \[
\xymatrix{
	T_p \chow^n(X_{\bar K}) \ar[r]^-{T_p \lambda^n} & H^{2n-1}(X_{\bar K},\integ_p(n))_\tau.
}
\]

\begin{rem}\label{R:L:p-p-QQp}
	Moreover, Lemma \ref{L:functor-can-iso} holds for $p$-adic coefficients, as well.  Finally, Lemma \ref{L:push-pull-QQ} holds with $\mathbb Q_p/\mathbb Z_p$-coefficients since, as we have seen, $H^1(X_{\bar K},\rat_p/\integ_p)$ is $p$-divisible of finite corank. 
	
\end{rem}

\subsection{Properties of the Bloch maps}\label{S:properties}

In this section we fix a field $K$.  The aim of this section consists simply, for future reference, in restating known
results due to Bloch \cite{bloch79} concerning the usual Bloch map $\lambda^n :
\operatorname{CH}^n(X_{\bar K})[l^\infty]  \to H^{2n-1}(X_{\bar K},\mathbb
Q_\ell/\mathbb Z_l(n))$ in the setting of the $l$-adic Bloch map
$T_l \lambda^n : T_\ell  \operatorname{CH}^n(X_{\bar K}) \to H^{2n-1}(X_{\bar
	K},\mathbb Z_l(n))_\tau$, as well as the finite level Bloch maps.
All statements for the $l$-adic Bloch map are direct consequences of the fact   that $T_l \lambda^n$ is simply obtained from $\lambda^n$ by applying the Tate
module functor. Alternatively for $\ell\ne \operatorname{char}(K)$, using the direct definition of the $\ell$-adic Bloch map via the $\ell^\nu$-Bloch maps, the proofs in  \cite{bloch79} carry over directly.   In fact, the proofs in \cite{bloch79} regarding the Bloch map go directly through the corresponding assertions about the $\ell^\nu$-Bloch maps.

\begin{pro}[Flat pull back and proper push forward]
	\label{P:pbpf}
	The Bloch map, the
	$l$-adic Bloch, and for all primes $\ell\ne \operatorname{char}(K)$,  the $\ell^\nu$-Bloch maps, are functorial for flat pull back and proper push forward.
\end{pro}

\begin{proof}
	The case of the Bloch map is \cite[Prop.~3.3]{bloch79} for $\ell\ne \operatorname{char}(K)$ and \cite[III.~Prop.~2.3]{grossuwaAJ} in the $p$-adic case. The proposition for the $l$-adic Bloch map can be obtained
	simply obtained by applying $T_l$ to the case of the Bloch map.  For the $\ell^\nu$-Bloch maps, this follows directly from the proof of  \cite[Prop.~3.3]{bloch79}.
\end{proof}

\begin{pro}[Bloch maps and Galois-equivariance] \label{P:BlGal}
	The Bloch map, the
	$l$-adic Bloch map, and for $\ell\ne \operatorname{char}(K)$, the $\ell^\nu$-Bloch maps, are $\operatorname{Aut}(\bar K/K)$-equivariant.
\end{pro}
\begin{proof}  Fix $\ell\ne \operatorname{char}(K)$.  
	Let $L/K$ be a finite Galois extension, and let $X/K$ be a smooth projective variety.  Then for each $\sigma\in \operatorname{Gal}(L/K)$, applying the previous proposition to the morphisms $\sigma:X_L\to X_L$ given by the Galois descent data, shows that each of the various Bloch maps is $\operatorname{Gal}(L/K)$-equivariant.
	The general case follows by passing to the limit over all finite Galois extensions.  For the $p$-adic case, this is \cite[III.~Prop.~2.1]{grossuwaAJ}.
\end{proof}

\begin{pro}[Bloch maps and correspondences]
	\label{P:correspondences}
	The Bloch map, $l$-adic Bloch map, and for all primes $\ell\ne\operatorname{char}(K)$, the $\ell^\nu$-Bloch maps, are compatible with the action of correspondences.
	Precisely, in the case of the $l$-adic Bloch map,
	let $X$ and $Y$  be smooth projective varieties over $K$ and let $\Gamma$ be a
	cycle on $X\times_K Y$ of codimension $\dim Y +n-m$. Then the following natural
	diagram
	\begin{equation*}\label{E:BlCorr}
	\xymatrix{
		T_l \operatorname{CH}^m(Y_{\bar K})   \ar[d]_{T_l
			\lambda^m}\ar[r]^<>(0.5){\Gamma_*}
		& T_l \operatorname{CH}^{n}(X_{\bar K})  \ar[d]_{T_l\lambda^{n}}\\
		H^{2m-1}(Y_{\bar K},\mathbb Z_l(m))_\tau \ar[r]^<>(0.5){\Gamma_*}
		& H^{2n-1}(X_{\bar K},\mathbb Z_l(n))_\tau
	}
	\end{equation*}
	is commutative.
\end{pro}

\begin{proof}
The case of the Bloch map is \cite[Prop.~3.5]{bloch79} for $\ell\ne \operatorname{char}(K)$ and \cite[III.~Prop.~2.9]{grossuwaAJ} for the $p$-adic case. The proposition for the $l$-adic Bloch map is
	simply obtained by applying $T_l$ to the case of the Bloch map.  For the $\ell^\nu$-Bloch maps, this follows directly from the proof of \cite[Prop.~3.5]{bloch79}.
\end{proof}

\begin{pro}[Bloch maps and specialization]
	The Bloch map, $\ell$-adic Bloch map, and $\ell^\nu$-Bloch maps are compatible with specialization. Precisely, in the case of the $\ell$-adic Bloch map, given a local ring
	$R$ with fraction field $K$ and residue field $K_0$, and a smooth projective morphism $\mathscr X \to \spec R$ with generic fiber
	$X$ and special fiber $X_0$, the following diagram

	\begin{equation*}\label{E:BlSpe}
	\xymatrix{
		T_\ell \operatorname{CH}^n(\bar X)   \ar[d]_{T_\ell \lambda^n}\ar[r]
		& T_\ell \operatorname{CH}^{n}(\bar X_0)  \ar[d]_{T_\ell \lambda^{n}}\\
		H^{2n-1}(\bar X,\mathbb Z_\ell(n))_\tau \ar[r]^\simeq
		& H^{2n-1}(\bar X_0,\mathbb Z_\ell(n))_\tau
	}
	\end{equation*}
	is commutative for  $\ell$ prime to $\operatorname{char}(X_0)$. Here, $\bar X$ and $\bar X_0$ denote the base-changes of $X$ and $X_0$ to $\bar K$ and~$\bar K_0$, respectively. The top
	horizontal arrow is obtained by applying $T_\ell$ to the specialization map (\emph{e.g.}, \cite[Ex.~20.3.5]{fulton}), while the bottom horizontal arrow is obtained from smooth
	proper base-change.
\end{pro}
\begin{proof}
	The case of the Bloch map is \cite[Prop.~3.8]{bloch79}. The proposition in the case of the $\ell$-adic Bloch map is
	simply obtained by applying $T_\ell$ to the case of the Bloch map.  For the $\ell^\nu$-Bloch maps, this follows directly from the proof of \cite[Prop.~3.8]{bloch79}.
\end{proof}

\begin{pro}[Bloch maps and Kummer sequence] \label{P:Kummer}
	Let $X$ be a smooth projective
	variety over $K$.
	The $l$-adic Bloch map $$T_l\lambda^1 : T_l
	\operatorname{CH}^1(X_{\bar K}) \to H^1(X_{\bar K},\integ_l(1))$$ is the
	natural isomorphism arising from the Kummer sequence
	$$\begin{CD}
	0 @>>> \integ/l^\nu \integ (1)  @>>>  \mathbb{G}_m  @>{l^\nu}>>
	\mathbb{G}_m @>>> 0
	\end{CD}$$
	and the identification $\operatorname{CH}^1(X_{\bar K}) = H^1(X_{\bar K},
	\mathbb{G}_m)$.
\end{pro}
\begin{proof}
	The case of the Bloch map is \cite[Prop.~3.6]{bloch79} in the case $\ell\ne\operatorname{char}(K)$ and \cite[III.~Prop.~3.1, Cor.~3.2]{grossuwaAJ} in the $p$-adic case. The proposition is
	simply obtained by applying $T_l$ to the case of the Bloch map, and noting
	that $H^1(X_{\bar K},\integ_l(1))$ is torsion-free.  \end{proof}

\begin{pro}[Bloch maps and Albanese morphism, Roitman's theorem]
	\label{P:Rojtman}
	Let $X$ be a smooth projective
	variety of dimension $d$ over $K$\,; if $l = \characteristic(K)$, assume $K$ is perfect. Then the following diagram
  	$$\xymatrix{T_l \operatorname{CH}^d(X_{\bar K}) \ar[r]^{T_l\lambda^d\quad\
		} \ar[rd]_{\operatorname{alb}}
		& H^{2d-1}(X_{\bar K},\integ_l(d))_\tau \ar@{=}[d] \\
		& T_l \operatorname{Alb}_X
	}
	$$ is commutative, where $\operatorname{alb}$ is obtained by applying $T_l$
	to the map $\operatorname{CH}^d(X_{\bar K}) \to  \operatorname{Alb}_X(\bar K)$
	mapping a zero cycle on $X_{\bar K}$ to the sum of the corresponding points on
	the Albanese. Moreover, $T_l\lambda^d$ is an isomorphism.
\end{pro}
\begin{proof}
	The case of the Bloch map is \cite[Prop.~3.9]{bloch79} for $\ell\ne \operatorname{char}(K)$ and \cite[III.~Prop.~3.14, Cor.~3.17]{grossuwaAJ} in the $p$-adic case. The commutativity of the diagram is
	simply obtained by applying $T_l$ to the case of the Bloch map. Finally, that $\mathrm{alb} :T_l \operatorname{CH}^d(X_{\bar K}) \to T_l \operatorname{Alb}_X$ is an isomorphism is due to Roitman \cite[Thm.~4.1]{bloch79}. Note that in \emph{loc.~cit.}~it is stated that $\mathrm{alb} : \operatorname{CH}^d(X_{\bar K}) \to  \operatorname{Alb}_X(\bar K)$ is an isomorphism on torsion for $\bar K$ algebraically closed\,; this implies the needed fact that $\operatorname{CH}^d(X_{\bar K}) \to  \operatorname{Alb}_X(\bar K)$ is an isomorphism on prime-to-$\mathrm{char}(K)$ torsion for $\bar K$ separably closed. Indeed, this follows from the general fact that the prime-to-$\mathrm{char}(K)$ torsion in the Chow group of a scheme of finite type over $K$ is invariant under purely inseparable extensions (see \emph{e.g.}, \cite[Lem.~4.10]{ACMVfunctor}),	and the fact that the prime-to-$\mathrm{char}(K)$ torsion of an abelian variety is invariant under extension of separably closed fields.
\end{proof}

Finally, we have the following $l$-adic analogue of a result of due to
Merkurjev--Suslin \cite{MS}.

\begin{pro}[Injectivity of the second Bloch map]\label{P:M-9.2}
	Let $X$ be a smooth projective variety over $K$\,; if $l = \characteristic(K)$, assume $K$ is perfect. 
	The second $l$-adic Bloch map

	\begin{align*}
	T_l \lambda^2: \ &T_l \operatorname{CH}^2(X_{\bar K})\longrightarrow H^3(X_{\bar
		K},\mathbb Z_l(2))_{\tau}
	\end{align*}
	is injective, as are the second  $\ell^\nu$-Bloch maps
	$\lambda^2[\ell^\nu]:\operatorname{CH}^2(X_{\bar K})[\ell^\nu]\to H^3(X_{\bar K},\mmu_{\ell^\nu}^{\otimes 2})/\delta_{\ell^\nu}$ for $\ell\ne \operatorname{char}(K)$.  
\end{pro}

\begin{proof} The injectivity of the Bloch map $\lambda^2$ in the case $\ell\ne \operatorname{char}(K)$ is due to
	Merkurjev--Suslin \cite{MS}\,; see also
	\cite[Prop.~9.2]{murre83} and  \cite[Prop.~3.1 and
	Rmk.~3.2]{CTR}. For the $p$-adic case this is \cite[III.~Prop.~3.4]{grossuwaAJ}. The injectivity of the second $l$-adic Bloch map follows \emph{via}
	applying $T_l$ to the Bloch map.
 	
For $\ell\ne \operatorname{char}(K)$, the fact that the second $\ell^\nu$-Bloch maps are injective is \cite[Prop.~6.1]{murre83}.  Indeed, we consider diagram \eqref{E:BM-E4d}. 
Using $K$-theoretic techniques, Murre showed that \eqref{E:BM-E4d} can be factored as follows \cite[Rem.~6.3]{murre83}\,:

\begin{equation}\label{E:BM-E4-Murre}
\xymatrix{
	0\ar@{->}[r]&\displaystyle \left( \frac{\ker \partial}{\ell^\nu \ker
		\partial}\right) \ar@{->}[r] 		
	\ar@{->>}[d]&\ker \partial_{\ell^\nu}
	\ar@{->}[r]\ar@{->>}[d]^{}&\operatorname{CH}^n(X_{\bar K^a})[\ell^\nu]\ar@{->}[r] \ar@{=}[d]&0\\
		0\ar@{->}[r]&\displaystyle \ker \alpha_\nu \ar@{->}[r] 
	\ar@{->>}[dd]&H^{n-1}_{\operatorname{zar}}(X_{\bar K^a},\mathcal K_n/\ell^\nu)
	\ar@{->}[r]^{\alpha_\nu}\ar@{->}[d]^{\beta_\nu}_\cong &\operatorname{CH}^n(X_{\bar K^a})[\ell^\nu]\ar@{->}[r] \ar[dd]^{-\lambda^n[\ell^\nu]}&0\\
	&&H^{n-1}_{\operatorname{zar}}(X_{\bar K^a},\mathbf H^n(\mmu_{\ell^\nu}^{\otimes
		n}))\ar@{->}[d]^{\gamma_\nu}&&\\
	0\ar[r]&\delta_{\ell^\nu} \ar[r]&H^{2n-1}_{\operatorname{\acute{e}t}}(X_{\bar K^a},\mmu_{\ell^\nu}^{\otimes n}) \ar[r]
	&H^{2n-1}_{\operatorname{\acute{e}t}}(X_{\bar K^a},\mmu_{\ell^\nu}^{\otimes
		n})/\delta_{\ell^\nu} \ar[r] &0\\
}
\end{equation}
Here  $\mathcal K_n$ is the sheaf on $X_{\bar K^a}$ associated to the pre-sheaf that assigns to each Zariski open subset $U\subseteq X_{\bar K^a}$ the algebraic $\operatorname{K}$-group $\operatorname{K}_n(\mathcal O_{X_{\bar K^a}}(U))$ (see e.g., \cite[\S 4.1]{murre83}).    The maps $\alpha_\nu$ and $\beta_\nu$ are defined in \cite[Rem.~6.3]{murre83}, and the map $\gamma_\nu$ is Murre's notation for the map defined in \eqref{E:BM-E1}.  The fact that $\beta_\nu$ is an isomorphism is explained in \cite[Rem.~6.3]{murre83} using \cite{MS}, and this isomorphism then defines the top vertical map in the center of  \eqref{E:BM-E4-Murre}.  The fact that the top right square of the diagram is commutative comes from the construction of $\alpha_\nu$ in \cite[Rem.~6.3]{murre83}\,; in fact, the commutativity of \eqref{E:BM-E4-Murre} is tacitly asserted in \cite{murre83}, so that the construction of the Bloch map given there agrees with that given in \cite{bloch79}.  The rest of  \eqref{E:BM-E4-Murre}  can be established via a diagram chase.

The key point in the case $n=2$ is that    Murre shows in  \cite[Prop.~6.1, Cor.~5.4(c)]{murre83}  that  $\gamma_\nu$ is an inclusion.  Applying the snake lemma to the bottom half of  \eqref{E:BM-E4-Murre} shows that the second $\ell^\nu$-Bloch maps are injective.  Note that taking the direct limit of the injective $\ell^\nu$-Bloch maps gives Murre's proof that $\lambda^2$ is injective, while taking the inverse limit gives another proof  that $T_\ell\lambda^2$ is injective.  
\end{proof}

\subsection{Restriction of the Bloch map to algebraically trivial cycle classes}\label{S:propertiesA}

Let $X$ be a smooth projective variety over a field $K$. From \eqref{E:P:les}, we have a diagram with exact row 
(see e.g., \cite[(3.33)]{grossuwaAJ} for the $p$-adic case)
\begin{equation}
\label{E:BlResA}
	\xymatrix{
		&& \operatorname{CH}^n(X_{\bar K})[l^\infty]\ar@{->}[d]^{\lambda^n} \ar@{-->}[rd] \\
		0 \ar[r] & H^{2n-1}(X_{\bar K},\integ_l(n))\otimes_{\mathbb Z_l} \rat_l/\integ_l \ar[r] & H^{2n-1}(X_{\bar K},\rat_l/\integ_l(n))   \ar[r] & H^{2n}(X_{\bar K},\integ_l(n))
	}
\end{equation}
	where the dashed arrow is, up to sign, the cycle class map (\cite[Cor.~4]{CTSS83}, \cite[Prop.~III.1.16 and Prop.~III.1.21]{grossuwaAJ}).  Since algebraically trivial cycles are homologically trivial, it follows that the image of $\A^n(X_{\bar K})[l^\infty]$ under $\lambda^n$ is contained in $H^{2n-1}(X_{\bar K},\mathbb Z_l(n))\otimes_{\mathbb Z_l}\mathbb Q_l/\mathbb Z_{l}\subseteq H^{2n-1}(X_{\bar K},\mathbb Q_l/\mathbb Z_l(n))$.  In other words, when we restrict the Bloch map to algebraically trivial cycle classes we obtain a map
\begin{equation}\label{E:BlResA1}
\lambda^n :\operatorname{A}^n(X_{\bar K})[l^\infty] \longrightarrow H^{2n-1}(X_{\bar K},\mathbb Z_l(n))\otimes_{\mathbb Z_l}\mathbb Q_l/\mathbb Z_l\subseteq H^{2n-1}(X_{\bar K},\mathbb Q_l/\mathbb Z_l(n)).
\end{equation}

\begin{pro}[Codimension-$1$] \label{P:KummerA}
	Let $X$ be a smooth projective
	variety over $K$\,; if $l= \characteristic(K)$, assume $K$ is perfect.
	The Bloch map \eqref{E:BlResA1}
	$$\lambda^1 :
	\operatorname{A}^1(X_{\bar K})[l^\infty] \longrightarrow H^1(X_{\bar K},\integ_l(1))\otimes_{\mathbb Z_l}\mathbb Q_l/\mathbb Z_l$$
	is an isomorphism, and taking Tate modules yields an isomorphism
 $$T_l \lambda^1 :
	T_l \operatorname{A}^1(X_{\bar K}) \longrightarrow H^1(X_{\bar K},\integ_l(1)).$$

\end{pro}
\begin{proof}
This follows from 
 Proposition \ref{P:Kummer}.  Indeed, we start with the fact that 
 $\operatorname{CH}^1(X_{\bar K})/\operatorname{A}^1(X_{\bar K})\simeq \operatorname{Pic}_{X/K}(\bar K)/\operatorname{Pic}^0_{X/K}(\bar K)=\operatorname{NS}(X_{\bar K})$  is a finitely generated $\mathbb Z$-module.  From this we can conclude that $T_l \operatorname{A}^1(X_{\bar K})=T_l\operatorname{CH}^1(X_{\bar K})$.  This gives the result for $T_l \lambda^1$.  
Then from the  identification $\operatorname{A}^1(X_{\bar K})=\operatorname{Pic}^0_{X/K}(\bar K)$, the torsion and Tate modules are free of the same rank, and so we have $\operatorname{A}^1(X_{\bar K})[l^\infty]= T_l \operatorname{A}^1(X_{\bar K}) \otimes_{\mathbb Z_l}\mathbb Q_l/\mathbb Z_l$. This gives the result for the Bloch map.  

 \end{proof}

\begin{pro}[Bloch maps and Albanese morphisms, Roitman's theorem]
	\label{P:RojtmanA}
	Let $X$ be a smooth projective
	variety of dimension $d$ over $K$. Then the following diagram
\begin{equation}\label{E:RojtmanA}
\xymatrix{ \operatorname{A}^d(X_{\bar K})[l^\infty]\ar[r]^<>(0.5){\lambda^d} \ar[rrd]_{\operatorname{alb}}
		& H^{2d-1}(X_{\bar K},\mathbb Z_l(d))\otimes _{\mathbb Z_l}\mathbb Q_l/\mathbb Z_l  \ar@{^(->}[r]&H^{2d-1}(X_{\bar K},\mathbb Q_l/\mathbb Z_l(d)) \ar@{=}[d]\\
		&& \operatorname{Alb}_X[l^\infty]
	}
	\end{equation}
is commutative, where $\operatorname{alb}$ is obtained by restricting to $\operatorname{A}^d(X_{\bar K})$  the map $\operatorname{CH}^d(X_{\bar K}) \to  \operatorname{Alb}_X(\bar K)$
	mapping a zero cycle on $X_{\bar K}$ to the sum of the corresponding points on
	the Albanese.   Moreover, $\lambda^d$, as well as the inclusion $H^{2d-1}(X_{\bar K},\mathbb Z_l(d))\otimes _{\mathbb Z_l}\mathbb Q_l/\mathbb Z_l  \hookrightarrow H^{2d-1}(X_{\bar K},\mathbb Q_l/\mathbb Z_l(d)) $,  are isomorphisms.  
	
Taking Tate modules, we obtain a commutative diagram
	$$\xymatrix{T_l \operatorname{A}^d(X_{\bar K}) \ar[r]^{T_l\lambda^d\quad\
		} \ar[rd]_{\operatorname{alb}}
		& H^{2d-1}(X_{\bar K},\integ_l(d))_\tau \ar@{=}[d] \\
		& T_l \operatorname{Alb}_X
	}
	$$	
	Moreover, $T_l\lambda^d$ is an isomorphism.
\end{pro}
\begin{proof}
	This just follows from Proposition \ref{P:Rojtman} and the fact that $\operatorname{A}^d(X_{\bar K})[l^\infty]= \operatorname{CH}^d(X_{\bar K})[l^\infty]$.  Indeed, the commutativity  of \eqref{E:RojtmanA} follows from this and the discussion above.  Then Proposition~\ref{P:Rojtman} and the fact that $\operatorname{A}^d(X_{\bar K})[l^\infty]= \operatorname{CH}^d(X_{\bar K})[l^\infty]$ implies that the composition of the top row is an isomorphism.  This forces $\lambda^d$ to be an isomorphism.  The result for Tate modules follows immediately.
	\end{proof}

Finally, we have the following $l$-adic analogue of a result due to
Merkurjev--Suslin \cite{MS}.

\begin{pro}[Injectivity of the second Bloch map]\label{P:M-9.2bis}
	Let $X$ be a smooth projective variety over $K$.
    	The second Bloch map 
	\begin{align*}
\lambda^2: \operatorname{A}^2(X_{\bar K})[l^\infty]\longrightarrow H^3(X_{\bar
		K},\mathbb Z_l(2))\otimes_{\mathbb Z_l}\mathbb Q_l/\mathbb Z_l,
	\end{align*}		the second $l$-adic Bloch map 
	\begin{align*}
	T_l \lambda^2: \ &T_l \operatorname{A}^2(X_{\bar K})\longrightarrow H^3(X_{\bar
		K},\mathbb Z_l(2))_{\tau},
	\end{align*}
	and for all primes $\ell\ne \operatorname{char}(K)$,   the second  $\ell^\nu$-Bloch maps
	$\lambda^2[\ell^\nu]:\operatorname{A}^2(X_{\bar K})[\ell^\nu]\to H^3(X_{\bar K},\mmu_{\ell^\nu}^{\otimes 2})/\delta_{\ell^\nu}$, are injective. 
\end{pro}

\begin{proof}
This just follows from Proposition~\ref{P:M-9.2} and from the inclusion of $\operatorname{A}^2(X_{\bar K})\subseteq \operatorname{CH}^2(X_{\bar K})$.  
\end{proof}

\bibliographystyle{amsalpha}
\bibliography{DCG}

\def\cprime{$'$}
\providecommand{\bysame}{\leavevmode\hbox to3em{\hrulefill}\thinspace}
\providecommand{\MR}{\relax\ifhmode\unskip\space\fi MR }
\providecommand{\MRhref}[2]{%
  \href{http://www.ams.org/mathscinet-getitem?mr=#1}{#2}
}
\providecommand{\href}[2]{#2}
\begin{thebibliography}{ACMV20}

\bibitem[ACMVa]{ACMVdiag}
Jeffrey~D. Achter, Sebastian Casalaina-Martin, and Charles Vial,
  \emph{Decomposition of the diagonal, algebraic representatives, and universal
  codimension-2 cycles in positive characteristic}, preprint, arXiv:2007.07470.

\bibitem[ACMVb]{ACMVfunctor}
\bysame, \emph{A functorial approach to regular homomorphisms}, preprint,
  arXiv:1911.09911.

\bibitem[ACMV17]{ACMVdcg}
\bysame, \emph{On descending cohomology geometrically}, Compositio Math.
  \textbf{153} (2017), no.~7, 1446--1478. \MR{3705264}

\bibitem[ACMV20]{ACMVdmij}
\bysame, \emph{Distinguished models of intermediate {J}acobians}, Journal of
  the Institute of Mathematics of Jussieu \textbf{19} (2020), 891--918.

\bibitem[Ant16]{Antieau}
Benjamin Antieau, \emph{On the integral {T}ate conjecture for finite fields and
  representation theory}, Algebr. Geom. \textbf{3} (2016), no.~2, 138--149.
  \MR{3477951}

\bibitem[Blo79]{bloch79}
S.~Bloch, \emph{Torsion algebraic cycles and a theorem of {R}oitman},
  Compositio Math. \textbf{39} (1979), no.~1, 107--127. \MR{539002 (80k:14012)}

\bibitem[BO74]{BlochOgus}
Spencer Bloch and Arthur Ogus, \emph{Gersten's conjecture and the homology of
  schemes}, Ann. Sci. \'{E}cole Norm. Sup. (4) \textbf{7} (1974), 181--201
  (1975). \MR{0412191}

\bibitem[BS83]{BS}
S.~Bloch and V.~Srinivas, \emph{Remarks on correspondences and algebraic
  cycles}, Amer. J. Math. \textbf{105} (1983), no.~5, 1235--1253. \MR{714776}

\bibitem[BS15]{BhattScholze}
Bhargav Bhatt and Peter Scholze, \emph{The pro-\'{e}tale topology for schemes},
  Ast\'{e}risque (2015), no.~369, 99--201. \MR{3379634}

\bibitem[BW20]{BenWittClGr}
O.~Benoist and O.~Wittenberg, \emph{The {C}lemens--{G}riffiths method over
  non-closed fields}, Algebraic Geometry \textbf{7} (2020), no.~6, 696--721.

\bibitem[CP09]{CPRes2}
Vincent Cossart and Olivier Piltant, \emph{Resolution of singularities of
  threefolds in positive characteristic. {II}}, J. Algebra \textbf{321} (2009),
  no.~7, 1836--1976. \MR{2494751}

\bibitem[CT93]{CT}
Jean-Louis Colliot-Th\'{e}l\`ene, \emph{Cycles alg\'{e}briques de torsion et
  {$K$}-th\'{e}orie alg\'{e}brique}, Arithmetic algebraic geometry ({T}rento,
  1991), Lecture Notes in Math., vol. 1553, Springer, Berlin, 1993, pp.~1--49.
  \MR{1338859}

\bibitem[CTR85]{CTR}
Jean-Louis Colliot-Th\'{e}l\`ene and Wayne Raskind, \emph{{${\mathscr
  K}_2$}-cohomology and the second {C}how group}, Math. Ann. \textbf{270}
  (1985), no.~2, 165--199. \MR{771978}

\bibitem[CTSS83]{CTSS83}
Jean-Louis Colliot-Th\'{e}l\`ene, Jean-Jacques Sansuc, and Christophe
  Soul\'{e}, \emph{Torsion dans le groupe de {C}how de codimension deux}, Duke
  Math. J. \textbf{50} (1983), no.~3, 763--801. \MR{714830}

\bibitem[FM94]{FM-filtration}
Eric~M. Friedlander and Barry Mazur, \emph{Filtrations on the homology of
  algebraic varieties}, Mem. Amer. Math. Soc. \textbf{110} (1994), no.~529,
  x+110, With an appendix by Daniel Quillen. \MR{1211371}

\bibitem[Fuc70]{fuchsIAG}
L\'{a}szl\'{o} Fuchs, \emph{Infinite abelian groups. {V}ol. {I}}, Pure and
  Applied Mathematics, Vol. 36, Academic Press, New York-London, 1970.
  \MR{0255673}

\bibitem[Ful98]{fulton}
William Fulton, \emph{Intersection theory}, second ed., Ergebnisse der
  Mathematik und ihrer Grenzgebiete. 3. Folge. A Series of Modern Surveys in
  Mathematics [Results in Mathematics and Related Areas. 3rd Series. A Series
  of Modern Surveys in Mathematics], vol.~2, Springer-Verlag, Berlin, 1998.
  \MR{1644323 (99d:14003)}

\bibitem[Gro66]{egaIV3}
A.~Grothendieck, \emph{\'{E}l\'ements de g\'eom\'etrie alg\'ebrique. {IV}.
  \'{E}tude locale des sch\'emas et des morphismes de sch\'emas. {III}}, Inst.
  Hautes \'Etudes Sci. Publ. Math. (1966), no.~28, 255. \MR{0217086 (36 \#178)}

\bibitem[Gro85]{gros85}
Michel Gros, \emph{Classes de {C}hern et classes de cycles en cohomologie de
  {H}odge-{W}itt logarithmique}, M\'{e}m. Soc. Math. France (N.S.) (1985),
  no.~21, 87. \MR{844488}

\bibitem[GS88]{grossuwaAJ}
Michel Gros and Noriyuki Suwa, \emph{Application d'{A}bel-{J}acobi {$p$}-adique
  et cycles alg\'{e}briques}, Duke Math. J. \textbf{57} (1988), no.~2,
  579--613. \MR{962521}

\bibitem[Ill79]{illusiedRW}
Luc Illusie, \emph{Complexe de de\thinspace {R}ham-{W}itt et cohomologie
  cristalline}, Ann. Sci. \'{E}cole Norm. Sup. (4) \textbf{12} (1979), no.~4,
  501--661. \MR{565469}

\bibitem[Jan90]{jannsenthesis}
Uwe Jannsen, \emph{Mixed motives and algebraic {$K$}-theory}, Lecture Notes in
  Mathematics, vol. 1400, Springer-Verlag, Berlin, 1990, With appendices by S.
  Bloch and C. Schoen. \MR{1043451}

\bibitem[Jan94]{jannsenseattle}
\bysame, \emph{Motivic sheaves and filtrations on {C}how groups}, Motives
  ({S}eattle, {WA}, 1991), Proc. Sympos. Pure Math., vol.~55, Amer. Math. Soc.,
  Providence, RI, 1994, pp.~245--302. \MR{1265533 (95c:14006)}

\bibitem[Kam15]{Kameko}
Masaki Kameko, \emph{On the integral {T}ate conjecture over finite fields},
  Math. Proc. Cambridge Philos. Soc. \textbf{158} (2015), no.~3, 531--546.
  \MR{3335426}

\bibitem[KM74]{katzmessing}
Nicholas~M. Katz and William Messing, \emph{Some consequences of the {R}iemann
  hypothesis for varieties over finite fields}, Invent. Math. \textbf{23}
  (1974), 73--77. \MR{332791}

\bibitem[Lau76]{laumon76}
G\'{e}rard Laumon, \emph{Homologie \'{e}tale}, S\'{e}minaire de
  g\'{e}om\'{e}trie analytique (\'{E}cole {N}orm. {S}up., {P}aris, 1974-75),
  Ast\'{e}risque, no. 36-37, Soc. Math. France, Paris, 1976, pp.~163--188.
  \MR{0444667}

\bibitem[Maz11]{mazurprob}
B.~Mazur, \emph{For the ``open problem session'' in the conference in honor of
  {J}oe {H}arris}, unpublished note, 2011.

\bibitem[Maz14]{mazurprobICCM}
\bysame, \emph{Open problems: {D}escending cohomology, geometrically}, Notices
  of the International Congress of Chinese Mathematicians \textbf{2} (2014),
  no.~1, 37 -- 40.

\bibitem[Mil80]{milneetale}
J.~S. Milne, \emph{\'{E}tale cohomology}, Princeton Mathematical Series,
  vol.~33, Princeton University Press, Princeton, N.J., 1980. \MR{559531
  (81j:14002)}

\bibitem[Mil82]{milne82}
\bysame, \emph{Zero cycles on algebraic varieties in nonzero characteristic:
  {R}ojtman's theorem}, Compositio Math. \textbf{47} (1982), no.~3, 271--287.
  \MR{681610}

\bibitem[Mil86]{milne86}
\bysame, \emph{Values of zeta functions of varieties over finite fields}, Amer.
  J. Math. \textbf{108} (1986), no.~2, 297--360. \MR{833360}

\bibitem[Mil06]{milneADT}
\bysame, \emph{Arithmetic duality theorems}, second ed., BookSurge, LLC,
  Charleston, SC, 2006. \MR{2261462}

\bibitem[MS82]{MS}
A.~S. Merkurjev and A.~A. Suslin, \emph{{$K$}-cohomology of {S}everi-{B}rauer
  varieties and the norm residue homomorphism}, Izv. Akad. Nauk SSSR Ser. Mat.
  \textbf{46} (1982), no.~5, 1011--1046, 1135--1136. \MR{675529}

\bibitem[Mur85]{murre83}
J.~P. Murre, \emph{Applications of algebraic {$K$}-theory to the theory of
  algebraic cycles}, Algebraic geometry, {S}itges ({B}arcelona), 1983, Lecture
  Notes in Math., vol. 1124, Springer, Berlin, 1985, pp.~216--261. \MR{805336
  (87a:14006)}

\bibitem[Nis55]{nishimura55}
Hajime Nishimura, \emph{Some remarks on rational points}, Mem. Coll. Sci. Univ.
  Kyoto Ser. A. Math. \textbf{29} (1955), 189--192. \MR{95851}

\bibitem[PY15]{PY}
Alena Pirutka and Nobuaki Yagita, \emph{Note on the counterexamples for the
  integral {T}ate conjecture over finite fields}, Doc. Math. (2015), no.~Extra
  vol.: Alexander S. Merkurjev's sixtieth birthday, 501--511. \MR{3404393}

\bibitem[RY00]{reichsteinyoussin}
Zinovy Reichstein and Boris Youssin, \emph{Essential dimensions of algebraic
  groups and a resolution theorem for {$G$}-varieties}, Canad. J. Math.
  \textbf{52} (2000), no.~5, 1018--1056, With an appendix by J\'{a}nos
  Koll\'{a}r and Endre Szab\'{o}. \MR{1782331}

\bibitem[Sam60]{samuelequivalence}
Pierre Samuel, \emph{Relations d'\'equivalence en g\'eom\'etrie alg\'ebrique},
  Proc. {I}nternat. {C}ongress {M}ath. 1958, Cambridge Univ. Press, New York,
  1960, pp.~470--487. \MR{0116010}

\bibitem[{Sch}20]{schreieder}
Stefan {Schreieder}, \emph{{Algebraic cycles and refined unramified
  cohomology}}, preprint, arXiv:2010.05814, October 2020.

\bibitem[SGA72]{SGA4-3}
\emph{Th\'{e}orie des topos et cohomologie \'{e}tale des sch\'{e}mas. {T}ome 1:
  {T}h\'{e}orie des topos}, Lecture Notes in Mathematics, Vol. 269,
  Springer-Verlag, Berlin-New York, 1972, S\'{e}minaire de G\'{e}om\'{e}trie
  Alg\'{e}brique du Bois-Marie 1963--1964 (SGA 4), Dirig\'{e} par M. Artin, A.
  Grothendieck, et J. L. Verdier. Avec la collaboration de N. Bourbaki, P.
  Deligne et B. Saint-Donat. \MR{0354652}

\bibitem[Suw88]{Suwa}
Noriyuki Suwa, \emph{Sur l'image de l'application d'{A}bel-{J}acobi de
  {B}loch}, Bull. Soc. Math. France \textbf{116} (1988), no.~1, 69--101.
  \MR{946279}

\bibitem[Tem17]{temkin17}
Michael Temkin, \emph{Tame distillation and desingularization by
  {$p$}-alterations}, Ann. of Math. (2) \textbf{186} (2017), no.~1, 97--126.
  \MR{3665001}

\bibitem[Via13]{Vial2}
Charles Vial, \emph{Niveau and coniveau filtrations on cohomology groups and
  {C}how groups}, Proc. Lond. Math. Soc. (3) \textbf{106} (2013), no.~2,
  410--444. \MR{3021467}

\bibitem[Voi07]{voisinI}
Claire Voisin, \emph{Hodge theory and complex algebraic geometry. {I}}, english
  ed., Cambridge Studies in Advanced Mathematics, vol.~76, Cambridge University
  Press, Cambridge, 2007, Translated from the French by Leila Schneps.
  \MR{2451566 (2009j:32014)}

\bibitem[Voi14]{voisinDiag}
\bysame, \emph{Chow rings, decomposition of the diagonal, and the topology of
  families}, Annals of Mathematics Studies, vol. 187, Princeton University
  Press, Princeton, NJ, 2014. \MR{3186044}

\bibitem[Voi15]{voisinUniv}
\bysame, \emph{Unirational threefolds with no universal codimension {$2$}
  cycle}, Invent. Math. \textbf{201} (2015), no.~1, 207--237. \MR{3359052}

\end{thebibliography}

\end{document}